\documentclass[a4paper,12pt]{amsart}
\usepackage{yhmath}
\usepackage{graphicx}
\usepackage{mathtools}
\usepackage{mathrsfs}
\usepackage{amsmath}
\usepackage{amsthm}
\usepackage{dsfont}
\usepackage{fancyhdr}
\usepackage{harpoon}
\usepackage{amssymb}
\usepackage{enumerate}
\usepackage[bookmarks=true,colorlinks,linkcolor=black]{hyperref}
\usepackage{geometry}
\usepackage{extarrows}
\usepackage{yfonts}
\usepackage{amsfonts}
\usepackage{tikz}
\usetikzlibrary{arrows}
\usepackage{blkarray}
\usepackage{multirow}
\usepackage{color}
\newtheorem{thm}{Theorem}[section]
\newtheorem{cor}[thm]{Corollary}
\newtheorem{conj}[thm]{Conjecture}

\newtheorem{lem}[thm]{Lemma}
\newtheorem{prop}[thm]{Proposition}
\theoremstyle{definition}
\newtheorem{defn}[thm]{Definition}

\newtheorem{rem}[thm]{Remark}
\theoremstyle{definition}

\numberwithin{equation}{section}

\DeclareMathOperator{\Imm}{Im}

\DeclareMathOperator{\Span}{Span}

\DeclareMathOperator{\Res}{Res}

\DeclareMathOperator{\Spec}{Spec}

\DeclareMathOperator{\Ad}{Ad}

\DeclareMathOperator{\Aut}{Aut}

\DeclareMathOperator{\graph}{graph}

\DeclareMathOperator{\supp}{supp}
\DeclareMathOperator{\Leb}{Leb}
\DeclareMathOperator{\sph}{sph}

\DeclareMathOperator{\BL}{BL}
\DeclareMathOperator{\op}{op}

    {\begin{center}%
     \textsc{Acknowledgements}
\end{center}}%
    {\null}
\begin{document}


\title{N\MakeLowercase{ew time-changes of unipotent flows on quotients of }L\MakeLowercase{orentz groups}}
\author{S\MakeLowercase{iyuan} T\MakeLowercase{ang}}%
\address{D\MakeLowercase{epartment of} M\MakeLowercase{athematics}, IU, B\MakeLowercase{loomington}, IN 47401}
\email{1992.siyuan.tang@gmail.com,  siyutang@indiana.edu}
\maketitle

\begin{abstract}
   We study the cocompact lattices $\Gamma\subset SO(n,1)$ so that the Laplace-Beltrami operator $\Delta$ on $SO(n)\backslash SO(n,1)/\Gamma$ has eigenvalues in $(0,\frac{1}{4})$, and then show that there exist time-changes of unipotent flows on $SO(n,1)/\Gamma$ that are not measurably conjugate to the unperturbed ones.

   A main   ingredient of the proof is a stronger version of the branching of the complementary series. Combining it with   a refinement of the works of Ratner and Flaminio-Forni is adequate for our purpose.
\end{abstract}

\tableofcontents

\section{Introduction}
\subsection{Main results}
Let $G$ be a semisimple Lie group, $\Gamma$ be a lattice of $G$. Let $X=G/\Gamma$ be the homogeneous space equipped with the Haar measure $\mu$.   Then classical  unipotent flows $u^{t}$ on $X$ have been studied by an extensive literature. Besides, one can build new parabolic flows in terms of $u^{t}$ via perturbations. Perhaps the simplest perturbations are \textit{time-changes}\index{time-changes}, i.e. flows that move points along the same orbits, but with different speeds. Time-changes preserve certain ergodic and spectral properties. For instance, \cite{forni2012time} and \cite{de2012spectral} showed that the sufficiently regular time-changes of horocycle flows have the \textit{Lebesgue spectrum}\index{Lebesgue spectrum}. Later \cite{simonelli2018absolutely} extended this result to the case of semisimple unipotent flows. In particular, we know that all time-changes of  unipotent flows satisfying a mild differentiability condition are \textit{strongly mixing}\index{strongly mixing} (the horocycle case was first discovered by \cite{marcus1977ergodic}).

One may then ask whether time-changes produce genuinely new flows, i.e. a time-change of the unipotent flow is actually not \textit{measurably conjugated} to the unperturbed one. However, the question is in general difficult and the answer is known only in a few cases. For horocycle flows, \cite{ratner1986rigidity} and \cite{flaminio2003invariant} showed that sufficiently regular time-changes which are measurably conjugate to the (unperturbed) horocycle flow are \textbf{rare} (in fact, they form a countable codimension subspace). However,   no similar results are known for other unipotent flows. On the other hand, it is worth mentioning that  \cite{ravotti2019parabolic} provides \textbf{new} examples of parabolic perturbations on $SL(3,\mathbf{R})/\Gamma$ for which one can study ergodic theoretical properties. They are not (but related to) time-changes. Whether  the perturbations produce genuinely new flows in that setting remains open.

 In this paper, we manage to generalize \cite{ratner1986rigidity} and \cite{flaminio2003invariant} to $G=SO(n,1)$ setting.  More precisely, let  $\mathfrak{g}=\mathfrak{so}(n,1)$ be the corresponding Lie algebra,   $U\in\mathfrak{g}$ be a nilpotent element. Then it induces a unipotent flow $\phi^{U}_{t}(x)=\exp(tU)x=u^{t}x$ on $X$. Let $\tau$ be a positive integrable function on $X$ with $\int_{X}\tau(x)d\mu(x)=1$. Then define a cocycle $\xi:X\times\mathbf{R}\rightarrow\mathbf{R}$   by
          \[\xi(x,t)\coloneqq\int_{0}^{t}\tau(\phi^{U}_{s}(x))ds=\int_{0}^{t}\tau(u^{s}x)ds.\]
Then the   flow  $\phi^{U,\tau}_{t}:X\rightarrow X$ obtained from the unipotent flow $u^{t}$ by the time-change $\tau$ is given by the relation
          \[\phi^{U,\tau}_{\xi(x,t)}(x)\coloneqq u^{t}x.\]
Besides, we require  that the time-changes have the effective mixing property. More precisely, let $\mathbf{K}(X)$ be the set  of all positive integrable functions $\alpha$ on $X$ such that $\alpha,\alpha^{-1}$ are bounded and satisfies
       \[\left|\int_{X}\alpha(x)\alpha(u^{t}x)d\mu(x)-\left(\int_{X}\alpha(x)\mu(x)\right)^{2}\right|\leq D_{\alpha}|t|^{-\sigma_{\alpha}}\]
       for some $D_{\alpha},\sigma_{\alpha}>0$. In other words, elements $\alpha\in\mathbf{K}(X)$ have polynomial decay of correlations. Note that \cite{kleinbock1999logarithm} has shown that sufficiently regular functions on $X$ are in  $\mathbf{K}(X)$.

 First of all, the construction of time-changes naturally connects it to \textit{cohomological properties}\index{cohomological properties}.
We say that two   functions $g_{1},g_{2}$ on $X$ are  \textit{measurable (respectively $L^{2}$, smooth, etc.) cohomologous over the flow $u^{t}$}\index{cohomologous} if there exists a measurable (respectively $L^{2}$, smooth, etc.) function  $f$ on $X$, called the \textit{transfer function}\index{transfer function}, such that
\begin{equation}\label{time change189}
  \int_{0}^{T}g_{1}(u^{t}x)-g_{2}(u^{t}x)dt=f(u^{T}x)-f(x).
\end{equation}
An elementary argument establishes that   the flows generated   by cohomologous time-changes are always measurably conjugated. On the contrary, we deduce the following generalization of \cite{ratner1986rigidity}:
\begin{thm}\label{time change192}   Let  $\tau\in \mathbf{K}(X)$. Suppose that there is a measurable conjugacy map  $\psi:(X,\mu)\rightarrow (X,\mu_{\tau})$ such that
            \[\psi (\phi^{U}_{t}(x))=\phi_{t}^{U,\tau}(\psi(x))\]
            for $t\in\mathbf{R}$ and $\mu$-a.e. $x\in X$, where $d\mu_{\tau}=\tau d\mu$.   Then $\tau(x)$ and $\tau( cx)$ are  measurably  cohomologous for all $c\in C_{G}(U)$. Besides, if $\tau(x)$ and $\tau( cx)$ are indeed $L^{1}$-cohomologous for all $c\in C_{G}(U)$, then    $1$ and $\tau$ are measurably cohomologous. Here $C_{G}(U)\coloneqq\{g\in G:\Ad g.U=U\}$ denotes the centralizer of $U$ in $G$.
\end{thm}
\begin{rem}
   It is possible to extend the result to two time-changes $\tau_{1}\in \mathbf{K}(G/\Gamma_{1})$, $\tau_{2}\in \mathbf{K}(G/\Gamma_{2})$, similar to \cite{ratner1986rigidity}. More precisely, we can assume that there is a measurable conjugacy map $\psi:(G/\Gamma_{1},\mu_{\tau_{1}})\rightarrow (G/\Gamma_{2},\mu_{\tau_{2}})$ such that
            \[\psi (\phi^{U,\tau_{1}}_{t}(x))=\phi_{t}^{U,\tau_{2}}(\psi(x)).\]
   Then we shall again obtain cohomologous results for $\tau_{1},\tau_{2}$ and, with further regularity assumptions for  $\tau_{1},\tau_{2}$, we have $\Gamma_{1}$ and $\Gamma_{2}$ are conjugate. The proofs will be a bit more complicated and we do not need it here.
\end{rem}

Thus, in order to find a time-changed flow that is not measurably conjugated to the unperturbed one, we should study the cohomological equation (\ref{time change189}). Or equivalently, the differential equation
\begin{equation}\label{time change190}
  g(x)=Uf(x)
\end{equation}
once $f$ is differentiable along $U$-direction.  \cite{flaminio2003invariant} studied the equation (\ref{time change190}) on the irreducible unitary representations of $SO(2,1)$ by classifying the $U$-invariant distributions.  Flaminio-Forni realized that the invariant distributions are the \textbf{only} obstructions to the existence of smooth solutions of equation (\ref{time change190}). Besides,  acting by the geodesic flow on the $U$-invariant distributions, Flaminio-Forni established precise asymptotics for the ergodic averages along the orbits of the horocycle flow on $SO(2,1)/\Gamma$ when $\Gamma$ is a cocompact lattice. However, it seems difficult to generalize these ideas to $SO(n,1)$ for $n\geq 3$. One reason is that the equation (\ref{time change190}) in the $SO(2,1)$-representations is an \textit{ordinary difference equation}\index{ordinary difference equation} (OdE), but in the $SO(n,1)$-representations becomes a \textit{partial difference equation}\index{partial difference equation} (PdE) when $n\geq 3$. However,  if we pay our attention to certain complementary series, then one may possibly restrict the $SO(n,1)$-representations to the subgroup $SO(n-1,1)$, and hence \cite{flaminio2003invariant} may apply.

Let $G=SO(n,1)$, $H=SO(n-1,1)$, $\Gamma\subset G$ be a cocompact lattice such that the Laplace-Beltrami operator $\Delta$ has  eigenvalues in $(0,\frac{1}{4})$. See \cite{randol1974small}, \cite{schoen1980geometric}, \cite{brooks1988injectivity} for    the existence of these lattices (see Section \ref{time change194}). Then
 \cite{zhang2011discrete} (see also  \cite{mukunda1968unitary}, \cite{speh2012restriction}, \cite{speh2016restriction}) has shown that $L^{2}(G/\Gamma)$ contains a $G$-complementary series $\pi_{\nu}$ so that it further contains a $H$-complementary series $\pi^{\flat}_{\nu-\frac{1}{2}}$ as a direct summand. (See Section \ref{time change115} for further discussion of  the required definitions and facts.) Here we prove that   the corresponding Sobolev spaces  also have this property:
\begin{thm}\label{time change214}
    Let  $n\geq 3$, $\rho^{\flat}<\nu<\rho$, $s\geq0$, $G=SO(n,1)$ and $H=SO(n-1,1)$.  Then  $(\pi^{\flat}_{\nu-\frac{1}{2}},W_{H}^{s}(\mathcal{H}_{\nu-\frac{1}{2}}^{\flat}))$ is a direct summand of $(\pi_{\nu},W_{G}^{s}(\mathcal{H}_{\nu}))$ restricted to $H$, where $\rho=\frac{n-1}{2}$ and $\rho^{\flat}=\frac{n-2}{2}$ denote the half sum of the positive roots in $\mathfrak{g}$ and $\mathfrak{h}$, respectively.
\end{thm}

Thus, by repeatedly using Theorem \ref{time change214}, we are able to study a  complementary series of $SO(2,1)$ as a direct summand of $L^{2}(G/\Gamma)$. Then by applying a similar argument of \cite{flaminio2003invariant}, we can study the ergodic average $\frac{1}{T}\int_{0}^{T}g(\phi^{U}_{t}(x))dt$ in terms of $U$-invariant distributions. Then we get
\begin{thm}\label{time change191}  Let the notation and assumptions be as above. Then
   there is a sufficiently regular function $g$ on $X=G/\Gamma$   of integral zero $\mu(g)=0$ that is not measurably cohomologous to $0$, i.e. there are no measurable functions $f$ satisfying
\[\int_{0}^{T}g(\phi^{U}_{t}(x))dt=f(\phi^{U}_{T}(x))-f( x).\]
Moreover, if there are some    $Z\in C_{\mathfrak{g}}(U)$, $\lambda\in\mathbf{R}$    such that  $\phi^{Z}_{\lambda}g$ is not $L^{2}$-cohomologous to $g$, then $\phi^{Z}_{\lambda}g$ is not  measurably cohomologous to $g$.
\end{thm}
Theorem \ref{time change192} and  \ref{time change191}    yield
\begin{cor}\label{time change202106.6}
   Let the notation and assumptions be as above. Then there is a time-change of a   unipotent flow on $X$  that is not measurably conjugate to the unperturbed unipotent flows.
\end{cor}
\begin{proof}
Let $g$ be given by Theorem \ref{time change191}.   Via \textit{Sobolev embedding}\index{Sobolev embedding} (Lemma \ref{time change126}), it is possible to choose $g$ to be continuous. After multiplying a constant if necessary, we can take $\tau=1+g$ to be positive and integrable. Now assume  that there is a measurable conjugacy map  $\psi:(X,\mu)\rightarrow (X,\mu_{\tau})$ such that
            \[\psi (\phi^{U}_{t}(x))=\phi_{t}^{U,\tau}(\psi(x))\]
            for $t\in\mathbf{R}$ and $\mu$-a.e. $x\in X$.  Then by Theorem  \ref{time change192},  $\tau(x)$ and $\tau( cx)$ are   measurably cohomologous for all $c\in  C_{G}(U)$.
               If there exists  $c\in C_{G}(U)$ such that $\tau(x)$ and $\tau( cx)$ are not $L^{1}$-cohomologous, then $g(x)=\tau(x)-1$ and $g( cx)=\tau(cx)-1$ are not $L^{1}$- (and hence are not $L^{2}$-) cohomologous. Thus, by Theorem \ref{time change191}, $\tau(x)$ and $\tau( cx)$ are not measurably cohomologous either, which leads to a contradiction.  Thus, we conclude that  $\tau(x)$ and $\tau( cx)$ are  indeed $L^{1}$-cohomologous for all  $c\in C_{G}(U)$. Via Theorem  \ref{time change192} again, we see that $\tau=1+g$ and $1$ are measurably cohomologous, but it again violates Theorem \ref{time change191}.
\end{proof}
Thus, we conclude that  sufficiently regular time-changes on $X$ which are measurably conjugate to the unperturbed unipotent flow are  rare, in the sense that  the complement of the set of these time-changes has at least  finite codimension.

 Besides, Theorem \ref{time change191} also implies that the \textit{central limit theorem} does not hold for unipotent flows on $X=G/\Gamma$:

  \begin{cor}\label{time change202106.9}  Let the notation and assumptions be as above. Then there is a function $g$ on $X$ such that,
  as $T\rightarrow\infty$, any weak limit of the probability distributions
    \[\frac{\frac{1}{T}\int_{0}^{T}g(\phi^{U}_{t}(x))dt}{\left\|\frac{1}{T}\int_{0}^{T}g(\phi^{U}_{t}(\cdot))dt\right\|_{L^{2}}}\]
has a nonzero compact support.
 \end{cor}
\subsection{Structure of the paper}
In Section \ref{time change230} we recall basic definitions, including  some basic material on the Lie algebra $\mathfrak{so}(n,1)$ (in  Section \ref{time change202106.1}), as well as time-changes (Section \ref{time change202106.2}). In Section \ref{time change202106.3}, we  deduce  Theorem \ref{time change192}. This requires studying the shearing property of $u_{X}^{t}$ for nearby points. More precisely, we provide a quantitative estimate (Proposition \ref{timechange20}) for the difference of nearby points in terms of the length of unipotent orbits. Then, the estimate can deduce extra equivariant properties (Lemma \ref{timechange3} and \ref{time change202106.5}). Then combining Ratner's theorem, we obtain Theorem \ref{timechange14} which states that the measurable conjugacies are almost algebraic. In particular, we obtain the cohomologous relations. In Section \ref{time change115}, we state a number of results of the representation theory, which will be used as tools to study the cohomological equations. In particular, we prove a Sobolev version of the branching of complementary series  (Theorem \ref{time change214} or Theorem \ref{time change116}). Finally, in Section \ref{time change220}, we apply  Flaminio-Forni argument (Theorem \ref{time change125}) to find a required time-change function $\tau$ (Theorem \ref{time change191}) in Corollary \ref{time change202106.7}. Then combining Theorem \ref{time change192}, we conclude that $\tau$ is a nontrivial time-change (Corollary \ref{time change202106.6}).  Besides, we present the central limit theorem of unipotent flows does not hold (Corollary \ref{time change202106.9}) in Corollary \ref{time change202106.8}.

 \noindent
 \textbf{Acknowledgements.}
The paper was written under the guidance of Prof. David Fisher for my PhD thesis, and I am sincerely grateful for his help. I would also like to thank Prof. Livio Flaminio, Prof. Giovanni Forni and Prof. Adam Kanigowski for useful conversations. Besides, I would like to express my deep appreciation to the referee, who examined the paper with great care, pointing out many inaccuracies, and suggesting improvements in various aspects of writing.

\section{Preliminaries}\label{time change230}
\subsection{Definitions}\label{time change202106.1}
Let $G\coloneqq SO(n,1)$; more precisely we define
    \[G\coloneqq\left\{g\in SL_{n+1}(\mathbf{R}):\left[
            \begin{array}{ccc}
              I_{n} &    \\
                 &   -1  \\
            \end{array}
          \right]g^{T}\left[
            \begin{array}{ccc}
              I_{n} &    \\
                 &   -1 \\
            \end{array}
          \right]=g^{-1}\right\}\]
           where $I_{n}$ is the $n\times n$ identity matrix. The corresponding Lie algebra is given by
          \begin{align}
\mathfrak{g}=&\left\{v\in \mathfrak{sl}_{n+1}(\mathbf{R}):\left[
            \begin{array}{ccc}
              I_{n} &    \\
                 &   -1  \\
            \end{array}
          \right]v^{T}\left[
            \begin{array}{ccc}
              I_{n} &    \\
                 &   -1 \\
            \end{array}
          \right]=-v\right\}  \;\nonumber\\
=&  \left\{\left[
            \begin{array}{ccc}
              \mathbf{l} &    \\
                 &   0  \\
            \end{array}
          \right] :\mathbf{l}\in \mathfrak{so}(n) \right\} \oplus\left\{ \left[
            \begin{array}{ccc}
              0  &  \mathbf{p}  \\
               \mathbf{p}^{T}  &  0    \\
            \end{array}
          \right]: \mathbf{p}\in\mathbf{R}^{n}\right\}. \;  \nonumber
\end{align}
          Let $E_{ij}$ be the $(n\times n)$-matrix with $1$ in the $(i,j)$-entry and $0$ otherwise. Let $e_{k}\in\mathbf{R}^{n}$ be the $k$-th standard basis  (vertical) vector. Set
          \[Y_{k}\coloneqq\left[
            \begin{array}{ccc}
              0 & e_{k}    \\
               e_{k}^{T}  &   0 \\
            \end{array}
          \right],\ \ \ \Theta_{ij}\coloneqq\left[
            \begin{array}{ccc}
              E_{ji}-E_{ij} &  0    \\
               0  &   0 \\
            \end{array}
          \right].\]
  Then $Y_{i},\Theta_{ij}$ form a basis of $\mathfrak{g}=\mathfrak{so}(n,1)$.
Let $\mathfrak{g}=\mathfrak{l}\oplus\mathfrak{p}$ be the corresponding Cartan decomposition. Let $\mathfrak{a}=\mathbf{R}Y_{n}\subset\mathfrak{p}$ be a maximal abelian subspace of $\mathfrak{p}$. Then the root space decomposition of $\mathfrak{g}$ is given by
\begin{equation}\label{time change187}
  \mathfrak{g}=\mathfrak{g}_{-1}\oplus\mathfrak{m}\oplus\mathfrak{a}\oplus\mathfrak{g}_{1}.
\end{equation}
Denote by $\mathfrak{n}\coloneqq\mathfrak{g}_{1}$ the sum of the positive root spaces.
Let $\rho$ be the half sum of positive roots. Sometimes, we adopt the convention by identifying $\mathfrak{a}^{\ast}$ with $\mathbf{C}$ via $\lambda\mapsto\lambda(Y_{n})$. Thus, $\rho=\rho(Y_{n})=(n-1)/2$. We write
\[ a^{t}\coloneqq \exp(tY_{n})\]
 for the geodesic flow.

 Let  $\Gamma\subset G$ be a lattice,  $X\coloneqq   G/\Gamma$, $\mu$ be the Haar probability measure on $X$. Fix a nilpotent $U\in  \mathfrak{g}_{-1}^{\flat}$. Then $U$ defines a unipotent flow
 \[\phi^{U}_{t}(x)=\exp(tU)x=u^{t}x\]
 on $G/\Gamma$ and satisfies
\[[Y_{n},U]=-U.\]
Then using the \textit{Killing form}\index{Killing form}, there exists $\tilde{U}\in\mathfrak{g}$ such that $\{U,Y_{n},\tilde{U}\}$ spans a $\mathfrak{sl}_{2}$-triple. Denote
\[\tilde{u}^{t}\coloneqq\exp(t\tilde{U}).\]
For convenience, we   choose
\begin{equation}\label{time change211}
  U\coloneqq\left[
            \begin{array}{ccc}
              0 & e_{n-1}   & e_{n-1}    \\
              -e_{n-1}^{T} & 0   & 0   \\
               e_{n-1}^{T}  & 0 &   0 \\
            \end{array}
          \right],\ \ \ \tilde{U}\coloneqq\left[
            \begin{array}{ccc}
              0 & -e_{n-1}   & e_{n-1}    \\
              e_{n-1}^{T} & 0   & 0   \\
               e_{n-1}^{T}  & 0 &   0 \\
            \end{array}
          \right].
\end{equation}
Then $\langle u^{t},a^{t},\tilde{u}^{t}\rangle$ generates $SO(2,1)\subset SO(n,1)$.
 Further, if we   consider $\mathfrak{g}$ as a $\mathfrak{sl}_{2}(\mathbf{R})$-representation via the adjoint map, then by  the complete reducibility of $\mathfrak{sl}_{2}(\mathbf{R})$, there is a orthogonal decomposition
\[\mathfrak{g}=\mathfrak{sl}_{2}(\mathbf{R})\oplus V^{\perp}\]
where $V^{\perp}$ consists of irreducible representations with highest weights $2$ and $0$. For elements $g\in\exp\mathfrak{g}$, we decompose
\[g=h\exp(v),\ \ \ h\in SO(2,1),\ \ \ v\in V^{\perp}.\]
Moreover, it is convenient to think about $h\in SO(2,1)$ as a $2\times 2$ matrix with determinant $1$. Thus, consider the   isogeny $\iota:SL_{2}(\mathbf{R})\rightarrow SO(2,1)\subset G$ induced by $\mathfrak{sl}_{2}(\mathbf{R})\rightarrow\Span\{U,Y_{n},\tilde{U}\}\subset\mathfrak{g}$. This is a two-to-one immersion.    In the following, for $h\in SO(2,1)$ and $v$ in an irreducible representation, we often write
\[h=\left[
            \begin{array}{ccc}
              a &   b  \\
            c &   d\\
            \end{array}
          \right],\ \ \ v=b_{0}v_{0}+\cdots+b_{\varsigma}v_{\varsigma}\]
where $v_{i}$ are weight vectors in $\mathfrak{g}$ of weight $i$. Notice that $h$ should more appropriately be written as $\iota(h)$.

For  the centralizer $ C_{\mathfrak{g}}(U)$, we  have the corresponding decomposition:
\begin{equation}\label{time change184}
  C_{\mathfrak{g}}(U)=\mathbf{R}U\oplus V^{\perp}_{C}
\end{equation}
where $V^{\perp}_{C}$ consists of  highest weight vectors other than $U$ (see also Lemma \ref{time change179}). More precisely, under the setting (\ref{time change211}), one may calculate
\begin{equation}\label{time change212}
 C_{\mathfrak{g}}(U)=\left\{\left[
            \begin{array}{ccc}
              \mathbf{c} &    \\
                 &   0  \\
            \end{array}
          \right] :\mathbf{c}\in \mathfrak{so}(n-2) \right\} \oplus\left\{ \left[
            \begin{array}{ccc}
              0  &  \mathbf{u} &  \mathbf{u}  \\
             -\mathbf{u}^{T}  &  0 &   0\\
               \mathbf{u}^{T}  &0 &  0    \\
            \end{array}
          \right]:  \mathbf{u}\in\mathbf{R}^{n-1}\right\}.
\end{equation}
Note that the first summand consists of semisimple elements, and the second summand consists of nilpotent elements.
\subsection{Time-changes}\label{time change202106.2}
Let  $\phi^{U,\tau}_{t}$ be a \textit{time change}\index{time change} for the unipotent flow $\phi^{U}_{t}$, $t\in\mathbf{R}$. More precisely, we assume
          \begin{itemize}
            \item  $\tau:X\rightarrow\mathbf{R}^{+}$ is a integrable nonnegative function on $X$ satisfying
          \[\int_{X}\tau(x)d\mu(x)=1,\]
          \item  $\xi:X\times\mathbf{R}\rightarrow\mathbf{R}$ is the cocycle defined by
\begin{equation}\label{time change224}
 \xi(x,t)\coloneqq\int_{0}^{t}\tau(\phi^{U}_{s}(x))ds=\int_{0}^{t}\tau(u^{s}x)ds,
\end{equation}
          \item $\phi^{U,\tau}_{t}:X\rightarrow X$ is given by the relation
          \[\phi^{U,\tau}_{\xi(x,t)}(x)\coloneqq u^{t}x.\]
          \end{itemize}
\begin{rem}\label{timechange36}
  Note that $\phi^{U,1}_{t}=\phi^{U}_{t}$. Besides, one can check that  $\phi^{U,\tau}_{t}$ preserves the probability measure on $X$ defined by  $d\mu_{\tau}\coloneqq\tau d\mu$. On the other hand, if  $\tau$ is smooth, then the time-change $\phi_{t}^{U,\tau}$ is the flow on $X$ generated by the smooth vector field $U_{\tau}\coloneqq U/\tau$ (see \cite{forni2012time}).
\end{rem}

The unipotent flows  $\phi^{U}_{t}$, as well as their time-changes  $\phi^{U,\tau}_{t}$ are \textit{parabolic flows}\index{parabolic flow}, in the sense that  nearby orbits diverge polynomially in time. We shall quantitatively study it via the effective ergodicity of the unipotent flows.

 Besides,
the construction of time-changes naturally connects to cohomological properties.
We say that two   functions $g_{1},g_{2}$ on $X$ are  \textit{measurable (respectively $L^{2}$, smooth, etc.) cohomologous over the flow $u^{t}$}\index{cohomologous} if there exists a measurable (respectively $L^{2}$, smooth, etc.) function  $f$ on $X$, called the \textit{transfer function}\index{transfer function}, such that
\begin{equation}\label{time change215}
 \int_{0}^{T}g_{1}(u^{t}x)-g_{2}(u^{t}x)dt=f(u^{T}x)-f(x)
\end{equation}
for $\mu$-a.e. $x\in X$. We also say that $g$ is \textit{measurably (respectively $L^{2}$, smooth, etc.) trivial}\index{trivial} if $g$ and $0$ are cohomologous. For the related discussion, see \cite{avila2019mixing} and references therein.
    Then conjugacies naturally arise from cohomologous time-changes. More precisely, one may verify that   two time-changes $\tau_{1},\tau_{2}$ are cohomologous via a transfer function $f$ iff the map $\psi_{f}: X\rightarrow X$ defined by
\[\psi_{f}: x\mapsto \phi_{z(x)}^{U}(x)\]
where $z:X\times\mathbf{R}\rightarrow\mathbf{R}$ is defined by the relation
\[f(x)=\xi_{2}(x,z_{f}(x))=\int_{0}^{z_{f}(x)}\tau_{2}(\phi^{U}_{s}(x))ds,\]
is an invertible conjugacy between $\phi_{t}^{U,\tau_{1}}$ and  $\phi_{t}^{U,\tau_{2}}$, i.e.
\[\psi_{f} (\phi^{U,\tau_{1}}_{t}(x))=\phi_{t}^{U,\tau_{2}}(\psi_{f}(x)).\]
On the other hand, if $f$ is differentiable along $U$-direction, then differentiate (\ref{time change215}) along $U$ and   we get the \textit{cohomological  equation}\index{cohomological equation}
\[  g_{1}(x)-g_{2}(x)=Uf(x).\]
We shall discuss it further in Section \ref{time change220}.

       In \cite{ratner1986rigidity}, Ratner considered a particular class $\mathbf{K}(X)$ of time changes. More precisely, $\mathbf{K}(X)$ consists of all positive integrable functions $\alpha$ on $X$ such that $\alpha,\alpha^{-1}$ are bounded and satisfies
       \[\left|\int_{X}\alpha(x)\alpha(u^{t}x)d\mu(x)-\left(\int_{X}\alpha(x)\mu(x)\right)^{2}\right|\leq D_{\alpha}|t|^{-\sigma_{\alpha}}\]
       for some $D_{\alpha},\sigma_{\alpha}>0$. This is the effective mixing property of the unipotent flow $\phi_{t}^{U}$. Note that \cite{kleinbock1999logarithm} (see also \cite{venkatesh2010sparse})  have shown that there is $\kappa>0$ such that
       \[\left|\langle\phi^{U}_{t}(f),g\rangle-\left(\int_{X}f(x)\mu(x)\right)\left(\int_{X}g(x)\mu(x)\right)\right|\ll(1+|t|)^{-\kappa}\|f\|_{W^{s}}\|g\|_{W^{s}}\]
       for $f,g\in C^{\infty}(X)$, where $s\geq \dim(K)$ and $W^{s}$ denotes the Sobolev space on $X=G/\Gamma$ that will be defined later (Section \ref{time change117}).

\section{Measurable conjugacies and transfer functions}\label{time change202106.3}
 In this section, we shall use the shearing properties of unipotent flows and show that any measurable conjugacy between unipotent flows and their time-changes is almost algebraic. More precisely, we deduce
         \begin{thm}\label{timechange14} Let the notation and assumptions be as above.
            Let $\phi_{t}^{U,\tau}$ be a time change for the unipotent flow $u^{t}$ with $\tau\in \mathbf{K}(X)$. Suppose that there is a measurable conjugacy map  $\psi:(X,\mu)\rightarrow (X,\mu_{\tau})$ such that
            \[\psi (\phi^{U}_{t}(x))=\phi_{t}^{U,\tau}(\psi(x))\]
            for $t\in\mathbf{R}$ and $\mu$-a.e. $x\in X$. Then there exists a measurable map $\varpi:X\times C_{G}(U)\rightarrow  C_{G}(U)$ such that
            \begin{equation}\label{dynamical systems2000}
              \psi(cx)=\varpi(x,c)\psi(x)
            \end{equation}
             for $c\in  C_{G}(U)$, $\mu$-almost all $x\in X$. Besides,  $\varpi(x,c)=u^{\alpha(x,c)} \beta(c)$ where $\alpha(x,c)\in\mathbf{R}$ and $\beta(c)\in\exp V^{\perp}_{C}$. Moreover, if  $\alpha(\cdot,c)\in L^{1}(X)$ for all $c\in C_{G}(U)$, then there are points $x_{0},y_{0}\in X$, an  automorphism $\Phi$ of $G$ that fixes  $SO(2,1)$ (i.e. $\Phi(g)=g$ for $g\in SO(2,1)$) and a  map $c:X\rightarrow C_{G}(U)$ such that
             \begin{equation}\label{time change171}
               \psi(gx_{0})= c(gx_{0}) \Phi(g)y_{0}
             \end{equation}
            for any $g\in G$. Similarly, $c(x)=u^{a(x)}b$ where $a(x)\in\mathbf{R}$ and $b\in\exp V_{C}^{\perp}$.
         \end{thm}
       The proof of Theorem \ref{timechange14} use the strategy similar to Ratner's theorem. We consider the unipotent orbits of nearby points and look at their images under the given measurable conjugacy $\psi$. By Lusin's theorem, the images inherit a similar behavior as in the range. Using this phenomenon, we can study the difference of the nearby points under $\psi$ and so obtain the extra equivariant properties of $\psi$. The difficulty is to pull the information of unipotent orbits from the homogeneous spaces back to Lie groups. It  requires us to observe the long unipotent orbits with ``gaps" and make use of their polynomial growth nature (see Proposition \ref{timechange20}).

   Note that combining    (\ref{dynamical systems2000})  (respectively (\ref{time change171})) with Corollary  \ref{dynamical systems2001}   (respectively Corollary \ref{time change172}), we obtain a criterion for the solutions of the cohomological equation (Theorem \ref{time change192}):
         \begin{cor}
          Let the notation and assumptions be as above.     Then $\tau(x)$ and $\tau(cx)$ are  measurably  cohomologous for all $c\in C_{G}(U)$. Besides, if $\tau(x)$ and $\tau(cx)$ are indeed $L^{1}$-cohomologous for all $c\in C_{G}(U)$, then    $1$ and $\tau$ are measurably cohomologous.
         \end{cor}

\subsection{Shearing properties}\label{time change202106.4}
We shall study the \textit{shearing property}\index{shearing property} of the unipotent flow $\varphi^{U}_{t}$. Roughly speaking, it states that if two points start out so close together that  we cannot tell them apart, then the first difference we see, which is often called the \textit{fastest relative motion}\index{fastest relative motion}, will fall in the centralizer $C_{G}(U)$. It is only much later that we will detect any other difference between their paths.

In the following, we shall prove Proposition \ref{timechange20}, which provide a quantitative estimate of the difference of two nearby points via the shearing property stated above. The philosophy of Proposition \ref{timechange20} is:
\begin{enumerate}[\ \ \ ]
  \item  For certain $\lambda>0$, $t:[0,\infty)\rightarrow[0,\infty)$, if $99\%$ of $s\in[0,\lambda]$ $d_{X}(u^{t(s)}y,u^{s}x)<\epsilon$, then the only possible situation is that there is a big
\textbf{interval} $I\subset [0,\lambda]$ (say of $90\%$ length) so that $d_{X}(u^{t(s)}y,u^{s}x)<\epsilon$ for all $s\in I$.
\end{enumerate}  Roughly speaking, it connects to the fact that polynomials do not have extreme oscillations.
The $SL(2,\mathbf{R})$ version of this  property has already been established by Ratner \cite{ratner1986rigidity}. The method is also inspired by the proof of \textit{Ratner's theorem}. See \cite{einsiedler2006ratner}, \cite{einsiedler2009effective} and references therein.

\begin{prop}[Shearing]\label{timechange20}  Let the notation and assumptions be as above. Given $\eta\in(0,1)$ and $m>1$, there are
\begin{itemize}
  \item  $\rho=\rho(\eta)>0$,
  \item $\theta=\theta(\rho)>0$,
  \end{itemize}
such that for any sufficiently small $\sigma\in(0,\sigma_{\rho})$,  there are
\begin{itemize}
  \item  a compact $K=K(\rho,\sigma)\subset X$ with $\mu(K)>1-\sigma$,
  \item  $\epsilon=\epsilon(K,m)\in(0,1)$ close to $0$
\end{itemize}   satisfying the following property: Let $x\in K$, $y\in B_{X}(x,\epsilon)$, and a subset $A\subset\mathbf{R}^{+}$ satisfy the following conditions
\begin{enumerate}[\ \ \ (i)]
  \item if $s\in A$, then
  \[u^{s}x\in  K \ \ \text{ and }\ \  d_{X}(u^{t(s)}y,u^{s}x)<\epsilon\]
  for some increasing function  $t:[0,\infty)\rightarrow[0,\infty)$,
  \item   we have the H\"{o}lder inequality:
  \begin{equation}\label{timechange25}
   |(t(s^{\prime})-t(s))-(s^{\prime}-s)|\leq |s^{\prime}-s|^{1-\eta}
  \end{equation}
      for all $s,s^{\prime}\in A$ with $s^{\prime}>s$, $\max\{(s^{\prime}-s),(t(s^{\prime})-t(s))\}\geq m$.
\end{enumerate}
Then for any $\lambda\in A$ satisfying $\Leb(A\cap[0,\lambda])>(1-\theta)  \lambda$, there is $s_{\lambda}\in A\cap[0,\lambda]$ such that
\begin{equation}\label{timechange33}
 u^{t(s_{\lambda})} y=h_{\lambda}\exp(v_{\lambda})u^{s_{\lambda}} x
\end{equation}
    where $h_{\lambda}\in SO(2,1)$ and $v_{\lambda}\in V^{\perp}$ satisfy
  \begin{align}
h_{\lambda}=&\left[
            \begin{array}{ccc}
               1+O(\lambda^{-2\rho}) &   O(\lambda^{-1-2\rho})  \\
              O(\epsilon) &   1+O(\lambda^{-2\rho})\\
            \end{array}
          \right] \;\nonumber\\
v_{\lambda}=& O(\lambda^{- \frac{1+2\rho}{2}\varsigma})v_{0}+O(\lambda^{- \frac{1+2\rho}{2}(\varsigma-1)})v_{1}+\cdots+O(\epsilon)v_{\varsigma}.\;  \nonumber
\end{align}
\end{prop}

\subsection{Quantitative estimates}
In order to prove Proposition \ref{timechange20}, we shall develop a collection $\alpha$ of finitely many subintervals of $[0,\lambda]$ through the assumptions. Then we shall show that there is a  big interval from the     collection $\alpha$. In the following, we first verify a combinatorial result that helps us to find the big interval in $\alpha$.

 Let $I$ be an interval in $\mathbf{R}$ and let $J_{i},J_{j}$ be disjoint subintervals of $I,J_{i}=[x_{i},y_{i}]$, $y_{i}<x_{j}$ if $i<j$. Denote
\[d(J_{i},J_{j})\coloneqq \Leb[y_{i},x_{j}]=x_{j}-y_{i}.\]
For a collection $\beta$ of finitely many intervals, we define
\[|\beta|\coloneqq \Leb\left(\bigcup_{J\in \beta}J\right).\]
Besides, for a collection $\beta$ of finitely many intervals, an interval $I$, let
\[\beta\cap I\coloneqq\{I\cap J:J\in\beta\}.\]
\begin{prop}[Existence of large intervals, Solovay, \cite{ratner1979cartesian}] \label{dynamical systems1011}
  Given $\rho\in(0,1)$, there is $\theta=\theta(\rho)\in(0,1)$ such that if $I$ is an interval of length $\lambda> 1 $   and $\mathcal{G}\cup\mathcal{B}=\{J_{1},\ldots,J_{n}\}$ is a partition of $I$ into good and bad intervals such that
  \begin{enumerate}[\ \ \ (1)]
    \item   for any two good intervals $J_{i},J_{j}\in\mathcal{G}$, we have
    \begin{equation}\label{time change175}
      d(J_{i},J_{j})\geq[\min\{\Leb(J_{i}),\Leb(J_{j})\}]^{1+\rho},
    \end{equation}
    \item $ \Leb(J)\leq  \frac{3}{4}\lambda$ for any good interval $J\in\mathcal{G}$,
    \item $ \Leb(J)\geq 1$ for any bad interval $J\in\mathcal{B}$,
  \end{enumerate}
  then the measure of bad intervals $\Leb(\bigcup_{J\in \mathcal{B}}J)\geq \theta\lambda$.
\end{prop}
\begin{rem} The idea of Proposition \ref{dynamical systems1011} is to consider   the arrangement of intervals in  $\alpha$. It turns out that under the assumptions, the worst arrangement would be like the complement of a Cantor set. A careful calculation of the quantities under this situation leads to the result.
\end{rem}
\begin{proof}
 Assume that  $\left(\frac{4}{3}\right)^{k-1} \leq \lambda\leq \left(\frac{4}{3}\right)^{k}$ for some $k\geq1$.  Let $\mathcal{G}_{n}\coloneqq\{J\in\mathcal{G}:\left(\frac{3}{4}\right)^{n+1}\lambda\leq |J|\leq \left(\frac{3}{4}\right)^{n}\lambda\}$, $\mathcal{G}_{\leq n}\coloneqq\bigcup_{i=1}^{n}\mathcal{G}_{i}$, and $\mathcal{B}_{\leq n}$ be the collection of remaining intervals forming $I\setminus\bigcup_{J\in\mathcal{G}_{\leq n}}J$.  Then given $n\in\mathbf{N}$, $J\in \mathcal{B}_{\leq n}$, by (\ref{time change175}), we have
    \begin{align}
\frac{|\mathcal{B}_{\leq n+1}\cap J|}{\Leb(J)}=&\frac{|\mathcal{B}_{\leq n+1}\cap J|}{|\mathcal{G}_{n+1}\cap J|+|\mathcal{B}_{\leq n+1}\cap J|}= \left(1+\frac{|\mathcal{G}_{n+1}\cap J|}{|\mathcal{B}_{\leq n+1}\cap J|}\right)^{-1}  \;\nonumber\\
\geq& \left(1+\frac{l\left(\frac{3}{4}\right)^{n+1}\lambda}{(l-1) \left(\frac{3}{4}\right)^{(n+2)(1+\rho)}\lambda^{1+\rho}}\right)^{-1} = \left(1+ C\left(\frac{3}{4}\right)^{(k-n)\rho}\right)^{-1}  \;  \nonumber
\end{align}
where $l\geq 2$ is the number of intervals in $\mathcal{G}_{n+1}\cap J$, and $C>0$ is some constant depending on $\rho$. One can also show that when $k=0,1$, we have a similar relation. By summing over $J\in \mathcal{B}_{\leq n}$, we obtain
\[\frac{|\mathcal{B}_{\leq n+1}|}{|\mathcal{B}_{\leq n}|}\geq\left(1+ C\left(\frac{3}{4}\right)^{(k-n)\rho}\right)^{-1}. \]
Note that by (2), $|\mathcal{B}_{\leq 0}|=\lambda$, and    by (3), $\mathcal{B}_{\leq n}=\mathcal{B}_{\leq n+1}$ for all $n\geq k$. We calculate
\[|\mathcal{B}|=|\bigcap_{k\geq 0}\mathcal{B}_{\leq k}|=\lim_{k\rightarrow\infty}|\mathcal{B}_{\leq k}|=\prod_{n=0}^{\infty}\frac{|\mathcal{B}_{\leq n+1}|}{|\mathcal{B}_{\leq n}|}\cdot \lambda\geq \prod_{n=0}^{k} \left(1+ C\left(\frac{3}{4}\right)^{(k-n)\rho}\right)^{-1}\cdot \lambda.\]
Take
\[\theta\coloneqq\prod_{m=0}^{\infty} \left(1+ C\left(\frac{3}{4}\right)^{m\rho}\right)^{-1}\leq\prod_{n=0}^{k} \left(1+ C\left(\frac{3}{4}\right)^{(k-n)\rho}\right)^{-1}\]
 and the proposition follows.
\end{proof}

In light of (\ref{time change175}), we say that two intervals $I,J\subset\mathbf{R}$ have an \textit{effective gap}\index{effective gap} if
\[   d(I,J)\geq[\min\{\Leb(I),\Leb(J)\}]^{1+\rho}\]
for some $\rho>0$.
Later, we shall obtain some quantitative results relative to the effective gap.

  \begin{lem}\label{timechange106}
  For sufficiently small vector $v\in\mathfrak{g}$, we have
  \[\log \exp(gvg^{-1})=gvg^{-1}\]
  for all $g\in G$, where $\log$ denotes the principal logarithm.
\end{lem}
\begin{proof}
  According to \cite{higham2008functions}, for any square complex matrix $v$, $\log \exp(v)=v$ iff $|\Imm \lambda_{i}|<\pi$ for every eigenvalue $\lambda_{i}$ of $v$. Then the consequence follows from the fact that $\Ad g$ does not change the eigenvalues of $v$.
\end{proof}
\begin{lem}\label{timechange102} Let $\mathfrak{g}$ be a   Lie algebra, and $\mathfrak{g}=V_{1}\oplus V_{2}$ be a decomposition of vector spaces. Then
  the multiplication map $\zeta:V_{1}\oplus V_{2}\rightarrow G$ defined by
  \[(\alpha,\beta)\mapsto\exp(\alpha)\exp(\beta)\]
  induces a diffeomorphism on small neighborhoods  $U_{1}\subset V_{1}$ and  $U_{2}\subset V_{2}$ of $0$.
\end{lem}
\begin{proof}
   Note that $d\zeta_{(0,0)}:(\alpha,\beta)\mapsto\alpha+\beta$. Then the consequence follows from the inverse function theorem.
\end{proof}

In the following, $A\ll B$   means there is a constant $C>0$  such that $A\leq CB$. Besides, we write $A\ll_{\kappa} B$ if the constant $C(\kappa)$ depends on some coefficient $\kappa$.

\begin{lem}\label{dynamical systems4} Fix numbers $\epsilon>0$, $\eta\in(0,1]$, a real polynomial $p(x)=v_{0}+v_{1}x+\cdots+v_{k}x^{k}\in \mathbf{R}[x]$. Assume further that there exist intervals $[0,\overline{l}_{1}]\cup[l_{2},\overline{l}_{2}]\cup\cdots\cup[l_{m},\overline{l}_{m}]$ such that
\begin{equation}\label{dynamical systems3}
  |p(t)|\ll\max\{\epsilon,t^{1-\eta}\}\ \  \text{ iff } \ \ t\in[0,\overline{l}_{1}]\cup[l_{2},\overline{l}_{2}]\cup\cdots\cup[l_{m},\overline{l}_{m}]
\end{equation}
Then $\overline{l}_{1}$  has the  lower bound $l$  depending on $\max_{i}|v_{i}|$, $\epsilon$, $\eta$ and the implicit constant such that $l\nearrow\infty$ as $\max_{i}|v_{i}|\searrow0$ for fixed $\epsilon,\eta$.
Besides, $m\leq k$ and we have
   \begin{enumerate}[\ \ \ (1)]
     \item  $|v_{i}|\ll_{k,\eta} \overline{l}_{1}^{1-i-\eta} $ for all $1\leq i\leq k$;
     \item Fix $\rho\in(0,1)$. For $1\leq j\leq k-1$, sufficiently large $\overline{l}_{j}$, assume that the intervals $[0,\overline{l}_{j}]$ and $[l_{j+1},\overline{l}_{j+1}]$ do not have an effective gap:
     \begin{equation}\label{dynamical systems6}
       l_{j+1}-\overline{l}_{j}\leq \min\{\overline{l}_{j},\overline{l}_{j+1}-l_{j+1}\}^{1+\rho}.
     \end{equation}
     Then there exists $\xi(\rho,k)\in(0,1)$ with $\xi(\rho,k)\rightarrow1$ as $\rho\rightarrow0$ such that
     \[|v_{i}|\ll_{k,\eta} \overline{l}_{j}^{\xi(\rho,k)(1-i-\eta)}\]
for all $1\leq i\leq k$.
   \end{enumerate}
\end{lem}
\begin{proof}
The number $m$ of intervals in (\ref{dynamical systems3}) can be bounded by $k$ via an elementary study of polynomials.

(1)   Let $F (x)\coloneqq  v_{1}(\overline{l}_{1}x)^{\eta} +\cdots+v_{k}(\overline{l}_{1}x)^{k-1+\eta}$ for $x\in[0,1]$.  Then we have
\[\left(
            \begin{array}{c}
             v_{1}\overline{l}_{1}^{\eta} \\
              v_{2}\overline{l}_{1}^{1+\eta} \\
             \vdots \\
             v_{k}\overline{l}_{1}^{k-1+\eta} \\
            \end{array}
          \right)=\left[
            \begin{array}{cccc}
              (1/k)^{\eta}&  (1/k)^{1+\eta} &  \cdots &  (1/k)^{k-1+\eta} \\
              (2/k)^{\eta}&  (2/k)^{1+\eta} &  \cdots &  (2/k)^{k-1+\eta} \\
                \vdots&  \vdots &  \ddots &  \vdots \\
                 1&  1 &  \cdots &  1 \\
            \end{array}
          \right]^{-1}\left(
            \begin{array}{c}
              F(1/k) \\
              F(2/k) \\
             \vdots \\
             F(1) \\
            \end{array}
          \right).\]
  By (\ref{dynamical systems3}), we know that   $|F (1/k)|,|F (2/k)|,\cdots,|F(1)|\ll 1 $. Thus, we obtain  $|v_{i}|\ll_{k,\eta} \overline{l}_{1}^{1-i-\eta} $ for all $1\leq i\leq k$.

   (2) This follows easily by induction. Assume that the statement holds for $j-1$. For $j$,  the only difficult situation is when $\overline{l}_{j}\leq l_{j+1}-\overline{l}_{j}$ and $\overline{l}_{j+1}-l_{j+1}\leq l_{j+1}-\overline{l}_{j}$.   If this is the case, then
\[ \overline{l}_{j+1}= (\overline{l}_{j+1}-l_{j+1})+(l_{j+1}-\overline{l}_{j})+\overline{l}_{j}\leq 3\overline{l}_{j}^{1+\rho}.\]
Thus, by induction hypothesis, we get
\[|v_{i}|\ll \overline{l}_{j}^{\xi(\rho,j)(1-i-\eta)}\ll \overline{l}_{j+1}^{\frac{\xi(\rho,j)}{1+\rho}(1-i-\eta)}\]
for all $1\leq i\leq k$.
    \end{proof}

\begin{lem}\label{time change179}
   By the weight decomposition, an irreducible $\mathfrak{sl}_{2}(\mathbf{R})$-representation $V_{\varsigma}$ is the direct sum of weight spaces, each of which is $1$ dimensional. More precisely, there exists a basis $v_{0},\ldots,v_{\varsigma}\in V_{\varsigma}$ such that
   \[U.v_{i}=(i+1)v_{i+1},\ \ \ Y_{n}.v_{i}= \frac{\varsigma-2i}{2}v_{i}.\]
\end{lem}
Thus, if $V_{\varsigma}$ is an irreducible representation of  $\mathfrak{sl}_{2}(\mathbf{R})$ with the highest weight $\varsigma\leq 2$, then for any $v=b_{0}v_{0}+\cdots+b_{\varsigma}v_{\varsigma}\in V_{\varsigma}$, we have
 \begin{align}
\exp(tU).v=&\sum_{n=0}^{\varsigma}\sum_{i=0}^{n}b_{i}\binom{n}{i}t^{n-i}v_{n},  \; \label{dynamical systems5}\\
 \exp(\omega Y_{n}).v=&  \sum_{n=0}^{\varsigma} b_{n}e^{ (\varsigma-2n)\omega /2} v_{n}. \;  \nonumber
\end{align}

In the following, we consider the decomposition $\mathfrak{g}=\mathfrak{sl}_{2}(\mathbf{R})\oplus V^{\perp}$, where  $\mathfrak{sl}_{2}(\mathbf{R})=\Span\{U,Y_{n},\tilde{U}\}\subset\mathfrak{g}$ is the $\mathfrak{sl}_{2}$-triple. We shall study $\mathfrak{sl}_{2}(\mathbf{R})$ and  $V^{\perp}$ separately. We can first assume that $V^{\perp}=V_{\varsigma}$ is irreducible.

 By Lemma \ref{timechange102}, for sufficiently small $\epsilon>0$, $g\in B_{G}(e,\epsilon)$, we have
\begin{equation}\label{dynamical systems1}
  g=h\exp v
\end{equation}
for some $h\in B_{SO(2,1)}(e,\epsilon)$  and $v\in B_{V^{\perp}}(0,\epsilon)$.
Now we discuss a   necessary condition for $h\in B_{SO(2,1)}(e,\epsilon)$  in a small neighborhood of the identity. Recall that we  consider $h$ as a $(2\times 2)$-matrix $h=\left[
            \begin{array}{ccc}
              a&  b \\
             c&  d\\
            \end{array}
          \right]\in SO(2,1)$. Then one may obtain  that a necessary condition for $h\in B_{SO(2,1)}(e,\epsilon)$ is that $|b|,|c|<\epsilon$, $1-\epsilon<|a|,|d|<1+\epsilon$.

Next, let $t(s)\in\mathbf{R}^{+}$ be a function of $s\in\mathbf{R}^{+}$. Then by (\ref{dynamical systems5}) , we have
\begin{align}
u^{t}gu^{-s}=&u^{t}h\exp v u^{-s} \;\nonumber\\
 =&(u^{t}hu^{-s})(u^{s}\exp(v)u^{-s}) \;\nonumber\\
 =&(u^{t}hu^{-s})\exp(\Ad u^{s}.v)  \;\nonumber\\
  =&(u^{t}hu^{-s})\exp\left(\sum_{n=0}^{\varsigma}\sum_{i=0}^{n}b_{i}\binom{n}{i}s^{n-i}v_{n}\right).   \;  \nonumber
\end{align}
Moreover, by Lemma \ref{timechange106},  \ref{timechange102}, $u^{t}gu^{-s}\ll \epsilon$ iff
\begin{equation}\label{time change173}
  u^{t}hu^{-s}\ll\epsilon,\ \ \ \Ad u^{s}.v=\sum_{n=0}^{\varsigma}\sum_{i=0}^{n}b_{i}\binom{n}{i}s^{n-i}v_{n}\ll\epsilon.
\end{equation}
Thus, we split the elements close to the identity into two parts, namely the $SO(2,1)$-part and the $V^{\perp}$-part.

As shown in (\ref{time change173}), we shall consider the elements of the form $u^{t}hu^{-s}\in B_{SO(2,1)}(e,\epsilon)$.
   A direct calculation shows
   \[u^{t}hu^{-s}=\left[
            \begin{array}{ccc}
              1&   \\
             t&  1\\
            \end{array}
          \right]\left[
            \begin{array}{ccc}
              a&  b \\
             c&  d\\
            \end{array}
          \right]\left[
            \begin{array}{ccc}
              1&    \\
            -s&  1\\
            \end{array}
          \right]=\left[
            \begin{array}{ccc}
              a-bs&  b \\
             c+(a-d)s-bs^{2}+(t-s)(a-bs)&  d+bt\\
            \end{array}
          \right].\]
   If we further require $|s-t|\ll_{\eta}\max\{\epsilon,s^{1-\eta}\}$ (cf. (\ref{timechange25})), then we see that
   \begin{align}
 &|-bs^{2}+(a-d)s+c+(-bs+a)(t-s)|<\epsilon \;\nonumber\\
\Rightarrow\ \ \ & |-bs^{2}+(a-d)s|-|c|-|(-bs+a)(t-s)|<\epsilon\;\nonumber\\
\Rightarrow\ \ \ & |-bs^{2}+(a-d)s|<2\epsilon+2|t-s|\;\nonumber\\
\Rightarrow\ \ \ & |-bs^{2}+(a-d)s|\ll_{\eta}\max\{\epsilon,s^{1-\eta}\}. \;\label{time change222}
\end{align}
By Lemma \ref{dynamical systems4}, we immediately obtain
\begin{lem}[Estimates for $SO(2,1)$-coefficients]\label{dynamical systems8}  Given $\eta\in(0,1)$, a sufficiently small $\epsilon>0$, a matrix $h=\left[
            \begin{array}{ccc}
              a&  b \\
             c&  d\\
            \end{array}
          \right]\in B_{SO(2,1)}(e,\epsilon)$, then the solutions $s\in[0,\infty)$ of the following inequality
          \begin{equation}\label{timechange17}
            |-bs^{2}+(a-d)s|\ll_{\eta} \max\{\epsilon,s^{1-\eta}\}
          \end{equation}
   consist of at most two intervals, say $[0,\overline{l}_{1}(h)]\cup[l_{2}(h),\overline{l}_{2}(h)]$  where  $\overline{l}_{1}$  has the  lower bound $l(\epsilon,\eta)$ such that $l(\epsilon,\eta)\nearrow\infty$ as $\epsilon\searrow0$ for fixed $\eta$.  Moreover, we have
   \begin{enumerate}[\ \ \ (1)]
     \item  $|b|\ll_{\eta} \overline{l}_{1}^{-1-\eta}$ and $|a-d|\ll_{\eta} \overline{l}_{1}^{-\eta}$;
     \item If we further assume     that the intervals $[0,\overline{l}_{1}]$ and $[l_{2},\overline{l}_{2}]$ do not have an effective gap (\ref{dynamical systems6}), i.e.   $l_{2}-\overline{l}_{1}\leq \min\{\overline{l}_{1},\overline{l}_{2}-l_{2}\}^{1+\rho}$, then
         \[|b|\ll_{\eta} \overline{l}_{2}^{\xi(\rho)(-1-\eta)},\ \ \ |a-d|\ll_{\eta} \overline{l}_{2}^{\xi(\rho)(-\eta)}.\]
   \end{enumerate}
\end{lem}
Next, we study the situation when $Ad u^{s}.v\ll \epsilon$. Again by Lemma \ref{dynamical systems4}, we have
\begin{lem}[Estimates for $V^{\perp}$-coefficients]\label{dynamical systems7}
  Fix $v=b_{0}v_{0}+\cdots+b_{\varsigma}v_{\varsigma}\in B_{V_{\varsigma}}(0,\epsilon)$. Assume that
   \[Ad u^{s}.v\ll \epsilon_{0}\ \ \ \text{ iff }\ \ \ s\in[0,\overline{l}_{1}(v)] \cup\cdots\cup[l_{m}(v),\overline{l}_{m}(v)] \]
    where  $\overline{l}_{1}$  has the  lower bound $l(\epsilon,\eta)$ such that $l(\epsilon,\eta)\nearrow\infty$ as $\epsilon\searrow0$ for fixed $\eta$.
   Then $m=m(v)$ is bounded by a constant depending on $\varsigma$. Moreover, for   $1\leq j\leq \varsigma-1$, the
   intervals $[0,\overline{l}_{j}]$ and $[l_{j+1},\overline{l}_{j+1}]$ do not have an effective gap (\ref{dynamical systems6}), i.e.   $l_{j+1}-\overline{l}_{j}\leq \min\{\overline{l}_{j},\overline{l}_{j+1}-l_{j+1}\}^{1+\rho}$, then  we have
   \[|b_{i}|\ll_{\varsigma,\eta} \overline{l}_{j}^{\xi(\rho,\varsigma)(-\varsigma+i)}.\]

\end{lem}

For $g=h\exp(v)\in G$, we conclude from   (\ref{time change222}), Lemma \ref{dynamical systems8} and   \ref{dynamical systems7} that
\begin{align}
 u^{t}hu^{-s}\ll\epsilon &\ \ \ \text{ implies } \ \ \ s\in [0,\overline{l}_{1}(h)]\cup[l_{2}(h),\overline{l}_{2}(h)]\;\nonumber\\
 Ad u^{s}.v\ll \epsilon  &\ \ \ \text{ iff } \ \ \ s\in[0,\overline{l}_{1}(v)]\cup\cdots\cup[l_{m(v)}(v),\overline{l}_{m(v)}(v)].\; \nonumber
\end{align}
Write $l_{1}(h)=l_{1}(v)=0$ and we shall consider the family of intervals
\begin{equation}\label{time change174}
  \{[l_{k}(g),\overline{l}_{k}(g)]\}_{k}\coloneqq\{[l_{i}(h),\overline{l}_{i}(h)]\cap[l_{j}(v),\overline{l}_{j}(v)]\}_{i,j}
\end{equation}
where the intervals $[l_{i}(h),\overline{l}_{i}(h)]$, $[l_{j}(v),\overline{l}_{j}(v)]$ are given by Lemma \ref{dynamical systems8},  \ref{dynamical systems7} respectively, and $\overline{l}_{k}(g)<l_{k+1}(g)$ for all $k$. Thus, in particular, $l_{1}(g)=0$ and $[0,\overline{l}_{1}(g)]=[0,\overline{l}_{1}(h)]\cap[0,\overline{l}_{1}(v)]$.

Now assume that there exists $k$ such that $[0,\overline{l}_{k}(g)]$  and $[l_{k+1}(g),\overline{l}_{k+1}(g)]$ do not have an effective gap (\ref{dynamical systems6}), i.e.
\[  l_{k+1}(g)-\overline{l}_{k}(g)\leq \min\{\overline{l}_{k}(g),\overline{l}_{k+1}(g)-l_{k+1}(g)\}^{1+\rho}.\]
Then clearly, the corresponding ``$SO(2,1)$-part" and ``$V^{\perp}$-part" should not have effective gaps either. More precisely, for the $SO(2,1)$-part, we define
\[i_{\geq k}\coloneqq\min\{i\in\{1,2\}:\overline{l}_{k}(g)\leq \overline{l}_{i}(h)\},\ \ \ i_{\leq k+1}\coloneqq\max\{i\in\{1,2\}:l_{k+1}(g)\geq l_{i}(h)\}.\]
Thus, we know
\[[0,\overline{l}_{k}(g)]\subset[0,\overline{l}_{i_{\geq k}}(h)],\ \ \ [l_{k+1}(g),\overline{l}_{k+1}(g)]\subset[l_{i_{\leq k+1}}(h),\overline{l}_{i_{\leq k+1}}(h)] \]
and hence $[0,\overline{l}_{i_{\geq k}}(h)]$  and $[l_{i_{\leq k+1}}(h),\overline{l}_{i_{\leq k+1}}(h)] $ do not have an effective gap (\ref{dynamical systems6}). Similarly, for the $V^{\perp}$-part, we define
\[j_{\geq k}\coloneqq\min\{j:\overline{l}_{k}(g)\leq \overline{l}_{j}(v)\},\ \ \ j_{\leq k+1}\coloneqq\max\{j:l_{k+1}(g)\geq l_{j}(v)\}.\]
Then we know
\[[0,\overline{l}_{k}(g)]\subset[0,\overline{l}_{j_{\geq k}}(v)],\ \ \ [l_{k+1}(g),\overline{l}_{k+1}(g)]\subset[l_{j_{\leq k+1}}(v),\overline{l}_{j_{\leq k+1}}(v)] \]
and hence $[0,\overline{l}_{j_{\geq k}}(v)]$  and $[l_{j_{\leq k+1}}(v),\overline{l}_{j_{\leq k+1}}(v)] $ do not have an effective gap (\ref{dynamical systems6}). Further, one observes
\begin{align}
[0,\overline{l}_{k}(g)]=&[0,\overline{l}_{i_{\geq k}}(h)]\cap [0,\overline{l}_{j_{\geq k}}(v)]\;\nonumber\\
 [l_{k+1}(g),\overline{l}_{k+1}(g)]= & [l_{i_{\leq k+1}}(h),\overline{l}_{i_{\leq k+1}}(h)]\cap [l_{j_{\leq k+1}}(v),\overline{l}_{j_{\leq k+1}}(v)].\; \nonumber
\end{align}
Now recall by the definition (\ref{time change174}) that the number of intervals in $\{[l_{k}(g),\overline{l}_{k}(g)]\}_{k}$ is bounded by a constant $C(\varsigma)>0$ because the numbers of intervals $\{[l_{i}(h),\overline{l}_{i}(h)]\}_{i}$, $\{[l_{j}(v),\overline{l}_{j}(v)]\}_{j}$   are. Since $\varsigma\leq 2$ when $\mathfrak{g}=\mathfrak{so}(n,1)$, we see that $C(\varsigma)$ is uniformly bounded for all $\varsigma$. Thus, we conclude that the number of intervals in $\{[l_{k}(g),\overline{l}_{k}(g)]\}_{k}$ is uniformly bounded for all $g\in G$.
Then, combining with Lemma \ref{dynamical systems7} and  \ref{dynamical systems8}, we obtain
\begin{lem}[Estimates for $G$-coefficients]\label{time change176}
     Let $g=h\exp v\in B_{G}(e,\epsilon)$ be as above, where
   \[h=\left[
            \begin{array}{ccc}
              a&  b \\
             c&  d\\
            \end{array}
          \right]\in SO(2,1),\ \ \ v=b_{0}v_{0}+\cdots+b_{\varsigma}v_{\varsigma}\in V_{\varsigma}.\]
          Next, let $t(s)\in\mathbf{R}^{+}$ be a function of $s\in\mathbf{R}^{+}$ which satisfies $|s-t(s)|\ll_{\eta}\max\{\epsilon,s^{1-\eta}\}$. Then there exist intervals $\{[l_{k}(g),\overline{l}_{k}(g)]\}_{k}$ such that
          \begin{equation}\label{time change221}
            u^{t}gu^{-s}\ll\epsilon\ \ \ \text{ implies }\ \ \ s\in \bigcup_{k}[l_{k}(g),\overline{l}_{k}(g)].
          \end{equation}
        where  $\overline{l}_{1}$  has the  lower bound $l(\epsilon,\eta)$ such that $l(\epsilon,\eta)\nearrow\infty$ as $\epsilon\searrow0$ for fixed $\eta$.    Besides, $k\leq C$ for some constant $C=C(\mathfrak{g})>0$, and
          \begin{enumerate}[\ \ \ (1)]
            \item   $|b|\ll_{\eta} \overline{l}_{1}(g)^{-1-\eta}$, $|a-d|\ll_{\eta} \overline{l}_{1}(g)^{-\eta}$, $|b_{i}|\ll_{\varsigma,\eta} \overline{l}_{1}(g)^{-\varsigma+i}$ for all $0\leq i\leq \varsigma$;
     \item If we further assume     that the intervals $[0,\overline{l}_{k}(g)]$ and $[l_{k+1}(g),\overline{l}_{k+1}(g)]$ do not have an effective gap (\ref{dynamical systems6}). Then there exists $\xi=\xi(\rho)\in(0,1)$ with $\xi\rightarrow1$ as $\rho\rightarrow0$ such that
      \[|b|\ll_{\eta} \overline{l}_{k}(g)^{-\xi(1+\eta)},\ \ \ |a-d|\ll_{\eta} \overline{l}_{k}(g)^{-\xi\eta},\ \ \ |b_{i}|\ll_{\varsigma,\eta} \overline{l}_{k}(g)^{-\xi (\varsigma-i)}\]
for all $1\leq i\leq \varsigma$.
          \end{enumerate}
\end{lem}
\begin{rem}
   Since $\varsigma\leq 2$ for $\mathfrak{g}=\mathfrak{so}(n,1)$, we might obtain Lemma \ref{time change176} via an explicit discussion of intervals (and this is simpler at first glance). However, as in general $V$ is not irreducible (for $n\geq4$), it would be convenient to repeat the above argument to conclude  Lemma \ref{time change176}.
\end{rem}
\subsection{Proof of Proposition \ref{timechange20}}
Now we start to prove Proposition \ref{timechange20}.
We shall adopt a similar strategy as in \cite{ratner1986rigidity}. More precisely, we shall construct a collection $\beta_{\rho}$ of disjoint subintervals of $\mathbf{R}^{+}$ so that its union covers $A\cap[0,\lambda]$ and  every pair has an effective gap (\ref{time change175}). Then, we apply Proposition \ref{dynamical systems1011} to obtain a large interval  in $\beta_{\rho}$. We first specify the quantities claimed in Proposition \ref{timechange20}.
\begin{itemize}
  \item  (Choice of $\rho$) Choose a small $\rho>0$ that satisfies
  \begin{equation}\label{time change182}
    \frac{1+2\rho}{\xi(2\rho)}<1+\eta,\ \ \   1+2\delta<1+2\rho<2\xi(2\rho)
  \end{equation}
    where $\xi(2\rho)$ was defined in Lemma \ref{time change176}, and $\delta\coloneqq 3\rho/4$.
  \item  (Choice of $\theta$) Let $\theta=\theta(\rho)$ be as in Proposition \ref{dynamical systems1011}.
  \item  (Choice of $\sigma_{\rho}$) Then  $\sigma_{\rho}>0$  can be chosen as
  \begin{equation}\label{time change183}
    \sigma_{\rho}<\frac{\rho}{4+6\rho}.
  \end{equation}
  \item (Choice of $\Delta$, $K_{1}$; injectivity radius) Let $\pi:G\rightarrow X$ be the natural quotient map. Since $\Gamma$ is discrete, there is a compact subset $K_{1}\subset X$, $\mu(K_{1})>1- \frac{1}{2}\sigma$ and $ \Delta\in(0,1)$ such that for any $g_{x}\in \pi^{-1}(K_{1})$, $g_{y}\in G$ satisfying
      \begin{equation}\label{timechange30}
       d(g_{x},g_{y})<2\Delta,\ \ \   d(u^{s}g_{x},u^{t} g_{y}\gamma)<2\Delta \text{ with } e\neq \gamma\in\Gamma ,
      \end{equation}
 we must have $\max\{|t|,|s|\}\geq m$. Here  $d$ denotes the metric on $X$, and  $m$ is given by the assumption of Proposition \ref{timechange20}. In particular, it implies that for any $g_{x}\in \pi^{-1}(K_{1})$, $g_{y}\in G$ satisfying
 \begin{equation}\label{timechange31}
    d(g_{x},g_{y})<2\Delta,\ \ \   d(g_{x}, g_{y}\gamma)<2\Delta
 \end{equation}
 for some $\gamma\in\Gamma$, then $\gamma=e$.
 \item (Choice of $K_{2}$, $K$, $T_{0}$, $T$; ergodicity of $a^{T}$)
Since the diagonal action $a^{T}$ is ergodic on $(X,\mu)$,  there is a compact subset $K_{2}\subset  G/\Gamma$, $\mu(K_{2})>1-\frac{1}{2}\sigma$ and $T_{0}=T_{0}(K_{2})>0$ such that  the relative length measure $K_{2}$ on $[x,a^{T}x]$ (and $[a^{-T}x,x]$) is greater than $1-\sigma$ for any $x\in K_{2}$,   $|T|\geq T_{0}$.   Assume that
\begin{equation}\label{dynamical systems1026}
  K\coloneqq K_{1}\cap K_{2}
\end{equation}
Note that $\mu(K)>1-\sigma$. The choice will be used in (\ref{dynamical systems3003}).
\item (Choice of $\epsilon$)
Let $0<\epsilon<\Delta$ be so small that for $g\in B_{G}(e,\epsilon)$
\begin{equation}\label{dynamical systems1020}
  \overline{l}_{1}(g)\geq l(\epsilon,\eta)>\max\{e^{(1+2\delta)^{-1}T_{0}},m\}
\end{equation}
where $\overline{l}_{1},l$ are defined in Lemma \ref{time change176}, and $\delta\coloneqq 3\rho/4$.
\end{itemize}

 Thus, $0<\rho,\xi,\theta,\epsilon<1$ and $K\subset X$ have been chosen.
 Next let us describe some notation that will be used later.
Let $x\in X$, $y\in B_{X}(x,\epsilon)$. We say that \textit{$(g_{x},g_{y})\in G\times G$ covers $(x,y)$} if $ d_{G}(g_{x},g_{y})<\epsilon$ and $\pi(g_{x})=x$, $\pi(g_{y})=y$.
\begin{defn}[$\epsilon$-block]  Suppose that  $x\in X$, $y\in B(x,\epsilon)$,   $(g_{x},g_{y})$ covers $(x,y)$, and $r\in (0,\infty]$ satisfies
\[d_{G}(u^{r}g_{x},u^{t(r)}g_{y})<\epsilon.\]
 Then we define    the \textit{$\epsilon$-block of $g_{x},g_{y}$ of length $r$}\index{$\epsilon$-block of $g_{x},g_{y}$ of length $r$}    by
 \[\BL(g_{x},g_{y})\coloneqq  \{(u^{s}g_{x},u^{t(s)}g_{y})\in G\times G:0\leq s\leq r\}.\]
 Similarly, we define the \textit{$\epsilon$-block of $x,y$ of length $r$} by
  \[\BL(x,y)\coloneqq \pi \BL(g_{x},g_{y})=\{(u^{s}x,u^{t(s)}y)\in X\times X:0\leq s\leq r\}.\]
   We also write
 \[\BL(x,y)=\{(x,y),(u^{r}x ,u^{t(r)}y)\}=\{(x,y),(\overline{x},\overline{y})\}\]
 emphasizing that $(x,y)$ is the first and $(\overline{x},\overline{y})$ is the last pair of the block $\BL(x,y)$.
\end{defn}
\noindent
\textbf{Construction of  $\beta_{0}$.}
 Let $x\in X$, $y\in B(x,\epsilon)$ and assume that $A\subset\mathbf{R}^{+}$ satisfies (i), (ii) as in Proposition \ref{timechange20} (and assume without loss of generality that $0\in A$). For $\lambda\in A$ denote
$A_{\lambda}\coloneqq A\cap[0,\lambda]$
and assume that
\begin{equation}\label{dynamical systems1036}
  \Leb(A_{\lambda})>(1-\theta)\lambda.
\end{equation}
Now we construct a collection $\beta_{0}$ of  $\epsilon$-blocks. Let $x_{1}\coloneqq x$, $y_{1}\coloneqq y$. Suppose that $(g_{x_{1}},g_{y_{1}})\in G\times G$ covers $(x_{1},y_{1})$ and
\[\overline{s}_{1}\coloneqq\sup\{s\in A_{\lambda}\cap[0,\overline{l}_{1}(g_{y_{1}}g_{x_{1}}^{-1})]: d_{G}( u^{t(s)}g_{y_{1}}  ,u^{s}g_{x_{1}}  )<\epsilon \}.\]
Let $\BL_{1}$ be the $\epsilon$-block of $x_{1},y_{1}$ of length $\overline{s}_{1}$, $\BL_{1}=\{(x_{1},y_{1}),(\overline{x}_{1},\overline{y}_{1})\}$.
To define $\BL_{2}$, we take
\[s_{2}\coloneqq\inf\{s\in A_{\lambda}:s>\overline{s}_{1} \}\]
and apply the above procedure to $x_{2}\coloneqq u^{s_{2}}x_{1},\ \ \ y_{2}\coloneqq  u^{t(s_{2})}y_{1}$ (Note that by (\ref{time change221}), $s_{2}>\overline{s}_{1}$).
 This process defines a collection $\beta_{0}=\{\BL_{1},\ldots, \BL_{n}\}$ of $\epsilon$-blocks on the orbit intervals $[x_{1},u^{\lambda}x_{1}]$, $[y_{1},u^{t(\lambda)}y_{1}]$ (see Figure \ref{time change223}):
\[x_{i}=u^{s_{i}}x_{1},\ \ \ \overline{x}_{i}=u^{\overline{s}_{i}}x_{1},\ \ \ y_{i}=u^{t_{i}} y_{1},\ \ \ \overline{y}_{i}=u^{\overline{t}_{i}}y_{1}.\]
 Note also that by the assumption of $A$, we have $x_{i},\overline{x}_{i}\in K$ for all $i$.

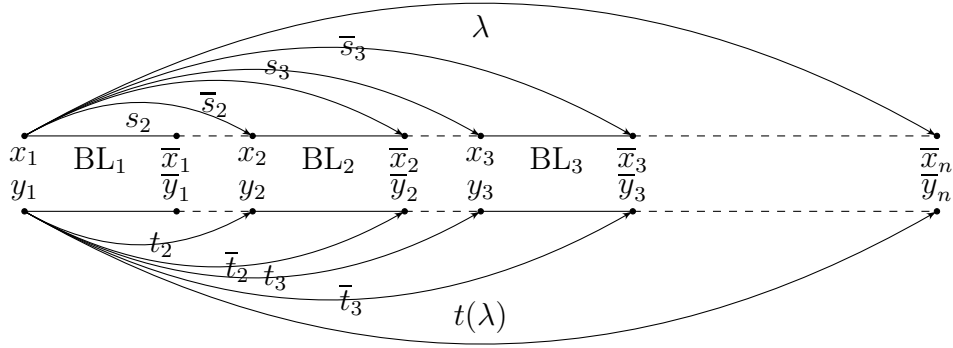
\begin{figure}
\centering
 \tikzstyle{init} = [pin edge={to-,thin,black}]
\begin{tikzpicture}[node distance=4cm,auto,>=latex']
\draw (0,0) coordinate (a0) --node[below] {$\BL_{1}$} (2,0) coordinate (a2) ;
\filldraw[black] (a0) circle(1pt) node[below] {$x_{1}$};
\filldraw[black] (a2) circle(1pt) node[below] {$\overline{x}_{1}$};

 \draw [dashed] (2,0) coordinate (a2) -- (3,0) coordinate (a3);
 \draw (3,0) coordinate (a3) --node[below] {$\BL_{2}$} (5,0) coordinate (a5) ;
\filldraw[black] (a3) circle(1pt) node[below] {$x_{2}$};
\filldraw[black] (a5) circle(1pt) node[below] {$\overline{x}_{2}$};

 \draw [dashed] (5,0) coordinate (a5) -- (6,0) coordinate (a6);
 \draw (6,0) coordinate (a6) --node[below] {$\BL_{3}$} (8,0) coordinate (a8) ;
\filldraw[black] (a6) circle(1pt) node[below] {$x_{3}$};
\filldraw[black] (a8) circle(1pt) node[below] {$\overline{x}_{3}$};

 \draw [dashed] (8,0) coordinate (a8) -- (12,0) coordinate (a12);
 \filldraw[black] (a12) circle(1pt) node[below] {$\overline{x}_{n}$};

 \path[->] (a0) edge [bend left] node [below]{$s_{2}$} (a3);
 \path[->] (a0) edge [bend left] node [below]{$\overline{s}_{2}$} (a5);

 \path[->] (a0) edge [bend left] node [right]{$s_{3}$} (a6);
 \path[->] (a0) edge [bend left] node [right]{$\overline{s}_{3}$} (a8);
  \path[->] (a0) edge [bend left] node [below]{$\lambda$} (a12);

  \draw (0,-1) coordinate (b0) --  (2,-1) coordinate (b2) ;
\filldraw[black] (b0) circle(1pt) node[above] {$y_{1}$};
\filldraw[black] (b2) circle(1pt) node[above] {$\overline{y}_{1}$};

 \draw [dashed] (2,-1) coordinate (b2) -- (3,-1) coordinate (b3);
 \draw (3,-1) coordinate (b3) --  (5,-1) coordinate (b5) ;
\filldraw[black] (b3) circle(1pt) node[above] {$y_{2}$};
\filldraw[black] (b5) circle(1pt) node[above] {$\overline{y}_{2}$};

 \draw [dashed] (5,-1) coordinate (b5) -- (6,-1) coordinate (b6);
 \draw (6,-1) coordinate (b6) --  (8,-1) coordinate (b8) ;
\filldraw[black] (b6) circle(1pt) node[above] {$y_{3}$};
\filldraw[black] (b8) circle(1pt) node[above] {$\overline{y}_{3}$};

 \draw [dashed] (8,-1) coordinate (b8) -- (12,-1) coordinate (b12);
 \filldraw[black] (b12) circle(1pt) node[above] {$\overline{y}_{n}$};

 \path[->] (b0) edge [bend right] node [right]{$t_{2}$} (b3);
 \path[->] (b0) edge [bend right] node [right]{$\overline{t}_{2}$} (b5);

 \path[->] (b0) edge [bend right] node [right]{$t_{3}$} (b6);
 \path[->] (b0) edge [bend right] node [right]{$\overline{t}_{3}$} (b8);
  \path[->] (b0) edge [bend right] node [above]{$t(\lambda)$} (b12);
\end{tikzpicture}
\caption{A collection of $\epsilon$-blocks $\{\BL_{1},\ldots,\BL_{n}\}$.  The solid straight lines are the unipotent orbits in the $\epsilon$-blocks and the dashed lines are the rest of the unipotent orbits. The bent curves indicate the   length defined by the letters.}
\label{time change223}
\end{figure}

  \begin{rem}\label{timechange28} Notice that  any $\BL_{i}=\{(x_{i},y_{i}),(\overline{x}_{i},\overline{y}_{i})\}\in\beta_{0}$ has length $\leq \overline{l}_{1}(g_{y_{i}}g_{x_{i}}^{-1})$. Write $g_{y_{i}}=h\exp(v)g_{x_{i}}$, where $h=\left[
            \begin{array}{ccc}
               a&  b \\
               c &   d\\
            \end{array}
          \right]\in SO(2,1)$. Then  by Lemma  \ref{time change176}, we immediately conclude that
 \begin{align}
|b| \ll_{\eta} &\overline{l}_{1}(h)^{-1-\eta}\leq \overline{l}_{1}(g_{y_{i}}g_{x_{i}}^{-1})^{-1-\eta}\leq |\BL_{i}|^{-1-\eta}   \;\nonumber\\
|a-d| \ll_{\eta}&  \overline{l}_{1}(h)^{-\eta} \leq \overline{l}_{1}(g_{y_{i}}g_{x_{i}}^{-1})^{-\eta}\leq |\BL_{i}|^{-\eta} \;  \nonumber
\end{align}
 where $|\BL_{i}|$ denotes the length of the $\epsilon$-block $\BL_{i}$.
\end{rem}

  For a collection $\beta_{0}$ of $\epsilon$-blocks, a shifting problem may occur.
  \begin{defn}[Shifting]\label{dynamical systems3004} For integers $i<j$, assume that $(g_{x_{i}},g_{y_{i}})\in G\times G$, $g_{y_{i}}\in B_{G}(g_{x_{i}},\epsilon)$ covers $(x_{i},y_{i})$. Then there is a unique $\gamma\in\Gamma$ such that
 \begin{equation}\label{dynamical systems1014}
  d_{G}(g_{x_{j}},g_{y_{j}}\gamma)<\epsilon
 \end{equation}
 where $g_{y_{j}}\coloneqq u^{t_{j}-t_{i}}g_{y_{i}}$, $g_{x_{j}}\coloneqq u^{s_{j}-s_{i}}g_{x_{i}}$. We  write
 \begin{itemize}
   \item (Shifting) $(x_{i},y_{i})\overset{\Gamma}{\sim}(x_{j},y_{j})$ if $\gamma\neq e$ in  (\ref{dynamical systems1014}),
   \item (Non-shifting) $(x_{i},y_{i})\overset{e}{\sim}(x_{j},y_{j})$ if $\gamma=e$ in  (\ref{dynamical systems1014}).
 \end{itemize}
\end{defn}
\noindent
\textbf{Construction of  $\beta_{\rho}$.}  Now we construct a new collection $\beta_{\rho}=\{\overline{\BL}_{1},\ldots,\overline{\BL}_{k}\}$ by the following procedure. The idea is to connect $\epsilon$-blocks in $\beta_{0}=\{\BL_{1},\ldots,\BL_{n}\}$  so that each pair of new blocks must have an effective gap.
  Take $\BL_{1}\in\beta_{0}$,   $g_{y_{1}}=h\exp(v)g_{x_{1}}$ and
       \[h=\left[
            \begin{array}{ccc}
              a&  b \\
             c&  d\\
            \end{array}
          \right]\in SO(2,1),\ \ \ v=b_{0}v_{0}+\cdots+b_{\varsigma}v_{\varsigma}\in V_{\varsigma}.\]
  Then by  Lemma  \ref{time change176}, one can write $u^{t(s)}gu^{-s}\in B_{G}(e,\epsilon)$ for
  \begin{equation}\label{time change177}
   s\in   \bigcup_{k}[l_{k}(g),\overline{l}_{k}(g)]
  \end{equation}
  where $k\leq C$ is uniformly bounded for all $g\in G$.
  Then consider the following two cases:
    \begin{enumerate}[\ \ \ (i)]
    \item  There is no $j\in\{2,\ldots,n\}$ such that $(x_{1},y_{1})\overset{e}{\sim}(x_{j},y_{j})$.
     \item There is $j\in\{2,\ldots,n\}$ such that   $(x_{1},y_{1})\overset{e}{\sim}(x_{j},y_{j})$.
    \end{enumerate}

  In case (i), we set $\overline{\BL}_{1}=\BL_{1}$. Then by Remark \ref{timechange28}, we have
  \begin{equation}\label{time change180}
    |b|\ll\overline{l}_{1}(g_{y_{1}}g_{x_{1}}^{-1})^{-1-\eta},\ \ \ |a-d|\leq  \overline{l}_{1}(g_{y_{1}}g_{x_{1}}^{-1})^{-\eta}
  \end{equation}

   In case (ii),  suppose that   $g_{x_{j}}=u^{s_{j}}g_{x_{1}}$, $g_{y_{j}}=u^{t_{j}}g_{y_{1}}$. Clearly, by the construction, $\overline{s}_{j}>\overline{l}_{1}(g_{y_{1}}g_{x_{1}}^{-1})$. On the other hand, by (\ref{time change177}), we get
    \[\overline{s}_{j}\in   \bigcup_{k}[l_{k}(g_{y_{1}}g_{x_{1}}^{-1}),\overline{l}_{k}(g_{y_{1}}g_{x_{1}}^{-1})]\]
    and $k\leq C$ is uniformly bounded for all $g\in G$.
    Assume that $j_{\max}$ is the maximal $j$ among $\overline{s}_{j}\in[l_{2}(g_{y_{1}}g_{x_{1}}^{-1}),\overline{l}_{2}(g_{y_{1}}g_{x_{1}}^{-1})]$.  Whether $[0,\overline{l}_{1}(g_{y_{1}}g_{x_{1}}^{-1})]$ and $[l_{2}(g_{y_{1}}g_{x_{1}}^{-1}),\overline{l}_{2}(g_{y_{1}}g_{x_{1}}^{-1})]$ have an effective gap leads to a dichotomy of choices:
   \[\overline{\BL}_{1}=\left\{\begin{array}{ll}
 \text{remains unchange}&,\text{ if } l_{2}(g_{y_{1}}g_{x_{1}}^{-1})-\overline{l}_{1}(g_{y_{1}}g_{x_{1}}^{-1})>\overline{l}_{1}(g_{y_{1}}g_{x_{1}}^{-1})^{1+2\rho}\\
 \{(x_{1},y_{1}),(\overline{x}_{j_{\max}},\overline{y}_{j_{\max}})\} &,\text{ otherwise}
\end{array}\right. .\]
If the first case occurs, we will not change $\overline{\BL}_{1}$ anymore.
 If the second case occurs, i.e. we  redefine $\overline{\BL}_{1}=\{(x_{1},y_{1}),(\overline{x}_{j_{\max}},\overline{y}_{j_{\max}})\}$, then we repeat the construction for the new $\overline{\BL}_{1}$ again:
  \begin{enumerate}[\ \  \ ]
    \item  Suppose that there is  $\overline{s}_{j}>\overline{l}_{2}(g_{y_{1}}g_{x_{1}}^{-1})$. Then assume $j_{\max}$ to be the maximal $j$ among $\overline{s}_{j}\in[l_{3}(g_{y_{1}}g_{x_{1}}^{-1}),\overline{l}_{3}(g_{y_{1}}g_{x_{1}}^{-1})]$. Then again, we set
  \[\overline{\BL}_{1}=\left\{\begin{array}{ll}
\text{remains unchange}&,\text{ if } l_{3}(g_{y}g_{x}^{-1})-\overline{l}_{3}(g_{y_{1}}g_{x_{1}}^{-1})>\overline{l}_{2}(g_{y_{1}}g_{x_{1}}^{-1})^{1+2\rho}\\
 \{(x_{1},y_{1}),(\overline{x}_{j_{\max}},\overline{y}_{j_{\max}})\} &,\text{ otherwise}
\end{array}\right. \]
and so on.
  \end{enumerate}
  The process will stop since the number of intervals is uniformly bounded for all $g\in G$. Now $\overline{BL}_{1}\in\beta_{\rho}$ has been constructed. By the choice
of $\overline{\BL}_{1}$ and Lemma \ref{time change176}, we conclude  that
\begin{equation}\label{time change178}
 |b|\ll_{\eta} |\BL_{1}|^{-\xi(1+\eta)},\ \ \ |a-d|\ll_{\eta} |\BL_{1}|^{-\xi\eta},\ \ \ |b_{i}|\ll_{\varsigma,\eta} |\BL_{1}|^{-\xi (\varsigma-i)}
\end{equation}
for all $1\leq i\leq \varsigma$.

Next, we repeat the above argument to construct $\overline{\BL}_{m+1}$. More precisely, suppose that $   \overline{\BL}_{m}=\{(x_{j_{m-1}+1},y_{j_{m-1}+1}),(\overline{x}_{j_{m}},\overline{y}_{j_{m}})\}\in\beta_{\rho}$ has been constructed. To define $\overline{\BL}_{m+1}$, we repeat the above argument  to  $\BL_{j_{m}+1}\in\beta_{0}$. Thus, $\beta_{\rho}$ is completely defined. Further, one may conclude some basic properties of $\beta_{\rho}$:
\begin{lem}\label{time change181}
   For any $\overline{\BL}_{i}=\{(x_{i}^{\prime},y_{i}^{\prime}),(\overline{x}_{i}^{\prime},\overline{y}_{i}^{\prime})\}$ in the collection  $\beta_{\rho}=\{\overline{\BL}_{1},\ldots,\overline{\BL}_{k}\}$ of $\epsilon$-blocks, we have
   \begin{equation}\label{timechange26}
     y_{i}^{\prime}=h_{i}\exp(v_{i})x_{i}^{\prime}
   \end{equation}
    where
    \[h_{i}=\left[
            \begin{array}{ccc}
               1+O(r_{i}^{-2\rho}) &   O(r_{i}^{-1-2\rho})  \\
              O(\epsilon) &   1+O(r_{i}^{-2\rho})\\
            \end{array}
          \right],\ \ \ v_{i}=O(r_{i}^{-\xi\varsigma})v_{0}+\cdots+O(\epsilon)v_{\varsigma}\]
  for some  $r_{i}\geq\max\{e^{(1+2\delta)^{-1}T_{0}},|\overline{\BL}_{i}|\}$ where $T_{0}$ is given by (\ref{dynamical systems1026}).
\end{lem}
\begin{proof}
 (\ref{timechange26}) follows immediately from (\ref{time change180}), (\ref{time change178}), (\ref{dynamical systems1020}), (\ref{time change182}).
\end{proof}
\begin{lem}\label{timechange27}
  For any $\overline{\BL}^{\prime}\neq \overline{\BL}^{\prime\prime}\in\beta_{\rho}$, we have
   \begin{equation}\label{dynamical systems1034}
    d(\overline{\BL}^{\prime},\overline{\BL}^{\prime\prime})>\max\{e^{(1+2\delta)^{-1}T_{0}},[\min\{|\overline{\BL}^{\prime}|,|\overline{\BL}^{\prime\prime}|\}]^{1+\rho}\}
  \end{equation}
  where the distance of blocks is defined by the distance of the intervals provided by the $x$-coordinate,  $\delta\coloneqq 3\rho/4$ and $|\BL|$ denotes the length of the $\epsilon$-block $\BL$.
\end{lem}

\begin{proof}

 Suppose $\overline{\BL}^{\prime},\overline{\BL}^{\prime\prime}\in\beta_{\rho}$ and write
\[\overline{\BL}^{\prime}=\{(x^{\prime},y^{\prime}),(\overline{x}^{\prime},\overline{y}^{\prime})\},\ \ \ \overline{\BL}^{\prime\prime}=\{(x^{\prime\prime},y^{\prime\prime}),(\overline{x}^{\prime\prime},\overline{y}^{\prime\prime})\}, \ \ \ x^{\prime\prime}=u^{s}\overline{x}^{\prime},\ \ \ y^{\prime\prime}=u^{t}\overline{y}^{\prime}.\]
 If    $\overline{\BL}^{\prime}\overset{e}{\sim}\overline{\BL}^{\prime\prime}$, then by the above construction, we know
  \[d(\overline{\BL}^{\prime},\overline{\BL}^{\prime\prime})\geq |\overline{\BL}^{\prime}|^{1+2\rho}\]
  and so (\ref{dynamical systems1034}) holds in this situation. It remains to show that if $\overline{\BL}^{\prime}\overset{\Gamma}{\sim}\overline{\BL}^{\prime\prime}$,  (\ref{dynamical systems1034}) also holds.
   Suppose that $\overline{\BL}^{\prime}\overset{\Gamma}{\sim}\overline{\BL}^{\prime\prime}$, and $g_{x^{\prime\prime}}=u^{s}g_{\overline{x}^{\prime}}$. It follows that
  \begin{equation}\label{dynamical systems1033}
   g_{y^{\prime\prime}}= u^{t}g_{\overline{y}^{\prime}}\gamma\ \ \ \text{ for some }e\neq \gamma\in\Gamma.
  \end{equation}
 Then via (\ref{timechange26}) and (\ref{time change178}), there exists $r\geq\max\{e^{(\frac{1}{2}+\delta)^{-1}T_{0}},\Leb(\overline{\BL}^{\prime})\}$  such that
  \begin{align}
   g_{\overline{y}^{\prime}}=& h^{\overline{y}^{\prime}}_{\overline{x}^{\prime}}\exp (v^{\overline{y}^{\prime}}_{\overline{x}^{\prime}})  u^{-s} g_{x^{\prime\prime}}\;\nonumber\\
 g_{\overline{y}^{\prime}} \gamma=& u^{-t}  h^{y^{\prime\prime}}_{x^{\prime\prime}}\exp (v^{y^{\prime\prime}}_{x^{\prime\prime}}) g_{x^{\prime\prime}}  \label{dynamical systems1028}
\end{align}
where $ g_{\overline{y}^{\prime}} g_{\overline{x}^{\prime}}^{-1}=h^{\overline{y}^{\prime}}_{\overline{x}^{\prime}}\exp (v^{\overline{y}^{\prime}}_{\overline{x}^{\prime}})  $, $g_{y^{\prime\prime}} g_{x^{\prime\prime}}^{-1}=h^{y^{\prime\prime}}_{x^{\prime\prime}}\exp (v^{y^{\prime\prime}}_{x^{\prime\prime}})$ can be estimated   by Lemma \ref{time change181}
\[h^{\overline{y}^{\prime}}_{\overline{x}^{\prime}}=\left[
            \begin{array}{ccc}
               1+O(\epsilon) &   O(r^{-1-2\rho})  \\
              O(\epsilon) &   1+O(\epsilon)\\
            \end{array}
          \right],\ \ \  v^{\overline{y}^{\prime}}_{\overline{x}^{\prime}}=O(r^{-\xi\varsigma})v_{0}+\cdots+O(\epsilon)v_{\varsigma}  \]
\[h^{y^{\prime\prime}}_{x^{\prime\prime}}=\left[
            \begin{array}{ccc}
               1+O(\epsilon) &   O(r^{-1-2\rho})  \\
              O(\epsilon) &   1+O(\epsilon)\\
            \end{array}
          \right],\ \ \ v^{y^{\prime\prime}}_{x^{\prime\prime}}=O(r^{-\xi\varsigma})v_{0}+\cdots+O(\epsilon)v_{\varsigma}.\]

  Now assume that one of $s,t$ is not greater than $r^{1+\rho}$. Then by (\ref{timechange25}) and (\ref{timechange30}), we know
  \begin{equation}\label{dynamical systems1030}
    0<s,t\leq O(r^{1+\rho}).
  \end{equation}
 Since $r>e^{(1+2\delta)^{-1}T_{0}}$, let $e^{\omega_{0}}\coloneqq   r^{1+2\delta}$ and we know $\omega_{0}>T_{0}$. Since $x^{\prime\prime}\in K\subset K_{2}$, it follows from the choice of $K_{2}$ and $T_{0}$  that the relative length measure of $K$ on $[x^{\prime\prime},a^{\omega_{0}}x^{\prime\prime}]$ is greater than $1-\sigma$. This implies that there is $\omega$ satisfying
  \[(1-\sigma)\omega_{0}<\omega\leq\omega_{0}\]
  such that $a^{\omega}x^{\prime\prime}\in K$ and therefore
  \begin{equation}\label{dynamical systems1031}
    a^{\omega}g_{x^{\prime\prime}}\in \pi^{-1}(K)
  \end{equation}
  where $a^{\omega}$ denotes the diagonal action.
  On the other hand, by (\ref{dynamical systems1028}), we have
  \begin{align}
  a^{\omega} g_{\overline{y}^{\prime}} =& (a^{\omega} h^{\overline{y}^{\prime}}_{\overline{x}^{\prime}}a^{-\omega})\exp (\Ad a^{\omega}. v^{\overline{y}^{\prime}}_{\overline{x}^{\prime}}) (a^{\omega}u^{-s}a^{-\omega}) a^{\omega} g_{x^{\prime\prime}}\;\nonumber\\
  a^{\omega} g_{\overline{y}^{\prime}}\gamma=& (a^{\omega}u^{-t} a^{-\omega}) (a^{\omega}h^{y^{\prime\prime}}_{x^{\prime\prime}}a^{-\omega})\exp (\Ad a^{\omega}.v^{y^{\prime\prime}}_{x^{\prime\prime}}) a^{\omega}g_{x^{\prime\prime}}  \label{dynamical systems1029}
\end{align}
 Notice that by the choice of $\omega$, we have
  \[e^{\omega/2}\in[ r^{(1-\sigma)(\frac{1}{2}+\delta)},r^{\frac{1}{2}+\delta}].\]
  Then according to  (\ref{dynamical systems1029}) and Lemma \ref{time change179}, we get
 \[ a^{\omega}\left[
            \begin{array}{ccc}
               1+O(\epsilon) &   O(r^{-1-2\rho})  \\
              O(\epsilon) &   1+O(\epsilon)\\
            \end{array}
          \right] a^{-\omega}=  \left[
            \begin{array}{ccc}
               1+O(\epsilon) &   O(r^{2\delta-2\rho})  \\
              O(\epsilon) &   1+O(\epsilon)\\
            \end{array}
          \right] \]
 \begin{align}
    &\Ad a^{\omega}. (O(r^{-\xi\varsigma})v_{0}+\cdots+O(\epsilon)v_{\varsigma})\; \label{dynamical systems3003}\\
= &  \left\{\begin{array}{ll}
O(r^{ 1+2\delta -2\xi })v_{0}+O(\epsilon)v_{1}+O(r^{ - (1-\sigma)(1+2\delta) })v_{2}&,\text{ if }   \varsigma=2\\
O(r^{\frac{1+2\delta-2\xi}{2} })v_{0} +O(r^{ - (1-\sigma)(\frac{1}{2}+\delta) })v_{1}&,\text{ if }   \varsigma=1\\
 O(\epsilon)v_{0} &, \text{ if }   \varsigma=0
\end{array}\right.      \nonumber
\end{align}
\[a^{\omega}u^{-t}a^{-\omega}=   u^{-te^{-\omega}}  , \ \  a^{\omega}u^{-s}a^{-\omega}=   u^{-se^{-\omega}}.   \]
 Thus,  by (\ref{dynamical systems3003}) and (\ref{time change182}), we can require $T_{0}$ sufficiently large, and then $r$ will be large so that
 \[ a^{\omega} h^{\overline{y}^{\prime}}_{\overline{x}^{\prime}}a^{-\omega},\ \exp (\Ad a^{\omega}. v^{\overline{y}^{\prime}}_{\overline{x}^{\prime}}),\    a^{\omega}h^{y^{\prime\prime}}_{x^{\prime\prime}}a^{-\omega},\ \exp (\Ad a^{\omega}.v^{y^{\prime\prime}}_{x^{\prime\prime}})\in B_{G}(e,\epsilon) .  \]
 On the other hand, by (\ref{time change183}), we have
 \[1+\rho-(1-\sigma)(1+2\delta)=1+\rho-(1-\sigma)(1+\frac{3}{2}\rho)<-\frac{1}{4}\rho\]
 and then by (\ref{dynamical systems1030})
  \[ |-te^{-\omega}|=O(r^{-\frac{1}{4}\rho})<\Delta,\ \ \ |-se^{-\omega}|=O(r^{-\frac{1}{4}\rho})<\Delta.\]
  It follows from (\ref{dynamical systems1026}), (\ref{dynamical systems1029})  that
   \[d_{G}(a^{\omega} g_{\overline{y}^{\prime}} , a^{\omega} g_{x^{\prime\prime}})<2\Delta\ \text{ and }\ d_{G}( a^{\omega} g_{\overline{y}^{\prime}} \gamma, a^{\omega} g_{x^{\prime\prime}})<2\Delta.\]
   Then by (\ref{timechange31}) and (\ref{dynamical systems1031}), we conclude
   $\gamma=e$,
   which contradicts (\ref{dynamical systems1033}). Thus, both $s,t$ are greater than $r^{1+\rho}$, and (\ref{dynamical systems1034}) follows.
\end{proof}

\begin{proof}[Proof of Proposition \ref{timechange20}] Let $I=A_{\lambda}= A\cap[0,\lambda]$, $\mathcal{G}$ be the subintervals of $I$ obtained by taking the $x$-coordinate of $\epsilon$-blocks in $\beta_{\rho}$. Note that according to the hypotheses,
\[\Leb\bigg(\bigcup_{J\in \mathcal{G}}J\bigg)\geq\Leb(A)\geq (1-\theta)\lambda.\]
Then  by (\ref{dynamical systems1034}), we can use Proposition \ref{dynamical systems1011} and obtain a good interval $J_{\lambda}=[x(\lambda),\overline{x}(\lambda)]\in\mathcal{G}$ satisfying
\begin{equation}\label{timechange34}
  \Leb(J_{\lambda})>\frac{3}{4}\lambda.
\end{equation}
 Correspondingly, there is a $\epsilon$-block  $\overline{\BL}(\lambda)=\{(x(\lambda),y(\lambda)),(\overline{x}(\lambda),\overline{y}(\lambda))\}\in\beta_{\rho}$ such that its $x$-coordinate has length greater than $\frac{3}{4}\lambda$. Then by (\ref{timechange26}) (and also the choice of $\xi(2\rho)$ in (\ref{time change182})), we get
   \begin{equation}\label{timechange32}
     y(\lambda)=h_{\lambda}\exp(v_{\lambda})x(\lambda).
   \end{equation}
    where
    \[h_{\lambda}=\left[
            \begin{array}{ccc}
               1+O(r^{-2\rho}) &   O(r^{-1-2\rho})  \\
              O(\epsilon) &   1+O(r^{-2\rho})\\
            \end{array}
          \right],\ \ \ v_{\lambda}=O(r^{- \frac{1+2\rho}{2}\varsigma})v_{0}+\cdots+O(\epsilon)v_{\varsigma}\]
  for some  $r\geq\max\{e^{(\frac{1}{2}+\delta)^{-1}T_{0}},\Leb(\overline{\BL}(\lambda))\}>\frac{3}{4}\lambda$.
Therefore, we complete the proof of Proposition \ref{timechange20}.
\end{proof}

\subsection{Proof of Theorem \ref{timechange14}}
Now we are in the position to verify the equivariant properties of $\psi$ (Theorem \ref{timechange14}) via Proposition \ref{timechange20}. We first consider the central direction $c\in C_{G}(U)$ in Lemma \ref{timechange3}. Then with the ``$a$-adjustment", we study the opposite unipotent direction $\tilde{u}$ in Lemma \ref{timechange12}. Finally, since
the central and the opposite unipotent directions generate the whole Lie group, we can obtain the rigidity of $\psi$ with the help of Ratner's theorem.

 Recall that in (\ref{time change224}), we have defined the cocycle $\xi$. Now define $z:X\times\mathbf{R}\rightarrow\mathbf{R}$ by the relation $t=\xi(x,z(x,t))$, i.e.
 \begin{equation}\label{time change225}
   t=\int_{0}^{z(x,t)}\tau(\phi^{U}_{s}(x))ds.
 \end{equation}
         Then by the conjugate assumption in Theorem \ref{timechange14}, we know that
         \[\psi(u^{t}x)=\phi_{t}^{U,\tau}(\psi(x))= u^{z(\psi(x),t)}\psi(x).\]
           Moreover, using ergodic theorem, we get
         \begin{equation}\label{timechange15}
          |t-z(\psi(x),t)|=o(t)
         \end{equation}
         for $\mu$-almost all $x\in X$. Further, according to   Lemma 3.1 \cite{ratner1986rigidity}, when $\tau\in\mathbf{K}(X)$, we have the effective ergodicity: given $\sigma>0$, there is $K=K(\sigma)\subset X$ with $\mu(K)>1-\sigma$ and $t_{K}$ such that
         \begin{equation}\label{timechange39}
           |t-z(\psi(x),t)|=O(t^{1-\eta})
         \end{equation}
         for some $\eta>0$,   all $t\geq t_{K}$ and $x\in K$.
\begin{prop}[Lusin's theorem]\label{timechange21}
  Let $(X,\mathscr{B},\mu)$ be the completion of $\mu$ on the Borel $\sigma$-algebra $\mathcal{B}_{X}$. Let $\psi:X\rightarrow X$ be measurable. Then given $\sigma\in(0,1)$, there is $K=K(\sigma)\in\mathscr{B}$, $\mu(K)>1-\sigma$ such that $\psi$ is uniformly continuous on $K$.
\end{prop}

By (\ref{timechange39}),  there are $P_{n}\subset X$ with $\mu (P_{n})>1- 2^{-n}$ and $\lambda_{n}$ such that
\begin{equation}\label{timechange35}
 |t-z(\psi(x),t)|=O(t^{1-\eta})
\end{equation}
         for some $\eta>0$,   all $t\geq \lambda_{n}$ and $x\in P_{n}$.   Now
let   $a_{n}\coloneqq a^{(1+\gamma)\log\lambda_{n}}$ for some $\gamma\in(0,2\rho)$, and let
\begin{equation}\label{time change228}
  \Psi_{n}(x)\coloneqq a_{n}\psi(a_{-n}x).
\end{equation}
 Our goal is to  show that $\Psi_{n}$, after passing to a subsequence, has a pointwise limit $\Psi$ as $n\rightarrow\infty$. By the ergodic theorem (\ref{timechange15}), $\Psi$ will be $u^{t}$-equivariant, and then \textit{Ratner theorem}\index{Ratner theorem} applies. First, by an elementary argument, we have
\begin{lem}\label{timechange38}
   For $\mu$-almost all $x\in X$, there exists a subsequence $\{n(x,l)\}_{l\in\mathbf{N}}\subset\mathbf{N}$ and $y(x)\in X$ such that
   \[\lim_{l\rightarrow\infty}\Psi_{n(x,l)}(x)= y(x).\]
\end{lem}
\begin{proof}
   Write $X= \bigcup_{n=1}K_{n}$, where $K_{n}$ are compact and $\mu(  K_{n})\nearrow 1$  as $n\rightarrow\infty$. We claim that $\mu(\Omega)=1$, where
  \[\Omega\coloneqq\bigcup_{n\geq1}\bigcap_{k\geq1}\bigcup_{m\geq k}\Psi_{m}^{-1}(K_{n}).\]
  But this follows from a direct calculation (recall that  $d\mu_{\tau}\coloneqq\tau d\mu$)
  \begin{align}
\mu\left(\bigcup_{n\geq1}\bigcap_{k\geq1}\bigcup_{m\geq k}\Psi_{m}^{-1}(K_{n})\right)\geq& \mu\left(\bigcap_{k\geq1}\bigcup_{m\geq k}\Psi_{m}^{-1}(K_{n})\right) \;\nonumber\\
=&\lim_{k\rightarrow\infty} \mu\left(\bigcup_{m\geq k}\Psi_{m}^{-1}(K_{n})\right) \geq \mu_{\tau}(a_{m}^{-1} K_{n}). \nonumber
\end{align}
Note that $\mu(a_{m}^{-1} K_{n})=\mu(K_{n})\nearrow 1$ as $n\rightarrow\infty$. Then since $\mu_{\tau}$ and $\mu$ are equivalent, $ \mu_{\tau}(a_{m}^{-1} K_{n} )\nearrow 1$ as $n\rightarrow\infty$ and the claim follows.
  For almost all $x\in \Omega$,   there exists $n\geq 1$ such that $ \Psi_{m}(x)\in K_{n}$ for infinitely many $m$.
\end{proof}

However, the subsequence  $\{n(x,l)\}_{l\in\mathbf{N}}\subset\mathbf{N}$ obtained from Lemma \ref{timechange38} relies on $x\in X=G/\Gamma$. To get rid of this, we manage to verify that the limit of $\Psi_{n}$ has some equivariant properties.
We first prove  (\ref{dynamical systems2000}) of Theorem \ref{timechange14}:

\begin{lem}\label{timechange3}
There is a measurable map   $\varpi:X\times C_{G}(U)\rightarrow  C_{G}(U)$  such that
\begin{equation}\label{time change227}
 \psi(cx)=\varpi(x,c)\psi(x)
\end{equation}
           for $c\in  C_{G}(U)$, $\mu$-almost all $x\in X$.
\end{lem}
\begin{proof} By Proposition \ref{timechange21}, for some $\sigma\in(0,1)$ close to $0$,  there is  $K_{1}\subset X$ such that $\mu(K_{1})>1-\sigma$ and $\psi|_{K_{1}}$ is uniformly continuous. On the other hand, by (\ref{timechange39}), there are $K_{2}\subset X$ with $\mu(K_{2})>1-\sigma$ and $t_{K_{2}}$ such that
         \begin{equation}\label{timechange16}
           |t-z(y,t)|=O(t^{1-\eta})
         \end{equation}
         for some $\eta>0$,   all $t\geq t_{K_{2}}$ and $y\in K_{2}$.
Then letting $m=t_{K_{2}}$ and according to Proposition \ref{timechange20}, we obtain quantities $\rho,\theta,\epsilon$ and a compact set  $K_{3}$ with $\mu(K_{3})>1-\sigma$. Let $K\coloneqq K_{1}\cap K_{2}\cap K_{3}$ and $\sigma$ so small that $\sigma\ll\theta$. Then $\mu(K)>1-3\sigma$. We shall study the unipotent orbits on $K$.

  Next, fix a sufficiently small $\delta>0$ so that $d(\psi(x),\psi(y))<\epsilon$ whenever $d(x,y)<\delta$ and $x,y\in K$. Given a $u^{t}$-generic point $x\in X$, there is $A_{x}\subset\mathbf{R}^{+}$ such that $\psi(u^{s}cx)\in B(\psi(u^{s}x),\epsilon)$ for all $s\in A_{x}$. More precisely,  given $c\in B_{G}(e,\delta)\cap C_{G}(U)$, by ergodic theorem, there is $\lambda_{0}\gg m$ such that
   \[u^{z(\psi(cx) ,s)}\psi(cx)\in B(u^{z(\psi(x),s)}\psi(x),\epsilon)\]
  for  $s\in A_{x}$ and $\Leb(A_{x}\cap[0,\lambda])\geq (1-\sigma )\lambda$ whenever $\lambda\geq\lambda_{0}$.
     Then via Proposition \ref{timechange20} (see also (\ref{timechange34})) for any $\lambda_{k}\geq\lambda_{0}$, there is an interval  $J_{k}=\{(x_{k},y_{k}),(\overline{x}_{k},\overline{y}_{k})\}\subset[0,\lambda_{k}]$ with  $(x_{k},y_{k})\overset{e}{\sim}(\overline{x}_{k},\overline{y}_{k})$ and $x_{k}=u^{z(\psi(x),s_{k})}\psi(x)$, $y_{k}=u^{z(\psi(cx),s_{k})}\psi(cx)$ for some $s_{k}\in\mathbf{R}^{+}$ such that
     \[ y_{k}
         =h_{k}\exp(v_{k}) x_{k}\]
    where
    \[h_{k}=\left[
            \begin{array}{ccc}
               1+O(\lambda_{k}^{-2\rho}) &   O(\lambda_{k}^{-1-2\rho})  \\
              O(\epsilon) &   1+O(\lambda_{k}^{-2\rho})\\
            \end{array}
          \right],\ \ \ v_{k}=O(\lambda_{k}^{- \frac{1+2\rho}{2}\varsigma})v_{0}+\cdots+O(\epsilon)v_{\varsigma}\]
and $|J_{k}|\geq \frac{3}{4}\lambda_{k}$.  Then we can choose an increasing sequence of $\lambda_{k}$ so that $J_{n}\cap J_{m}\neq\emptyset$ for any $n, m\in\mathbf{N}$. But it forces $J_{k}\subset J_{k+1}$. Thus, write $s_{x}=\inf \bigcup_{k}J_{k}$ and we conclude
          \[
          u^{z(\psi(cx),s_{x})}\psi(cx)=h_{x} \exp(v_{x})u^{z(\psi(x),s_{x})} \psi(x),\ \ \ h_{x}=\left[
            \begin{array}{ccc}
               1  &     \\
              O(\epsilon) &   1 \\
            \end{array}
          \right],\ \ \ v_{x}\in C_{G}(U). \]
Thus, for  $c\in B_{G}(e,\delta)\cap C_{G}(U)$, we set
\[\varpi(x,c)\coloneqq u^{-z(\psi(cx),s_{x})}h_{x} \exp(v_{x})u^{z(\psi(x),s_{x})}. \]
For general $c\in C_{G}(U)$, we can define it by iteration, since
\begin{equation}\label{timechange24}
  \varpi(x,c^{k})=\prod_{j=0}^{k-1}\varpi(c^{j}x,c).
\end{equation}
The consequence follows.
\end{proof}
Now we explore   further properties of $\varpi$. First of all,  $\varpi$ satisfies
\begin{equation}\label{time change216}
 u^{z(\psi(cx),t)}\varpi(x,c)\psi(x)=\varpi(u^{t}x,c)u^{z(\psi(x),t)}\psi(x)
\end{equation}
for $\mu$-a.e. $x\in X$. Moreover
\begin{lem}
  For $t\in\mathbf{R}$, we have
   \begin{equation}\label{time change185}
  u^{z(\psi(cx),t)}\varpi(x,c)=\varpi(u^{t}x,c)u^{z(\psi(x),t)}
\end{equation}
for $\mu$-a.e. $x\in X$.
\end{lem}
\begin{proof}
   First of all, by (\ref{time change225}), we have the cocycle identity
   \begin{equation}\label{time change217}
     z(\psi(x),T+t)=z(\psi(u^{t}x),T)+z(\psi(x),t)
   \end{equation}
   for all $t,T\in\mathbf{R}$ and $\mu$-a.e. $x\in X$. Let $g_{x}\in G$ be a representative of $\psi(x)$, i.e. $g_{x}\Gamma=\psi(x)$. Then for any $t\in\mathbf{R}$, we deduce from (\ref{time change216}) that
   \[g^{-1}_{x}Y_{c}(x,t)g_{x}\in \Gamma \]
   for  $\mu$-a.e. $x\in X$, where
   \[Y_{c}(x,t)\coloneqq(\varpi(u^{t}x,c)u^{z(\psi(x),t)})^{-1}u^{z(\psi(cx),t)}\varpi(x,c).\]
    It follows that there is a $\gamma\in\Gamma$ such that
   \[|\{t\in\mathbf{R}:g^{-1}_{x}Y_{c}(x,t)g_{x}= \gamma \text{ for }\mu\text{-a.e. } x\in X\}|>0.\]
   Fix arbitrary $t_{0}\in\mathbf{R}$ such that $g^{-1}_{x}Y_{c}(x,t_{0})g_{x}= \gamma$. Then we have
   \begin{equation}\label{time change226}
   |\{T\in\mathbf{R}:Y_{c}(x,T+t_{0})=Y_{c}(x,t_{0}) \text{ for }\mu\text{-a.e. } x\in X\}|>0.
   \end{equation}
   On the other hand, note that by (\ref{time change217}), we have
   \[Y_{c}(x,T+t_{0})(Y_{c}(x,t_{0}))^{-1}=Y_{c}(u^{t_{0}}x,T).\]
   Then after replacing $x$ by $u^{t_{0}}x$, we can assume that (\ref{time change226}) holds for $t_{0}=0$.

   Now suppose that there are $T,t\in\mathbf{R}$ such that
   \begin{equation}\label{time change218}
    Y_{c}(x,T)=Y_{c}(x,t)=Y_{c}(x,0)\equiv e
   \end{equation}
   for $\mu$-a.e. $x\in X$. Then we claim that $ Y_{c}(x,T+t)=e$ as well. In fact, replacing $x$ by $u^{t}x$ and using (\ref{time change217}), (\ref{time change218}),  we get
    \begin{align}
e= &Y_{c}(u^{t}x,T)=(\varpi(u^{T+t}x,c)u^{z(\psi(u^{t}x),T)})^{-1}u^{z(\psi(cu^{t}x),T)}\varpi(u^{t}x,c)\;\nonumber\\
 =& (\varpi(u^{T+t}x,c)u^{z(\psi(u^{t}x),T)+z(\psi(x),t)})^{-1}u^{z(\psi(cu^{t}x),T)}(\varpi(u^{t}x,c)u^{z(\psi(x),t)}) \;\nonumber\\
 =& (\varpi(u^{T+t}x,c)u^{z(\psi(x),T+t)})^{-1}u^{z(\psi(cu^{t}x),T)}(u^{z(\psi(cx),t)}\varpi(x,c))\;\nonumber\\
=& (\varpi(u^{T+t}x,c)u^{z(\psi(x),T+t)})^{-1}u^{z(\psi(cx),T+t)}\varpi(x,c)=Y_{c}(x,T+t) \;  \nonumber
\end{align}
for $\mu$-a.e. $x\in X$. Thus, we see that $\{T\in\mathbf{R}:Y_{c}(x,T)=e \text{ for }\mu\text{-a.e. } x\in X\}$ is a group with positive Lebesgue measure, which can only be the whole $\mathbf{R}$.
\end{proof}

In light of (\ref{time change185}), we consider the orthogonal decomposition (\ref{time change184}) and write
\begin{equation}\label{time change188}
  \varpi(x,c)=u^{\alpha(x,c)} \beta(x,c)
\end{equation}
where $\alpha(x,c)\in\mathbf{R}$ and $\beta(x,c)\in\exp V^{\perp}_{C}$. Then by (\ref{time change185}), we have
\begin{equation}\label{time change186}
  z(\psi(cx),t)+\alpha(x,c)=\alpha(u^{t}x,c)+z(\psi(x),t),\ \ \ \beta(x,c)=\beta(u^{t}x,c)
\end{equation}
for all $t\in\mathbf{R}$. Via the ergodicity of the unipotent flow $u^{t}$, we conclude that
\[\beta(x,c)\equiv\beta(c)\]
 for all $c\in C_{G}(U)$.
Besides, $\beta:C_{G}(U)\rightarrow\exp V^{\perp}_{C}$ must be surjective. This is because $\psi$ is bijective and so for a.e. $x\in X$,  $\varpi(x,\cdot):C_{G}(U)\rightarrow C_{G}(U)$ is surjective.

 On the other hand, consider
\[F(x,c_{1},c_{2})\coloneqq  (\varpi(x,c_{1}c_{2}))^{-1}\varpi(c_{2}x,c_{1})  \varpi(x,c_{2})\]
for $x\in X$, $c_{1},c_{2}\in C_{G}(U)$.
By (\ref{time change185}), one can show that
\[F(u^{t}x,c_{1},c_{2})=F(x,c_{1},c_{2})\]
for $c_{1},c_{2}\in C_{G}(U)$, a.e. $x\in X$. Then by the ergodicity,  $F(x,c_{1},c_{2})\equiv F(c_{1},c_{2})$. Besides, by  (\ref{time change227}), we know that
\[\psi(x)=F(x,c_{1},c_{2})\psi(x).\]
Since $\psi$ is bijective, we conclude that
\[  \varpi(x,c_{1}c_{2})=  \varpi(c_{2}x,c_{1})  \varpi(x,c_{2})\]
for a.e. $x$, $c_{1},c_{2}\in C_{G}(U)$. In particular, we have $\beta(c_{1} c_{2})=\beta(c_{1})\beta(c_{2})$. Further,
 we always have $\beta(u^{t})\equiv e$. Therefore, we can restrict our attention to $\exp V^{\perp}_{C}$ and conclude that $d\beta: V^{\perp}_{C}\rightarrow V^{\perp}_{C}$ is an automorphism.

Lemma \ref{timechange3} can be interpreted by the language of cohomology.  More precisely, Lemma \ref{timechange3} implies the time change  $\tau$ and $\tau\circ c$ are measurably cohomologous.
    \begin{cor}\label{dynamical systems2001} Let  $\tau\in \mathbf{K}(X)$. Suppose that there is a measurable conjugacy map  $\psi:(X,\mu)\rightarrow (X,\mu_{\tau})$ such that
            \[\psi (\phi^{U}_{t}(x))=\phi_{t}^{U,\tau}(\psi(x))\]
            for $t\in\mathbf{R}$ and $\mu$-a.e. $x\in X$. Then
            $\tau(x)$ and $\tau( cx)$ are (measurably) cohomologous for all $c\in C_{G}(U)$. Besides, $\alpha(\cdot,c)\in L^{1}(X)$ for some $c\in C_{G}(U)$ iff  the transfer function is   in $L^{1}$.
          \end{cor}
\begin{proof}
  By (\ref{time change186}), we have
  \begin{align}
 &\int_{0}^{z(\psi(x),t)}\tau(u^{s}\psi(x))-\tau(u^{s}\beta(c)\psi(x))ds\;\nonumber\\
 =&\int_{0}^{z(\psi(cx),t)}\tau(u^{s}\psi(cx))ds-\int_{0}^{z(\psi(x),t)}\tau(u^{s}\beta(c)\psi(x))ds\;\nonumber\\
=&\int_{0}^{z(\psi(cx),t)}\tau(u^{\alpha(x,c)+s}\beta(c)\psi(x))ds-\int_{0}^{z(\psi(x),t)}\tau(u^{s}\beta(c)\psi(x))ds\;\nonumber\\
=&\int_{0}^{\alpha(x,c)+z(\psi(cx),t)}\tau(u^{s}\beta(c)\psi(x))ds-\int_{0}^{\alpha(x,c)}\tau(u^{s}\beta(c)\psi(x))ds-\int_{0}^{z(\psi(x),t)}\tau(u^{s}\beta(c)\psi(x))ds\;\nonumber\\
 =&\int_{0}^{z(\psi(x),t)+\alpha(u^{t}x,c)}\tau(u^{s}\beta(c)\psi(x))ds-\int_{0}^{z(\psi(x),t)}\tau(u^{s}\beta(c)\psi(x))ds-\int_{0}^{\alpha(x,c)}\tau(u^{s}\beta(c)\psi(x))ds\;\nonumber\\
=& \int_{0}^{\alpha(u^{t}x,c)}\tau(u^{s}\beta(c)\psi(u^{t}x))ds-\int_{0}^{\alpha(x,c)}\tau(u^{s}\beta(c)\psi(x))ds. \;  \nonumber
\end{align}
Thus, we can take the transfer function as
\[g_{c}(y)\coloneqq\int_{0}^{\alpha(\psi^{-1}(y),c)}\tau(u^{s}\beta(c)y)ds.\]
  Then $\tau(x)$ and $\tau( \beta(c)x)$ are (measurably) cohomologous for all $c\in C_{G}(U)$. Since $d\beta: V^{\perp}_{C}\rightarrow V^{\perp}_{C}$ is surjective, this is equivalent to say that  $\tau(x)$ and $\tau( cx)$ are (measurably) cohomologous for all $c\in C_{G}(U)$.
Since $\tau$ is bounded, we conclude that $g_{c}\in L^{1}(X)$ iff $\alpha(\cdot,c)\in L^{1}(X)$.
\end{proof}

If $\tau(x)$ and $\tau( \beta(c)x)$ are cohomologous with a $L^{1}$ transfer function, then we are able to do more via the \textit{ergodic theorem}.
\begin{lem}\label{time change202106.5}
  If $\alpha(\cdot,c)\in L^{1}(X)$ for all $c\in C_{G}(U)$, then   for any $c\in \exp (V^{\perp}_{C}\cap\mathfrak{g}_{-1})$ (see (\ref{time change187})), there exists $C\in \exp (V^{\perp}_{C}\cap\mathfrak{g}_{-1})$ such that
   \[ \lim_{l\rightarrow\infty}d(\Psi_{n}(cx),C\Psi_{n}(x))=0\]
   for $\mu$-almost all $x\in X$, where $\Psi_{n}$ is given by (\ref{time change228}).
\end{lem}
\begin{proof} Fix an (orthonormal) basis $\{U,V_{1},\ldots,V_{n-2}\}\subset\mathfrak{g}_{-1}$. For $c_{i}=\exp V_{i}$, $\phi^{V_{i}}_{t}(x)=\exp(tV_{i})x=c_{i}^{t}x$ defines an (ergodic) unipotent flow.
  Thus, if  $\alpha(\cdot,c_{i})$ is integrable, via (\ref{timechange24}), (\ref{time change188}) and ergodic theorem, we obtain
\begin{equation}\label{timechange6}
  \left|\frac{1}{k}\alpha(x,c_{i}^{k})-\int\alpha(y,c_{i})d\mu(y)\right|\rightarrow0
\end{equation}
for $\mu$-almost all $x\in X$.
Thus, by Lemma \ref{timechange3} and (\ref{time change188}), one can calculate
          \begin{multline}
          \Psi_{n}(c_{i}x)= a_{n}\psi(a_{-n}c_{i}x)
=  a_{n}\psi(a_{-n}c_{i}a_{n}a_{-n}x)
=a_{n} \psi( c_{i}^{\lambda_{n}^{1+\gamma}}a_{-n}x) \\
    = a_{n} u^{\alpha(a_{-n}x,c_{i}^{\lambda_{n}^{1+\gamma}})} \beta(c_{i}^{\lambda_{n}^{1+\gamma}}) \psi(a_{-n}x)  =  u^{\lambda_{n}^{-(1+\gamma)}\alpha(a_{-n}x,c_{i}^{\lambda_{n}^{1+\gamma}})} \cdot a_{n}\beta(c_{i}^{\lambda_{n}^{1+\gamma}}) a_{-n}\cdot \Psi_{n}(x) .     \nonumber
          \end{multline}

Since $V_{i}\in V^{\perp}_{C}\cap \mathfrak{g}_{-1}$ is nilpotent (recall (\ref{time change212})), the fact that $d\beta: V^{\perp}_{C}\rightarrow V^{\perp}_{C}$ is an automorphism implies $d\beta(V_{i})\in \mathfrak{g}_{-1}$ is also nilpotent.
Write $\beta(c_{i})=\exp(v_{i})$, where $v_{i}\in\mathfrak{g}_{-1}$. Then
\[a_{n}\beta(c_{i}^{\lambda_{n}^{1+\gamma}}) a_{-n}=a_{n}\beta(c_{i})^{\lambda_{n}^{1+\gamma}} a_{-n}=\exp(v_{i}).\]

Next, by (\ref{timechange6}), we can enlarge $\lambda_{n}$ so that $\mu(W_{n})>1-2^{-n}$, where
\[W_{n}\coloneqq\left\{y\in X:\left|\frac{1}{\lambda_{n}^{1+\gamma}} \alpha(y, c^{\lambda_{n}^{1+\gamma}})-\int\alpha(\cdot,c)\right|<\frac{1}{n}\right\}.\]
It follows that $\mu(\bigcup_{m\geq1}\bigcap_{n\geq m}a_{n}W_{n})=1$. Then for any $x\in \bigcup_{m\geq1}\bigcap_{n\geq m}a_{n}W_{n}$, there exists a number $m>0$ such that for any $n\geq m$, we have
\[\left|\frac{1}{\lambda_{n}^{1+\gamma}} \alpha(a_{n}^{-1}x, c^{\lambda_{n}^{1+\gamma} })-\int\alpha(\cdot,c)\right|<\frac{1}{n}.\]
Thus,  we conclude that for $\mu$-almost all $x\in X$,
\[ \lim_{n\rightarrow\infty}d_{G}(u^{\lambda_{n}^{-(1+\gamma)}\alpha(a_{-n}x,c^{\lambda_{n}^{1+\gamma}})} ,u^{\int\alpha(\cdot,c)} )=0.\]
The consequence follows.
\end{proof}

 Let $\mathfrak{sl}_{2}(\mathbf{R})=\Span\{U,Y_{n},\tilde{U}\}\subset\mathfrak{g}$ be a $\mathfrak{sl}_{2}$-triple, $\tilde{u}=\exp(\tilde{U})$. It is again convenient to consider  $u,a,\tilde{u}\in SO(2,1)$ as $(2\times 2)$-matrices. Then we have
 \begin{lem}\label{timechange12}  Let the notation and assumption be as above. For   $\delta>0$, let $\tilde{u}^{p}\in B_{G}(e,\delta)$ for $p\in\mathbf{R}$.      Then  for sufficiently small $\delta>0$ and for $\mu$-almost all $x\in X$, there exists an element $C_{\tilde{u}}(x,p)\in C_{G}(\mathfrak{sl}_{2}(\mathbf{R}))$ such that \[\lim_{n\rightarrow\infty}d(\Psi_{n}(\tilde{u}^{p}x),C_{\tilde{u}}(x,p)\tilde{u}^{p}\Psi_{n}(x))=0.\]
\end{lem}
\begin{proof} Recall  we have defined $P_{n}$ in (\ref{timechange35}). Note  that
          \[\mu\left(\bigcup_{k\geq 1}\bigcap_{n\geq k} a_{n}P_{n}\right)=1.\]
        Suppose that   $x,\tilde{u}^{p}x\in \bigcup_{k\geq 1}\bigcap_{n\geq k} a_{n}P_{n}$.
            Let $t(\lambda_{n})\coloneqq\frac{\lambda_{n}}{1-p\lambda_{n}^{-\gamma}}$ and consider
\begin{multline}
  d(u^{t}a_{n}^{-1}\tilde{u}^{p}x,u^{\lambda_{n}}a_{n}^{-1}x)=d(u^{t}\tilde{u}^{p\lambda_{n}^{-1-\gamma}}a_{n}^{-1}x,u^{\lambda_{n}}a_{n}^{-1}x)\\
 =d\left( \left[
            \begin{array}{cc}
              1-\lambda_{n}^{-\gamma}p  & \lambda_{n}^{-1-\gamma} p \\
              0 &  1+\lambda_{n}^{-1-\gamma}pt\\
            \end{array}
          \right]u^{\lambda_{n}}a_{n}^{-1}x,u^{\lambda_{n}}a_{n}^{-1}x\right)\leq\delta.\nonumber
\end{multline}
     Then by the continuity of $\psi$, we get
 \begin{equation}\label{timechange10}
 d( u^{z(\psi(a_{n}^{-1}\tilde{u}^{p}x),t)}\psi(a_{n}^{-1}\tilde{u}^{p}x),u^{z(\psi(a_{n}^{-1}x),\lambda_{n})}\psi(a_{n}^{-1}x))<\epsilon.
 \end{equation}
 Similarly, letting $\tilde{t}(\lambda_{n})\coloneqq\frac{z(\psi(a_{n}^{-1}x),\lambda_{n})}{1-pz(\psi(a_{n}^{-1}x),\lambda_{n})\lambda_{n}^{-1-\gamma}}$,  we get
 \begin{align}
& d( u^{\tilde{t}}  \tilde{u}^{p\lambda_{n}^{-1-\gamma} } \psi(a_{n}^{-1}x), u^{z(\psi(a_{n}^{-1}x),\lambda_{n})}\psi(a_{n}^{-1}x))\;\nonumber\\
=& d( u^{\tilde{t}} \tilde{u}^{p\lambda_{n}^{-1-\gamma} } u^{-z(\psi(a_{n}^{-1}x),\lambda_{n})}u^{z(\psi(a_{n}^{-1}x),\lambda_{n})}\psi(a_{n}^{-1}x),u^{z(\psi(a_{n}^{-1}x),\lambda_{n})}  \psi(a_{n}^{-1}x) )\;\nonumber\\
=& d\left( \left[
            \begin{array}{cc}
                1-\lambda_{n}^{-1-\gamma}p z( \psi(a_{n}^{-1}x),\lambda_{n})& \lambda_{n}^{-1-\gamma} p \\
              0 &  1+\lambda_{n}^{-1-\gamma}pt\\
            \end{array}
          \right]  u^{z(\psi(a_{n}^{-1}x),\lambda_{n})}  \psi(a_{n}^{-1}x),u^{z(\psi(a_{n}^{-1}x),\lambda_{n})}  \psi(a_{n}^{-1}x)\right)<\delta.\;  \label{timechange11}
\end{align}
 Combining (\ref{timechange10}) with (\ref{timechange11}), we obtain
 \[d(u^{z(\psi(a_{n}^{-1}\tilde{u}^{p}x),t)}\psi(a_{n}^{-1}\tilde{u}^{p}x),u^{\tilde{t}}  \tilde{u}^{p\lambda_{n}^{-1-\gamma} } \psi(a_{n}^{-1}x))\ll\epsilon.\]
In order to apply Proposition \ref{timechange20}, we need to consider
    \begin{align}
&|z(\psi(a_{n}^{-1}\tilde{u}^{p}x),t)-\tilde{t}|\;\nonumber\\
\leq&  |z(\psi(a_{n}^{-1}\tilde{u}^{p}x),t)-t|+|t-\tilde{t}|\;\nonumber\\
=& O(t^{1-\eta})+\left|\frac{\lambda_{n}}{1-p\lambda_{n}^{-\gamma}}-\frac{z(\psi(a_{n}^{-1}x),\lambda_{n})}{1-pz(\psi(a_{n}^{-1}x),\lambda_{n})\lambda_{n}^{-1-\gamma}}\right|  \;\nonumber\\
\leq& O(\lambda_{n}^{1-\eta})+\left|\frac{\lambda_{n}}{1-p\lambda_{n}^{-\gamma}}-\frac{\lambda_{n}}{1-pz(\psi(a_{n}^{-1}x),\lambda_{n})\lambda_{n}^{-1-\gamma}}\right| \;\nonumber\\
&+\left|\frac{\lambda_{n}}{1-pz(\psi(a_{n}^{-1}x),\lambda_{n})\lambda_{n}^{-1-\gamma}}-\frac{z(\psi(a_{n}^{-1}x),\lambda_{n})}{1-pz(\psi(a_{n}^{-1}x),\lambda_{n})\lambda_{n}^{-1-\gamma}}\right| \;\nonumber\\
\leq& O(\lambda_{n}^{1-\eta})+\left|\frac{p\lambda_{n}^{-\gamma}(z(\psi(a_{n}^{-1}x),\lambda_{n})-\lambda_{n})}{(1-p\lambda_{n}^{-\gamma})(1-pz(\psi(a_{n}^{-1}x),\lambda_{n})\lambda_{n}^{-1-\gamma})}\right|+O(\lambda_{n}^{1-\eta})  \;\nonumber\\
=& O(\lambda_{n}^{1-\eta})+  o(\lambda_{n}^{1-\eta})+O(\lambda_{n}^{1-\eta}). \;  \nonumber
\end{align}
Thus,   via Proposition \ref{timechange20}, we conclude that
\[u^{t_{\lambda_{n}}}\psi(a_{n}^{-1}\tilde{u}^{p}x)=h_{n}\exp(v_{n}) u^{s_{\lambda_{n}}}\tilde{u}^{p\lambda_{n}^{-1-\gamma} } \psi(a_{n}^{-1}x) \]
    where
    \[h_{n}=\left[
            \begin{array}{ccc}
               1+O(\lambda_{n}^{-2\rho}) &   O(\lambda_{n}^{-1-2\rho})  \\
              O(\epsilon) &   1+O(\lambda_{n}^{-2\rho})\\
            \end{array}
          \right],\ \ \ v_{n}=O(\lambda_{n}^{- \frac{1+2\rho}{2}\varsigma})v_{0}+\cdots+O(\epsilon)v_{\varsigma}\]
     for some $t_{\lambda_{n}},s_{\lambda_{n}}\ll\lambda_{n}$.

It follows that
 \begin{align}
u^{t_{\lambda_{n}}\lambda_{n}^{-1-\gamma}}\Psi_{n}(\tilde{u}^{p}x)=& a_{n} h_{n}\exp(v_{n}) u^{s_{\lambda_{n}}}\tilde{u}^{p\lambda_{n}^{-1-\gamma} } \psi(a_{n}^{-1}x) \;\nonumber\\
=& a_{n} h_{n}a_{n}^{-1}\exp(\Ad a_{n}. v_{n}) u^{s_{\lambda_{n}}\lambda_{n}^{-1-\gamma}}\tilde{u}^{p} \Psi_{n}(x). \;  \nonumber
\end{align}
Then, one can calculate
\[a_{n} h_{n}a_{n}^{-1}=\left[
            \begin{array}{ccc}
               1+O(\lambda_{n}^{-2\rho}) &    O(\lambda_{n}^{\gamma-2\rho}) \\
               O(\lambda_{n}^{-1-\gamma}) &  1+O(\lambda_{n}^{-2\rho})\\
            \end{array}
          \right],\ \ \ \Ad a_{n}.v_{n}=O(\lambda_{n}^{  \frac{\gamma -2\rho}{2}\varsigma})v_{0}+\cdots+O(\epsilon)v_{\varsigma}.\]
Thus, letting $n\rightarrow\infty$, the consequence follows.
\end{proof}

It is worth noting that $\{c,\tilde{u}^{p}:c\in \exp\mathfrak{g}_{-1},p\in\mathbf{R}\}$ already generates the whole group $G=SO(n,1)$.
Thus, using Lemma   \ref{timechange3},  \ref{timechange12} and \textit{Fubini's theorem}, we get
\begin{cor}\label{timechange37}
   There exists a sufficiently small $\delta>0$, a map $f(g)\in G$ such that   for $\mu$-almost all $x\in X$,  we have
\begin{equation}\label{timechange13}
  \lim_{n\rightarrow\infty}d(\Psi_{n}(gx),f(g)\Psi_{n}(x))=0
\end{equation}
for almost all $g\in B_{G}(e,\delta)$.
\end{cor}

Now fix $x\in X$ so that Corollary \ref{timechange37} and Lemma \ref{timechange38} apply. Then by Lemma \ref{timechange38},   we can fix a universal subsequence $\{n(l)\}_{l\in\mathbf{N}}\subset\mathbf{N}$ and $y\in X$ such that
   \[\lim_{l\rightarrow\infty}\Psi_{n(l)}(x)= y.\]
     Write $\Psi(x)\coloneqq y$. Then,  (\ref{timechange13}) implies that $\Psi_{n(l)}(gx)\rightarrow f(g)y\eqqcolon\Psi(gx)$ as $l\rightarrow\infty$ for $g\in B_{G}(e,\delta)$. Finally, since $u^{t}$ is ergodic, we have $\mu(u^{\mathbf{R}}B_{G}(e,\delta)x)=1$ and
      \[\Psi(u^{t}gx)\coloneqq\lim_{l\rightarrow\infty} \Psi_{n(l)}(u^{t}gx)=\lim_{l\rightarrow\infty} u^{\lambda_{n(l)}^{-1-\gamma} z(\psi(gx),\lambda_{n(l)}^{1+\gamma} t)} \Psi_{n(l)}(gx)=u^{t}\Psi(gx)\]
 is well defined for   $u^{t}g\in u^{\mathbf{R}}B_{G}(e,\delta)$. In other words, we obtain a (surjective) $u^{t}$-equivariant map $\Psi:X\rightarrow X$. Next, consider the graph map $\overline{\Psi}:X\rightarrow X\times X$ defined by
 \[\overline{\Psi}:x\mapsto(x,\Psi(x)).\]
 Then $\overline{\Psi}_{\ast}\mu$ is a $(u^{t}\times u^{t})$-invariant and ergodic measure supported on $\graph(\Psi)$. By \textit{Ratner's theorem}\index{Ratner's theorem}, we conclude that there is a subgroup $S\leq G\times G$ and a point $(x_{0},y_{0})\in X\times X$ such that
 \[\graph(\Psi)=\supp(\overline{\Psi}_{\ast}\mu)=S.(x_{0},y_{0}).\]
It is then not hard to see that $S$ is the graph of an automorphism $\Phi:G\rightarrow G$ (cf.  \cite{morris2005ratner}).  Thus, we see that
\[\Psi(gx_{0})=\Phi(g)y_{0}\]
 is an affine map.
 By Lemma    \ref{timechange12}, we know that  $\Phi(\tilde{u}^{p})=C_{\tilde{u}}(p)\tilde{u}^{p}$ for some $C_{\tilde{u}}(p)\in C_{G}(U)$. On the other hand,  the Jacobson–Morozov theorem asserts that all $\mathfrak{sl}_{2}$-triples are conjugate under the action of the group $C_{G}(U)$. Since $\Phi$ fixes $u^{t}$, we conclude that $\Phi$ fixes $SO(2,1)$.

On the other hand, since
 \[\lim_{l\rightarrow\infty}d(\Psi_{n(l)}(gx),\Psi(gx))=0\]
for   $g\in B_{G}(e,\delta)$, $\mu$-almost all $x\in X$.   Thus, for sufficiently large $l\in\mathbf{N}$, most points $x\in X$, we have
\[\epsilon>d(\Psi_{n(l)}(u^{s}x),\Psi(u^{s}x))=d(u^{\lambda_{n(l)}^{-1-\gamma} z(\psi(gx),\lambda_{n(l)}^{1+\gamma} s)}\Psi_{n(l)}(x),u^{s}\Psi(x))\] for most of the time $s\in\mathbf{R}$. Applying Proposition \ref{timechange20}  to $t(s)=\lambda_{n(l)}^{-1-\gamma} z(\psi(gx),\lambda_{n(l)}^{1+\gamma} s)$ (similar to the proof of Lemma \ref{timechange3}), there exists $c(x)\in C_{G}(U)$ such that
\[\Psi_{n(l)}(x)= c(x)\Psi(x).\]
It follows that there exists a function $c:X\rightarrow\mathbf{C}$ such that
\[\psi(gx_{0})= c(gx_{0}) \Phi(g)y_{0}.\]
Therefore, we have proved Theorem \ref{timechange14}.

Similar to Corollary \ref{dynamical systems2001}, by (\ref{time change171}), we have
 \begin{cor}\label{time change172}  Let  $\tau\in \mathbf{K}(X)$. Suppose that there is a measurable conjugacy map  $\psi:(X,\mu)\rightarrow (X,\mu_{\tau})$ such that
            \[\psi (\phi^{U}_{t}(x))=\phi_{t}^{U,\tau}(\psi(x))\]
            for $t\in\mathbf{R}$ and $\mu$-a.e. $x\in X$. Assume further that
            $\tau(x)$ and $\tau( cx)$ are $L^{1}$-cohomologous for all $c\in C_{G}(U)$. Then $1$ and $\tau$ are cohomologous.
          \end{cor}
\begin{proof}  Write $c(x)=u^{a(x)}b$, i.e. by (\ref{time change171}), $\psi(gx_{0}) =   u^{a(gx_{0})}b\Phi(g)y_{0}$.
  Note that $a(gx_{0})+z(\psi(gx_{0}),t)=t+a(u^{t}gx_{0})$. It follows that
  \begin{align}
 &\int_{0}^{t}1-\tau(u^{s}b  \Phi(g)y_{0})ds\;\nonumber\\
 =& \int_{0}^{z(\psi(gx_{0}),t)}\tau(u^{s}u^{a(gx_{0})}b\Phi(g)y_{0})ds-\int_{0}^{t} \tau(u^{s}b  \Phi(g)y_{0})ds\;\nonumber\\
=&    \int_{0}^{z(\psi(gx_{0}),t)+a(gx_{0})}\tau(u^{s} b\Phi(g)y_{0})ds- \int_{0}^{a(gx_{0})}\tau(u^{s} b\Phi(g)y_{0})ds-\int_{0}^{t} \tau(u^{s}b  \Phi(g)y_{0})ds\;\nonumber\\
=&   \int_{0}^{t+a(u^{t}gx_{0})}\tau(u^{s} b\Phi(g)y_{0})ds- \int_{0}^{a(gx_{0})}\tau(u^{s} b\Phi(g)y_{0})ds-\int_{0}^{t} \tau(u^{s}b  \Phi(g)y_{0})ds\;\nonumber\\
   =&   \int_{0}^{a(u^{t}gx_{0})}\tau(u^{s} b\Phi(u^{t}g)y_{0})ds- \int_{0}^{a(gx_{0})}\tau(u^{s} b\Phi(g)y_{0})ds. \;\nonumber
\end{align}
 Then $1$ and $\tau( b\Phi(g)y_{0})$ are cohomologous. Because $\tau$ and
$\tau\circ b$ are cohomologous by assumption, the consequence follows.
\end{proof}

\section{Restriction of representations}\label{time change115}
\subsection{Unitary representations of $SO(n,1)$}\label{time change118}
Now we adopt the standard notation in \cite{knapp2001representation} Chapter 7 to develop the unitary representation of $G$. Let $\sigma_{\mathbf{n}}$ be an irreducible unitary representation of $M=SO(n-1)$, where $\mathbf{n}$ indicates the highest weight. Besides,  we require $\mathbf{n}=(n_{i})_{1\leq i\leq\lfloor\frac{n-1}{2}\rfloor}$ satisfies
\[\begin{array}{ll}  0\leq n_{1}\leq\cdots\leq n_{k-1} &, \text{ if }n=2k  \\
   |n_{1}|\leq n_{2}\leq\cdots\leq n_{k}  &, \text{ if }n=2k+1  \\
 \end{array}.\]
  Then for $\nu\in\mathbf{C}$, let $(\mathcal{H}_{\mathbf{n},\nu},\pi_{\mathbf{n},\nu})$ be the induced representation of $G$ from $MAN$ given by
\[\{f:G\rightarrow\mathbf{C}\big|f(gme^{tY_{n}}n)=e^{-(\nu+\rho) t}\sigma_{\mathbf{n}}(m)^{-1}f(g),\ me^{tY_{n}}n\in MAN,\ f|_{K}\in L^{2}(K)\}\]
where $K=SO(n)$ is a maximal compact subgroup of $G$,
with the group operation
\[(\pi_{\mathbf{n},\nu}(g)f)(x)=f(g^{-1}x).\]
 It is possible to show that
\begin{equation}\label{time change200}
 \begin{array}{ll}  \pi_{\mathbf{n},\nu}\text{ is unitary equivalent to }\pi_{\mathbf{n},-\nu} &, \text{ if }n=2k  \\
 \pi_{\mathbf{n},\nu}\text{ is unitary equivalent to }\pi_{\mathbf{n}_{1},-\nu}  &, \text{ if }n=2k+1  \\
 \end{array}
\end{equation}
 where $\mathbf{n}_{1}=(-n_{1},n_{2},\ldots,n_{k})$.

Note that $f$ in $\pi_{\mathbf{n},\nu}$ are invariant under $M$. Thus, $\mathcal{H}_{\mathbf{n},\nu}$ can be realized on $L^{2}(K/M)=L^{2}(S^{n-1})$.
The natural $L^{2}$-norm on $L^{2}(S^{n-1})$ can define a unitary representation for $\pi_{\nu}$ only when $\nu=it$ for $t\in\mathbf{R}$. It is tempered, and called the \textit{principal series}\index{principal series}. However, it is still possible to unitarize the representations for $\nu\in(-\rho,0)\cup(0,\rho)$ by other norms (see Theorem \ref{time change196}). They are called the \textit{complementary series}\index{complementary series} and  not tempered.

For a fixed $(\mathcal{H}_{\mathbf{n},\nu},\pi_{\mathbf{n},\nu})$, the $K$-restricted representation of $K=SO(n)$ is a direct sum of $K$-irreducible representations $\mathcal{H}_{\mathbf{m}}$. Thus,  we have
\begin{equation}\label{time change193}
 \mathcal{H}_{\mathbf{n},\nu}=\bigoplus_{\mathbf{m}} \mathcal{W}_{\mathbf{m}}
\end{equation}
where $\mathbf{m}=(m_{i})_{1\leq i\leq \lceil\frac{n-1}{2}\rceil}$ indicates the highest weight and satisfies
\[\begin{array}{ll}  |m_{1}|\leq n_{1}\leq m_{2}\leq n_{2}\leq\cdots\leq m_{k-1}\leq n_{k-1}\leq m_{k}<\infty &, \text{ if }n=2k  \\
   |n_{1}|\leq m_{1}\leq n_{2}\leq m_{2}\leq\cdots\leq m_{k-1}\leq n_{k}\leq m_{k}<\infty  &, \text{ if }n=2k+1  \\
 \end{array}.\]
 There is a standard orthonormal basis for $\mathcal{W}_{\mathbf{m}}$ (and hence for  $\mathcal{H}_{\mathbf{n},\nu}$), called the \textit{Gelfand-Tsetlin basis}\index{Gelfand–Tsetlin bases}. See \cite{cetlin1950finite}, \cite{hirai1962infinitesimal}, \cite{ramirez2013invariant} for more details. However, we do not need it here.

 In the following, we are mainly interested in the case $\sigma_{\mathbf{n}}=1$ and hence $\mathbf{n}=0$.  (It follows that $\mathbf{m}=(0,\ldots,0,m_{k})$ and so we consider $\mathbf{m}$ as an integer.) In this case, the representations are \textit{spherical}\index{spherical representation} (or \textit{class one}\index{class one representation}) and we shall denote $(\mathcal{H}_{0,\nu},\pi_{0,\nu})$   by  $(\mathcal{H}_{\nu},\pi_{\nu})$. For more information about the general cases, one may see \cite{hirai1962irreducible}, \cite{thieleker1974unitary}, and so on.

In order to make the restriction map clear, we  review some facts about spherical harmonics (see \cite{johnson1977composition}, \cite{vilenkin1978special}, also \cite{zhang2011discrete}). We identify $\mathfrak{p}$ with $\mathbf{R}^{n}$, and consider the adjoint action of $K$ on $\mathfrak{p}$. We fix a $K$-invariant inner product on $\mathfrak{p}$ so that $Y_{1},\ldots,Y_{n}$ form an orthonormal basis.  Then the homogeneous space $K/M\cong S^{n-1}$.
  Let $\hat{K}$ be the unitary dual of $K$, i.e. the set of equivalent classes of irreducible finite dimensional representations of $K$. If $(\pi_{\gamma},V_{\gamma})\in\gamma\in \hat{K}$, let $V^{M}_{\gamma}\coloneqq\{v\in V_{\gamma}:\pi(M)v=v\}$ be the space of $M$-fixed vectors. General representation theory, namely \textit{Frobenius reciprocity}\index{Frobenius reciprocity} and \textit{Peter-Weyl theorem}\index{Peter-Weyl theorem}, implies that
\[L^{2}(S^{n-1})=\bigoplus_{\gamma\in\hat{K}}n_{\gamma}V_{\gamma}\]
where $n_{\gamma}=\dim V^{M}_{\gamma}$.

However, we can explore further properties of  $n_{\gamma}$ and $V^{M}_{\gamma}$.  Let $x_{1},\ldots,x_{n}$ be the standard coordinates for $\mathfrak{p}=\mathbf{R}^{n}$. Let $r^{2}=\sum_{i=1}^{n}x^{2}_{i}$ and $\Delta_{n}=\sum_{i=1}^{n}\partial^{2} /\partial x^{2}_{i}$ be the standard \textit{Laplacian}\index{Laplacian}. Let $\mathcal{P}^{p}$ be the space of all homogeneous polynomials of degree $p$ in the variables $x_{1},\ldots,x_{n}$ and let $W^{p}=\ker\Delta_{n}|_{\mathcal{P}^{p}}$ be the \textit{spherical harmonics}\index{spherical harmonics}. Clearly, $W^{p}$ is a $K$-representation, and $W^{0}$ is the trivial representation. Besides, it is known that
\[L^{2}(S^{n-1})=\bigoplus_{p\geq0}W_{p}\]
where we are identifying elements of $W^{p}$ and their restrictions to the unit sphere $S^{n-1}\subset\mathbf{R}^{n}$.
Moreover, it is proved for $p\geq 1$ that
\[\begin{array}{ll}   W_{p}=\mathbf{C}\chi_{p}\oplus\mathbf{C}\chi_{-p} &, \text{ if }n=2  \\
   W_{p}\text{ is irreducible}  &, \text{ if }n\geq3  \\
 \end{array}\]
where $\chi_{p}$ is the character on $S^{1}$ of degree $p$. Thus, we conclude that
\begin{equation}\label{time change195}
  \mathcal{W}_{\mathbf{m}}=\left\{\begin{array}{ll}  \mathbf{C}\chi_{\mathbf{m}} &, \text{ if }n=2  \\
  W_{\mathbf{m}} &, \text{ if }n\geq 3 \\
 \end{array}\right.
\end{equation}
to align  the notation.   The subspace $(\mathcal{W}_{\mathbf{m}})^{M}$ of $M$-fixed vectors is $1$-dimensional
\[(\mathcal{W}_{\mathbf{m}})^{M}=\mathbf{C}\phi_{\mathbf{m}}\]
where $\phi_{\mathbf{m}}$  is a generator normalized by $\phi_{\mathbf{m}}(Y_{n})=1$. They depend only on the last variable $x_{n}\in S^{n-1}$ of $x=(x_{1},\ldots,x_{n})$. In the following,
we put the upper-index the dimension $n$ as we shall  treat it as a variable, such as  $\phi_{\mathbf{m}}=\phi^{n}_{\mathbf{m}}$.
\begin{lem}[Theorem 3.1 \cite{johnson1977composition}]\label{discrete components5}
   The polynomials $\phi^{n}_{\mathbf{m}}$ is given as follows:
   \[x_{n}=\cos\xi,\ \ \ \phi^{n}_{\mathbf{m}}(x_{n})\coloneqq\cos^{\mathbf{m}}\xi F(-\frac{\mathbf{m}}{2},-\frac{\mathbf{m}-1}{2},\frac{n-1}{2},-\tan^{2}\xi) \]
   where $F(a,b,c,x)$ is the \textit{Gauss hypergeometric function}\index{Gauss hypergeometric function} $\prescript{}{2}F_{1}$,
\[F(a,b,c,x)=\sum_{m=0}^{\infty}\frac{(a)_{m}(b)_{m}}{(c)_{m}}\frac{x^{m}}{m!}\]
and $(a)_{m}=\prod_{j=0}^{m-1}(a+j)$ is the \textit{Pochammer symbol}\index{Pochammer symbol}.
\end{lem}

Next, we introduce the notation on the subgroup $H=SO(n-1,1)\subset G$. In the following, we shall use superscript $\flat$ to indicate the corresponding $H$-data, as we obtained for $G$, and many of them are obtained by restriction of $H\subset G$. For example, we write
\[H=K^{\flat}AN^{\flat}\]
for the \textit{Iwasawa decomposition}\index{Iwasawa decomposition} of $H$ (note that as $H\subset G$, we may require the maximal abelian subgroups $A$ of $H$ and $G$ coincide). Besides, $K^{\flat}=K\cap H$ and $N^{\flat}=N\cap H$.
Here, for simplicity, we choose $H$ so that $Y_{1}$  is invariant under $K^{\flat}$.

Again, we are able to construct   the unitary representation of $H$. For $\nu\in\mathbf{C}$, let $(\pi^{\flat}_{\nu},\mathcal{H}^{\flat}_{\nu})$ be the induced representation of $H$ from $M^{\flat}AN^{\flat}$ given by
\[\{f:H\rightarrow\mathbf{C}\big|f(gme^{tY_{n}}n)=e^{-(\nu+\rho^{\flat}) t}f(g),\ me^{tY_{n}}n\in M^{\flat}AN^{\flat},\ f|_{K^{\flat}}\in L^{2}(K^{\flat})\}.\]
Similarly as for $G$, the complementary series of $H$ are defined for $\nu\in(-\rho^{\flat},0)\cup(0,\rho^{\flat})$. Now for $\nu\in(\rho^{\flat},\rho)$, we can define the restriction map $\Res:\mathcal{H}_{-\nu}\rightarrow\mathcal{H}^{\flat}_{\frac{1}{2}-\nu}$ by
\begin{equation}\label{time change202}
 \Res:f\mapsto f|_{H}.
\end{equation}
One important consequence is that $\Res$ is $H$-equivariant, i.e.
\begin{equation}\label{time change201}
 \Res(\pi_{-\nu}(h)f)=\pi^{\flat}_{\frac{1}{2}-\nu}(h)\Res(f)
\end{equation}
 for all $h\in H$ and $f\in \mathcal{H}_{-\nu}$. When we realize them as elements in $L^{2}$, then the restriction map (\ref{time change202}) becomes $\Res:L^{2}(K/M)\rightarrow L^{2}(K^{\flat}/M^{\flat})$ by
\[\Res: f\mapsto f|_{Y_{1}=0}.\]

It is known  that
\begin{equation}\label{time change197}
  L^{2}(K^{\flat}/M^{\flat})=L^{2}(S^{n-2})=\bigoplus_{l} \mathcal{V}_{l}
\end{equation}
where $\mathcal{V}_{l}$ is the space of harmonic polynomials in $n-1$ variables of degree $\mathbf{m}$ defined in (\ref{time change195}).
Then we have
\begin{lem}[Lemma 3.3 \cite{zhang2011discrete}]\label{time change204}
    The branching of $ \mathcal{W}_{\mathbf{m}}$ and $\Res(\mathcal{W}_{\mathbf{m}})$ under $K^{\flat}$ is given by
    \begin{equation}\label{time change198}
      \mathcal{W}_{\mathbf{m}}=\bigoplus_{|l|\leq \mathbf{m}}\widetilde{\mathcal{V}}_{l},\ \ \ \Res(\mathcal{W}_{\mathbf{m}})= \bigoplus_{\substack{|l|\leq \mathbf{m}\\
\mathbf{m}-l\text{ even}}}\mathcal{V}_{l}
    \end{equation}
where $\widetilde{\mathcal{V}}_{l}\subset L^{2}(S^{n-1})$ denotes the $K^{\flat}$-irreducible representation of highest weight $l$ in $L^{2}(K/M)$. Further,  the isomorphism $\mathcal{V}_{l}\rightarrow \widetilde{\mathcal{V}}_{l}$ is given by
\[h(x_{2},\ldots,x_{n})\mapsto h(x_{2},\ldots,x_{n})\phi^{n+2s}_{p-s}(x_{1})\]
where $\phi$ is given in (\ref{discrete components5}).
\end{lem}
\subsection{Casimir  and Laplace operators}\label{time change194}
In this section, we review the Casimir operators and Laplace-Beltrami operators on $SO(n,1)$. See \cite{ramirez2013invariant} and the references therein. The \textit{Casimir operator}\index{Casimir operator} for $SO(n,1)$ is
\[\Box_{n}\coloneqq-\sum_{k=1}^{n}Y^{2}_{k}+\sum_{1\leq i<j\leq n}\Theta_{ij}^{2}.\]
It is in the center of the universal enveloping algebra of $\mathfrak{g}$, and therefore acts as a scalar $c_{n}(\mathbf{n},\nu)$ in any irreducible unitary representation $\mathcal{H}_{\mathbf{n},\nu}$. By \cite{thieleker1974unitary} Theorem 3 (or \cite{thieleker1973quasi} Lemma 6), we know that
\begin{align}
 c_{n}(\mathbf{n},\nu)=&  \rho_{n}^{2}-\nu^{2}-\langle\mathbf{n},\mathbf{n}+2\rho_{M_{n}}\rangle\;\nonumber\\
           =&\left\{\begin{array}{ll} \rho_{n}^{2}-\nu^{2}- \sum_{i=1}^{k-1}n_{i}(n_{i}+2i-1) &, \text{ if }n=2k  \\
   \rho_{n}^{2}-\nu^{2}- \sum_{i=1}^{k}n_{i}(n_{i}+2i-2) &, \text{ if }n=2k+1  \\
 \end{array}\right.\ \;  \label{time change102}
\end{align}
where $\rho_{n}$ and $\rho_{M_{n}}$ are the half-sum of positive roots of $SO(n,1)$ and $M_{n}=SO(n-1)$ respectively. Similarly, the \textit{Casimir operator}\index{Casimir operator} of $K_{n}=SO(n)$ is given by
\[\Box_{K_{n}}\coloneqq \sum_{1\leq i<j\leq n}\Theta_{ij}^{2}.\]
It again acts as a scalar in any irreducible unitary representation $\mathcal{H}_{\mathbf{m}}$. As $\mathbf{m}$ indicates the highest weight of $\mathcal{H}_{\mathbf{m}}$, we conclude from the standard representation theory (e.g. \cite{humphreys2012introduction} Section 23) that the scalar is
\begin{align}
 c_{K_{n}}(\mathbf{m})=& \langle \mathbf{m},\mathbf{m}+2\rho_{K_{n}}\rangle\;\nonumber\\
           =&\left\{\begin{array}{ll} - \sum_{i=1}^{k}m_{i}(m_{i}+2i-2) &, \text{ if }n=2k  \\
  - \sum_{i=1}^{k}m_{i}(m_{i}+2i-1) &, \text{ if }n=2k+1  \\
 \end{array}\right.\ \;  \label{time change103}
\end{align}
where $\rho_{K_{n}}$ is the half sum of the positive roots of $K_{n}$.

Note that now the \textit{Laplace-Beltrami operator}\index{Laplace-Beltrami operator}  $\Delta$ is then defined by
\[\Delta\coloneqq\Box_{n}-2\Box_{K_{n}}.\]
Since then $\Delta$ commute with $K_{n}$, we can define the Laplace-Beltrami operator and Sobolev norms on $K\backslash G/\Gamma$, after making a standard identification between $L^{2}(K\backslash G/\Gamma)$ and the subspace $L^{2}( G/\Gamma)^{K}$ of $K$-invariant elements of $L^{2}(G/\Gamma)$.

Also, recall that a \textit{spherical representation}\index{spherical representation} of $G=SO(n,1)$ (e.g. \cite{corlette1990hausdorff} Section 4) is a   representation   which contains a nontrivial $K$-fixed vector. Now define the spherical part $L^{2}(G/\Gamma)^{\sph}$ to be the minimal subrepresentation containing the $K$-fixed part $L^{2}(G/\Gamma)^{K}$. Then the spherical part $L^{2}(G/\Gamma)^{\sph}$ decomposes discretely or continuously into irreducible spherical unitary representations of $G$:
\[L^{2}(G/\Gamma)^{\sph}=\int\pi_{\lambda}d\mu(\lambda).\]
 Harish-Chandra (e.g. \cite{shalom2000rigidity}) has shown that the spherical  representation $\pi_{\lambda}$ occurs in the decomposition, correspond to the $L^{2}$-spectrum of the Laplacian $\Delta$ acting on the locally symmetric space $K\backslash G/\Gamma$. In particular,   the complementary series $\pi_{\nu}$ ($0<\nu<\rho_{n}$) lies in the support of $\mu$ iff $\rho_{n}^{2}-\nu^{2}$ lies in the spectrum of $\Delta$. In other words, the complementary series occurs iff the spectrum satisfies
 \[\Spec(\Delta)\cap(0,\rho_{n}^{2})\neq\emptyset.\]
 It is easy to see from many points of view that the smallest nonzero eigenvalue in $\Spec(\Delta)$ can be made arbitrarily small, even for cocompact lattice $\Gamma$. For instance, by \cite{millson1976first}, there exists a hyperbolic manifold $X$ with positive first Betti number. Then let $X^{k}$ be the cyclic covering of degree $k$ induced by a fixed surjective homomorphism
 \[\pi_{1}(X)\xrightarrow{\varphi}\mathbf{Z}\rightarrow\mathbf{Z}/k\]
 where $\varphi$ is independent of $k$. Then there are constants $c_{1}(X)$ and $c_{2}(X)$ such that the smallest nonzero eigenvalue $\lambda(X)$ in $\Spec(\Delta_{X})$ satisfies
 \[ c_{1}(X) k^{-2}\leq\lambda(X^{k})\leq c_{2}(X) k^{-2}\]
 as what we wanted.    See \cite{randol1974small}, \cite{schoen1980geometric}, \cite{brooks1988injectivity} for more details. Thus, we take it for granted that there exist cocompact lattices $\Gamma$ for which $L^{2}(G/\Gamma)$ contains complementary series with spectral parameter  $\nu\in(\rho_{n-1},\rho_{n})$ as a direct summand.

 Although not needed in our proof, it is  worth mentioning other results for the study of $\Spec(\Delta)$. For example,  Lax and Phillips  have shown that for \textit{geometrically finite}\index{geometrically finite} discrete subgroup $\Gamma$, the spectrum $\Spec(\Delta)$ of $\Delta$ on   $\mathbf{H}^{n}/\Gamma$ has at most finitely many $L^{2}$-eigenvalues in the interval  $[0,\rho_{n}^{2})$ \cite{lax1982asymptotic} and purely absolutely continuous spectrum of infinite multiplicity in $[\rho_{n}^{2},\infty)$ \cite{lax1984translation}. On the other hand, let $\widehat{G}^{\sph}$ be the spherical unitary dual of $G=SO(n,1)$, that is
 \[\widehat{G}^{\sph}=\{\pi_{\lambda}\bmod \pm1:\lambda\in i\mathbf{R}\cup[-\rho_{n},\rho_{n}]\}.\]
 Then let $\widehat{G}^{\sph}_{\Aut}$ be its automorphic dual, consisting of all $\pi_{\lambda}$ which occur in $L^{2}(G/\Gamma)$ where $\Gamma$ varies over all congruence subgroups of $G(\mathbf{Z})$.
 We have the following \textit{generalized Ramanujan conjecture}\index{generalized Ramanujan conjecture}
for $G$.
 \begin{conj}[Generalized Ramanujan conjecture] Let $G=SO(n,1)$. Then
 \[\widehat{G}^{\sph}_{\Aut}=i\mathbf{R}\cup\{\rho_{n},\rho_{n}-1,\ldots,\rho_{n}-\lfloor\rho_{n}\rfloor\}.\]
 \end{conj}
 For $n=2$, it reduces to the \textit{Selberg's $1/4$ conjecture}\index{Selberg's 1/4 conjecture}. See \cite{sarnak2005notes} for more details.
\subsection{Hilbert and Sobolev structures}\label{time change117}
As mentioned in Section \ref{time change118}, for $\nu\in(-\rho,0)\cup(0,\rho)$, one may define a spherical complementary series $(\mathcal{H}_{\nu},\pi_{\nu})$. Besides, the elements $\mathcal{H}_{\nu}$ can be realized on $L^{2}(K/M)=L^{2}(S^{n-1})$. Then the norms $\|\cdot\|_{\mathcal{H}_{\nu}}$ can be obtained by
\begin{thm}[Theorem 6.2 \cite{johnson1977composition}, \cite{kostant1969existence}]\label{time change196} For $\nu\in(-\rho,\rho)$, $w=\sum_{\mathbf{m}}w_{\mathbf{m}}\in L^{2}(S^{n-1})=\bigoplus_{\mathbf{m}}\mathcal{W}_{\mathbf{m}}$, the norm $\|\cdot\|_{\pi_{\nu}}$ on $\mathcal{H}_{\nu}$ is given by
\begin{equation}\label{time change229}
  \|w\|^{2}_{\mathcal{H}_{\nu}}=\sum_{\mathbf{m}}d_{\mathbf{m}}(\nu)\|w_{\mathbf{m}}\|^{2}
\end{equation}
  where $\|w_{\mathbf{m}}\|^{2}$ is the $L^{2}$-norm, and
  \[d_{\mathbf{m}}(-\nu)=\frac{(\rho+\nu)_{\mathbf{m}}}{(\rho-\nu)_{\mathbf{m}}}=\frac{\Gamma(\rho+\nu+\mathbf{m})}{\Gamma(\rho+\nu)\Gamma(\rho-\nu+\mathbf{m})}.\]
\end{thm}
\begin{rem}
  Via \textit{Stirling’s formula}, we can estimate
  \begin{equation}\label{time change203}
    d_{\mathbf{m}}(-\nu)\asymp_{n,\nu}(1+\mathbf{m})^{2\nu}.
  \end{equation}
where $A\asymp B$ means there is a constant $C>0$ such that $C^{-1}B\leq A\leq CB$.
On the other hand, the norm clearly indicates that $\langle w_{\mathbf{m}_{1}},w_{\mathbf{m}_{2}}\rangle_{\mathcal{H}_{-\nu}}=0$ for $w_{\mathbf{m}_{1}}\in\mathcal{W}_{\mathbf{m}_{1}}$ and $w_{\mathbf{m}_{2}}\in\mathcal{W}_{\mathbf{m}_{2}}$. Thus  we still have the orthogonal decomposition
(\ref{time change193}):
\[\mathcal{H}_{-\nu}=\bigoplus_{\mathbf{m}} \mathcal{W}_{\mathbf{m}}.\]
\end{rem}

Having been introduced the norm (Hilbert structure) $\|\cdot\|_{\mathcal{H}_{\nu}}$ on $\mathcal{H}_{\nu}$, we can then discuss the Sobolev structure on it.
Let $\mathcal{H}$ be a unitary representation of $G$. As in \cite{flaminio2003invariant} and other related results, the Laplace-Beltrami operator $\Delta_{G}$ gives unitary representation spaces a   Sobolev structure. The \textit{Sobolev space of order $s\geq 0$}\index{Sobolev space} is the Hilbert space $ W_{G}^{s}(\mathcal{H})\subset  \mathcal{H}$ that is the maximal domain given by the inner product
\[\langle f,g \rangle_{W_{G}^{s}(\mathcal{H})}\coloneqq\langle (1+\Delta_{G})^{s}f,g\rangle\]
for $f,g\in\mathcal{H}$. Besides, the space of smooth vectors is given by
\[C^{\infty}(\mathcal{H})=\bigcap_{s\geq0}W_{G}^{s}(\mathcal{H}).\]
Denote by $\mathcal{E}^{\prime}(\mathcal{H})\coloneqq (C^{\infty}(\mathcal{H}))^{\prime}$ its distributional dual.
 Note that, when $\mathcal{H}=L^{2}(G/\Gamma)$, $W_{G}^{s}(G/\Gamma)\coloneqq W_{G}^{s}(L^{2}(G/\Gamma))$ coincides with the natural Sobolev structure on $G/\Gamma$ and hence $C^{\infty}(G/\Gamma)$ is the space of infinite differentiable functions on $G/\Gamma$.
On the other hand, for $s>0$, the distributional dual of $W_{G}^{s}(\mathcal{H})$ is the Sobolev space $W_{G}^{-s}(\mathcal{H})=(W_{G}^{s}(\mathcal{H}))^{\prime}\subset\mathcal{E}^{\prime}(\mathcal{H})$.

For an irreducible unitary representation $\mathcal{H}_{\mathbf{n},\nu}$ of $G=SO(n,1)$, the Sobolev inner product can be computed via (\ref{time change102}), (\ref{time change103}): for $f= \sum_{\mathbf{m}}f_{\mathbf{m}},g= \sum_{\mathbf{m}}g_{\mathbf{m}}\in W^{s}_{G}(\mathcal{H}_{\mathbf{n},\nu})$, we have
\begin{align}
 \langle f,g\rangle_{W_{G}^{s}(\mathcal{H}_{\mathbf{n},\nu})} =&  \langle (I+\Delta_{G})^{s}f,g\rangle_{\mathcal{H}_{\mathbf{n},\nu}}\;\nonumber\\
  =&  \sum_{\mathbf{m}}\langle (I+\Delta_{G})^{s} f_{\mathbf{m}},g\rangle_{\mathcal{H}_{\mathbf{n},\nu}}\;\nonumber\\
  =&  \sum_{\mathbf{m}}(1+c_{n}(\mathbf{n},\nu)+c_{K_{n}}(\mathbf{m}))^{s} \langle  f_{\mathbf{m}},g_{\mathbf{m}}\rangle_{\mathcal{H}_{\mathbf{n},\nu}}. \;  \label{time change107}
\end{align}
\begin{rem}
  Again, (\ref{time change107}) indicates that $\langle w_{\mathbf{m}_{1}},w_{\mathbf{m}_{2}}\rangle_{W_{G}^{s}(\mathcal{H}_{\mathbf{n},\nu})} =0$ for $w_{\mathbf{m}_{1}}\in\mathcal{W}_{\mathbf{m}_{1}}$ and $w_{\mathbf{m}_{2}}\in\mathcal{W}_{\mathbf{m}_{2}}$. Thus, we still have the orthogonal decomposition
(cf. (\ref{time change193})):
\[W_{G}^{s}(\mathcal{H}_{\mathbf{n},\nu})=\bigoplus_{\mathbf{m}} \mathcal{W}_{\mathbf{m}}.\]
\end{rem}
It is easy to estimate the coefficients
\begin{lem}\label{time change108} Let the notation and assumptions be as above. Then
\[1+c_{n}(\mathbf{n},\nu)+c_{K_{n}}(\mathbf{m})\asymp_{\mathbf{n},\nu} 1+\|\mathbf{m}\|_{\infty}^{2}\]
where $\|\mathbf{m}\|_{\infty}$ is the maximal number of $\mathbf{m}=(m_{1},\ldots,m_{k})$.
\end{lem}
\subsection{Norms estimate}
In this section, we shall show that certain $G$-complementary series contains a $H$-complementary series as a discrete component. The result for $G=SO(3,1)$, $H=SO(2,1)$ has already been shown by \cite{mukunda1968unitary}. Here we adopt the  method as in \cite{zhang2011discrete} (or \cite{speh2012restriction}), and make a slight generalization. More precisely, the idea in  \cite{zhang2011discrete} is to estimate the operator norm of the projection with respect to the norms on Hilbert spaces. In the following, instead of thinking about the Hilbert norm, we make a more precise estimate for the Sobolev norm. We include the proofs to keep the paper as self-contained as possible.
\begin{thm}\label{time change116}
   Let  $n\geq 3$, $\rho^{\flat}<\nu<\rho$, $s\geq0$, $G=SO(n,1)$ and $H=SO(n-1,1)$.  Then  $(\pi^{\flat}_{\nu-\frac{1}{2}},W_{H}^{s}(\mathcal{H}_{\nu-\frac{1}{2}}^{\flat}))$ is a direct summand of $(\pi_{\nu},W_{G}^{s}(\mathcal{H}_{\nu}))$ restricted to $H$.
\end{thm}
In the following,   we   replace $\pi_{\nu}$ and $\pi_{\nu-\frac{1}{2}}^{\flat}$ by the unitarily equivalent representations $\pi_{-\nu}$ and $\pi_{\frac{1}{2}-\nu}^{\flat}$  via (\ref{time change200}) (or Section 6 \cite{johnson1977composition}), for the sake of introducing the restriction map.
First of all, we estimate the operator norm of the restriction map $\Res:W^{s}_{G}(\mathcal{H}_{-\nu})\rightarrow W^{s}_{H}(\mathcal{H}^{\flat}_{\frac{1}{2}-\nu})$ for $\nu\in(\rho^{\flat},\rho)$. By (\ref{time change193}),  (\ref{time change197}), (\ref{time change107}), we have
\begin{equation}\label{time change199}
  W^{s}_{G}(\mathcal{H}_{-\nu})=\bigoplus_{\mathbf{m}} \mathcal{W}_{\mathbf{m}} ,\ \ \ W^{s}_{H}(\mathcal{H}^{\flat}_{\frac{1}{2}-\nu})=\bigoplus_{l} \mathcal{V}_{l}.
\end{equation}
Via (\ref{time change198}), for $|l|\leq \mathbf{m}$, we consider the orthogonal projections
\[P_{\mathbf{m},l}:\mathcal{W}_{\mathbf{m}}\rightarrow \widetilde{\mathcal{V}}_{l},\ \ \ \overline{P}_{\mathbf{m},l}:\Res(\mathcal{W}_{\mathbf{m}})\rightarrow  \mathcal{V}_{l},\ \ \ \Res_{\mathbf{m},l}\coloneqq \overline{P}_{\mathbf{m},l}\Res:\mathcal{W}_{\mathbf{m}}\rightarrow \mathcal{V}_{l}.\]
Then $\Res=\sum_{\mathbf{m}}\sum_{|l|\leq \mathbf{m}}\Res_{\mathbf{m},l}$. Using  the orthogonality (\ref{time change199}), we can deduce an estimate for the operator norms via an elementary argument:
\begin{lem}\label{time change207} The operator norm  $\|\cdot\|_{\op}$ of $\Res:W^{s}_{G}(\mathcal{H}_{-\nu})\rightarrow W^{s}_{H}(\mathcal{H}^{\flat}_{\frac{1}{2}-\nu})$ satisfies
\[ \|\Res\|_{\op}^{2}=\sup_{l}\sum_{\mathbf{m}\geq|l|}\|\Res_{\mathbf{m},l}\|_{\op}^{2}.\]
\end{lem}

\begin{proof}
  This is exactly Lemma 3.2 \cite{zhang2011discrete}. Recall Lemma \ref{time change204} and fix arbitrarily $w=\sum_{\mathbf{m}}\sum_{\mathbf{m}\geq|l|}P_{\mathbf{m},l}w_{\mathbf{m}}\in W^{s}_{G}(\mathcal{H}_{-\nu})$.
  Then by \textit{Cauchy-Schwarz inequality}, we have
  \begin{align}
 \|\Res w\|^{2}_{W^{s}_{H}(\mathcal{H}^{\flat}_{\frac{1}{2}-\nu})} =&   \sum_{l}\|\sum_{\mathbf{m}\geq|l|}\Res_{\mathbf{m},l} P_{\mathbf{m},l} w\|^{2}_{W^{s}_{H}(\mathcal{H}^{\flat}_{\frac{1}{2}-\nu})} \;\nonumber\\
  \leq&  \sum_{l}\left(\sum_{\mathbf{m}\geq|l|}\|\Res_{\mathbf{m},l}\|_{\op}\|P_{\mathbf{m},l} w\|_{W^{s}_{G}(\mathcal{H}_{-\nu})} \right)^{2}\;\nonumber\\
  \leq&  \sum_{l}\left(\sum_{\mathbf{m}\geq|l|}\|\Res_{\mathbf{m},l}\|_{\op}^{2}\right)\left(\sum_{\mathbf{m}\geq|l|}\|P_{\mathbf{m},l} w\|_{W^{s}_{G}(\mathcal{H}_{-\nu})}^{2} \right)\;\nonumber\\
  =& \left(\sup_{l}\sum_{\mathbf{m}\geq|l|}\|\Res_{\mathbf{m},l}\|_{\op}^{2}\right)\cdot\| w\|_{W^{s}_{G}(\mathcal{H}_{-\nu})}^{2}. \;  \nonumber
\end{align}
On the other hand, since $\Res$ is $H$- (or $K^{\flat}$-) equivariant, so is $\Res^{\ast}$. But then each $\Res_{\mathbf{m},l}^{\ast}$ is a scalar constant of an isometry operator by \textit{Schur's lemma}. Thus, for any $v\in\mathcal{V}_{l}$, we have
\begin{multline}
  \| \Res^{\ast} v\|^{2}_{W^{s}_{G}(\mathcal{H}_{-\nu})} =   \|\sum_{|l|\leq\mathbf{m}} \Res_{\mathbf{m},l}^{\ast} v\|^{2}_{W^{s}_{G}(\mathcal{H}_{-\nu})}
  = \sum_{|l|\leq\mathbf{m}} \|\Res_{\mathbf{m},l}^{\ast} v\|^{2}_{W^{s}_{G}(\mathcal{H}_{-\nu})} \\
=\sum_{|l|\leq\mathbf{m}} \|\Res_{\mathbf{m},l}^{\ast} \|^{2}_{\op}\|v\|^{2}_{W^{s}_{H}(\mathcal{H}^{\flat}_{\frac{1}{2}-\nu})}
  =  \left(\sum_{|l|\leq\mathbf{m}} \|\Res_{\mathbf{m},l} \|^{2}_{\op}\right)\cdot\|v\|^{2}_{W^{s}_{H}(\mathcal{H}^{\flat}_{\frac{1}{2}-\nu})}.\nonumber
\end{multline}
The consequence follows.
\end{proof}

Thus, we want to estimate the operator norm of $\Res_{\mathbf{m},l}$. With the help of the harmonic analysis, we can obtain the operator norm in $L^{2}$-sense.
\begin{lem}[Proposition 3.4 \cite{zhang2011discrete}]
   Let the notation and assumptions be as above. Then for $\mathbf{m}-l$ even, the $(L^{2}(S^{n-1}),L^{2}(S^{n-2}))$-norm of $\Res_{\mathbf{m},l}:\mathcal{W}_{\mathbf{m}}\rightarrow\mathcal{V}_{l}$ is given by
   \[\|\Res_{\mathbf{m},l}\|_{L^{2}}^{2}=\frac{(2\mathbf{m}+n-2)\Gamma(\frac{n}{2})\Gamma(\frac{n+\mathbf{m}+l-2}{2})\Gamma(\frac{\mathbf{m}-l+1}{2})}{\Gamma(\frac{n-1}{2})\Gamma(\frac{1}{2})\Gamma(\frac{\mathbf{m}-l+2}{2})\Gamma(\frac{n+\mathbf{m}+l-1}{2})}\asymp_{n,\nu}\frac{\mathbf{m}+1}{(\mathbf{m}+l+1)^{\frac{1}{2}}(\mathbf{m}-l+1)^{\frac{1}{2}}}.\]
\end{lem}

Then by Theorem \ref{time change196} (and (\ref{time change203})), the $(\mathcal{H}_{-\nu},\mathcal{H}^{\flat}_{\frac{1}{2}-\nu})$-norm of $\Res_{\mathbf{m},l}:\mathcal{W}_{\mathbf{m}}\rightarrow\mathcal{V}_{l}$ is given by
\[ \|\Res_{\mathbf{m},l}\|_{\mathcal{H}}^{2}=\frac{d_{l}^{\flat}(\frac{1}{2}-\nu)}{d_{\mathbf{m}}(-\nu)}\|\Res_{\mathbf{m},l}\|_{L^{2}}^{2}\asymp_{n,\nu}\frac{(1+l)^{2\nu-1}}{(1+\mathbf{m})^{2\nu}}\|\Res_{\mathbf{m},l}\|_{L^{2}}^{2}\]
where $d_{l}^{\flat}(\frac{1}{2}-\nu)=\frac{(\rho-\frac{1}{2}+\nu)_{l}}{(\rho+\frac{1}{2}-\nu)_{l}}$ denotes the coefficients in (\ref{time change229}) for $\mathcal{H}^{\flat}_{\frac{1}{2}-\nu}$.
Finally, by (\ref{time change107}) (and Lemma \ref{time change108}),  the $(W_{G}^{s}(\mathcal{H}_{-\nu}),W_{H}^{s}(\mathcal{H}^{\flat}_{\frac{1}{2}-\nu}))$-norm of $\Res_{\mathbf{m},l}:\mathcal{W}_{\mathbf{m}}\rightarrow\mathcal{V}_{l}$ is given by
\[ \|\Res_{\mathbf{m},l}\|_{\op}^{2}\asymp_{n,\nu} \frac{(1+l)^{2s}}{(1+\mathbf{m})^{2s}}\|\Res_{\mathbf{m},l}\|_{\mathcal{H}}^{2}.\]
 Thus, we conclude
 \begin{prop}\label{time change206}
    There is a constant $C=C(n,\nu,s)>1$ such that
    \[C^{-1}\leq\inf_{l}\sum_{\mathbf{m}\geq|l|}\|\Res_{\mathbf{m},l}\|_{\op}^{2}\leq\sup_{l}\sum_{\mathbf{m}\geq|l|}\|\Res_{\mathbf{m},l}\|_{\op}^{2}\leq C\]
 \end{prop}
 \begin{proof}
    By Lemma \ref{time change204} and the above estimates, for $\nu\in(\rho^{\flat},\rho)$, $2k=\mathbf{m}-l$, we have
\begin{align}
\sum_{\mathbf{m}\geq|l|}\|\Res_{\mathbf{m},l}\|_{\op}^{2}=&  \sum_{  \substack{\mathbf{m}\geq|l|\\
\mathbf{m}-l\text{ even}}}\|\Res_{\mathbf{m},l}\|_{\op}^{2} \;\nonumber\\
  \asymp&_{n,\nu}    \sum_{  \substack{\mathbf{m}\geq|l|\\
\mathbf{m}-l\text{ even}}}\frac{(1+l)^{2s}}{(1+\mathbf{m})^{2s}}\|\Res_{\mathbf{m},l}\|_{\mathcal{H}}^{2}\;\nonumber\\
\asymp&_{n,\nu}    \sum_{  \substack{\mathbf{m}\geq|l|\\
\mathbf{m}-l\text{ even}}}\frac{(1+l)^{2\nu+2s-1}}{(1+\mathbf{m})^{2\nu+2s}}\|\Res_{\mathbf{m},l}\|_{L^{2}}^{2}\;\nonumber\\
\asymp&_{n,\nu}    \sum_{  \substack{\mathbf{m}\geq|l|\\
\mathbf{m}-l\text{ even}}}\frac{(1+l)^{2\nu+2s-1}}{(1+\mathbf{m})^{2\nu+2s}}\frac{\mathbf{m}+1}{(\mathbf{m}+l+1)^{\frac{1}{2}}(\mathbf{m}-l+1)^{\frac{1}{2}}}\;\nonumber\\
  =& \sum_{ k\geq0}\frac{(1+l)^{2\nu+2s-1}}{(1+l+2k)^{2\nu+2s}}\frac{1+l+2k}{(1+2l+2k)^{\frac{1}{2}}(1+2k)^{\frac{1}{2}}}. \;  \label{time change205}
\end{align}
  It remains to show that (\ref{time change205}) is controlled by   constants independent of $l\geq0$. This can be done by the standard integral test. More precisely, the series is controlled by the first term
  \[\frac{(1+l)^{2\nu+2s-1}}{(1+l)^{2\nu+2s}}\frac{1+l}{(1+2l)^{\frac{1}{2}}}=\frac{1}{\sqrt{2l+1}}\]
  and the integral
  \[\int_{0}^{\infty}\frac{(1+l)^{2\nu+2s-1}}{(1+l+2k)^{2\nu+2s}}\frac{1+l+2k}{(1+2l+2k)^{\frac{1}{2}}(1+2k)^{\frac{1}{2}}}dk.\]
 The first term is bounded above by a constant independent of $l$. For the integral, we change the variable from $k$ to $xl$ and obtain
 \begin{align}
  &\overset{k=xl}{=\joinrel=}    \int_{0}^{\infty}\frac{(1+l)^{2\nu+2s-1}}{(1+l+2lx)^{2\nu+2s}}\frac{1+l+2lx}{(1+2l+2lx)^{\frac{1}{2}}(1+2lx)^{\frac{1}{2}}}ldx\;\nonumber\\
 &\asymp_{n,\nu,s}    \int_{0}^{\infty}\frac{1}{(1+2x)^{2\nu+2s-1}} \frac{1}{\sqrt{(2+2x)\cdot 2x}}dx.  \;  \nonumber
\end{align}
Since $\nu>\rho^{\flat}\geq \frac{1}{2}$, the latter integral is finite (and independent of $l$). This is already enough for Proposition \ref{time change206}.
 \end{proof}

 Finally, using  the \textit{open mapping theorem}\index{open mapping theorem}, we  are able to prove Theorem \ref{time change116}:

 \begin{proof}[Proof of Theorem \ref{time change116}]
   By (\ref{time change200}), we can replace $\pi_{\nu}$ and $\pi_{\nu-\frac{1}{2}}^{\flat}$ by the unitarily equivalent representations $\pi_{-\nu}$ and $\pi_{\frac{1}{2}-\nu}^{\flat}$. Since $\sup_{l}\sum_{\mathbf{m}\geq|l|}\|\Res_{\mathbf{m},l}\|_{\op}^{2}$ is bounded above and equal to the operator norm of $\Res$ by Lemma \ref{time change207}, we conclude that the restriction map $\Res:W^{s}_{G}(\mathcal{H}_{-\nu})\rightarrow W^{s}_{H}(\mathcal{H}_{\frac{1}{2}-\nu}^{\flat})$ is continuous.

   On the other hand, since $\inf_{l}\sum_{\mathbf{m}\geq|l|}\|\Res_{\mathbf{m},l}\|_{\op}^{2}$ is bounded below, we  may then deduce that $\Res:W^{s}_{G}(\mathcal{H}_{-\nu})\rightarrow W^{s}_{H}(\mathcal{H}_{\frac{1}{2}-\nu}^{\flat})$ is surjective. More precisely, assume that $v=\sum_{l}v_{l}\in W^{s}_{H}(\mathcal{H}_{\frac{1}{2}-\nu}^{\flat})$. Then fix $l$, and we conclude from Lemma \ref{time change204} that $\Res_{\mathbf{m},l}|_{\widetilde{\mathcal{V}}_{l}}$ are isomorphisms for all $\mathbf{m}\geq |l|$ and $\mathbf{m}-l$ even. For simplicity, the sums $\sum$ are all over $\mathbf{m}$ with $\mathbf{m}\geq |l|$ and $\mathbf{m}-l$ even. Let $u_{l}=\sum c^{\mathbf{m}}_{l}u^{\mathbf{m}}_{l}$ where
   \[u^{\mathbf{m}}_{l}\coloneqq (\Res_{\mathbf{m},l}|_{\widetilde{\mathcal{V}}_{l}})^{-1}v_{l},\ \ \ c^{\mathbf{m}}_{l}\coloneqq\frac{\|v_{l}\|_{W^{s}_{H}(\mathcal{H}_{\frac{1}{2}-\nu}^{\flat})}}{\sum\|\Res_{\mathbf{m},l}\|_{\op}^{2}}\cdot\frac{\|\Res_{\mathbf{m},l}\|_{\op}}{\|u^{\mathbf{m}}_{l}\|_{W^{s}_{G}(\mathcal{H}_{-\nu})}}.\]
   Then
   \begin{equation}\label{time change208}
    \Res u_{l}=\sum \Res_{\mathbf{m},l} c^{\mathbf{m}}_{l}u^{\mathbf{m}}_{l}=\sum  c^{\mathbf{m}}_{l}v_{l}.
   \end{equation}
Since $\Res_{\mathbf{m},l}|_{\widetilde{\mathcal{V}}_{l}}$ is a $K^{\flat}$-equivariant isomorphism, by \textit{Schur's lemma}, we know that $\Res_{\mathbf{m},l}|_{\widetilde{\mathcal{V}}_{l}}$ is a scalar constant of an isometry operator. Then we get
\begin{align}
\|\Res u_{l}\|_{W^{s}_{H}(\mathcal{H}_{\frac{1}{2}-\nu}^{\flat})}=&  \sum\| \Res_{\mathbf{m},l} c^{\mathbf{m}}_{l}u^{\mathbf{m}}_{l}\|_{W^{s}_{H}(\mathcal{H}_{\frac{1}{2}-\nu}^{\flat})}  \;\nonumber\\
=&  \sum\| \Res_{\mathbf{m},l}\|_{\op} \|c^{\mathbf{m}}_{l}u^{\mathbf{m}}_{l}\|_{W^{s}_{G}(\mathcal{H}_{-\nu})} \;\nonumber\\
  =& \sum\| \Res_{\mathbf{m},l}\|_{\op}\cdot \frac{\|v_{l}\|_{W^{s}_{H}(\mathcal{H}_{\frac{1}{2}-\nu}^{\flat})}}{\sum\|\Res_{\mathbf{m},l}\|_{\op}^{2}}\cdot\frac{\|\Res_{\mathbf{m},l}\|_{\op}}{\|u^{\mathbf{m}}_{l}\|_{W^{s}_{G}(\mathcal{H}_{-\nu})}}\cdot \|u^{\mathbf{m}}_{l}\|_{W^{s}_{G}(\mathcal{H}_{-\nu})} \;\nonumber\\
  =&  \|v_{l}\|_{W^{s}_{H}(\mathcal{H}_{\frac{1}{2}-\nu}^{\flat})} .\;  \label{time change209}
\end{align}
 Thus, $\Res u_{l}= v_{l}$ by (\ref{time change208}) and (\ref{time change209}). On the other hand, one has
\begin{align}
\|\Res u_{l}\|^{2}_{W^{s}_{H}(\mathcal{H}_{\frac{1}{2}-\nu}^{\flat})} =&  \left(\sum\| \Res_{\mathbf{m},l}\|_{\op} \|c^{\mathbf{m}}_{l}u^{\mathbf{m}}_{l}\|_{W^{s}_{G}(\mathcal{H}_{-\nu})}\right)^{2} \;\nonumber\\
  =&\sum\| \Res_{\mathbf{m},l}\|_{\op}^{2} \sum\|c^{\mathbf{m}}_{l}u^{\mathbf{m}}_{l}\|_{W^{s}_{G}(\mathcal{H}_{-\nu})}^{2}\;\nonumber\\
  =&\sum\| \Res_{\mathbf{m},l}\|_{\op}^{2}  \| u_{l} \|_{W^{s}_{G}(\mathcal{H}_{-\nu})}^{2}\;\nonumber\\
  \geq& \inf_{l}\left(\sum\| \Res_{\mathbf{m},l}\|_{\op}^{2} \right) \| u_{l} \|_{W^{s}_{G}(\mathcal{H}_{-\nu})}^{2}.\; \label{time change210}
\end{align}
Combining  (\ref{time change210}) with (\ref{time change209}), we get
\begin{align}
 \|v\|^{2}_{W^{s}_{H}(\mathcal{H}_{\frac{1}{2}-\nu}^{\flat})} =\sum_{l}v_{l}\|v_{l}\|^{2}_{W^{s}_{H}(\mathcal{H}_{\frac{1}{2}-\nu}^{\flat})} =& \sum_{l}\|\Res u_{l}\|^{2}_{W^{s}_{H}(\mathcal{H}_{\frac{1}{2}-\nu}^{\flat})} \;\nonumber\\
  \geq& \inf_{l}\left(\sum\| \Res_{\mathbf{m},l}\|_{\op}^{2} \right) \sum_{l} \| u_{l} \|_{W^{s}_{G}(\mathcal{H}_{-\nu})}^{2}.\;\nonumber
\end{align}
Thus, $u\coloneqq\sum_{l}   u_{l}\in W^{s}_{G}(\mathcal{H}_{-\nu})$ is well defined and satisfies $\Res u=v$, which proves the surjectivity.

    Now the orthogonal decomposition implies $W^{s}_{G}(\mathcal{H}_{-\nu})=\ker(\Res)\oplus\ker(\Res)^{\perp}$. It induces a continuous $H$-equivariant bijection
   \[\ker(\Res)^{\perp}\rightarrow W^{s}_{H}(\mathcal{H}_{\frac{1}{2}-\nu}^{\flat}).\]
   The consequence then follows from the  \textit{open mapping theorem}.
 \end{proof}

\section{Effective estimates for ergodic averages}\label{time change220}
In Section \ref{time change194}, we see that it is possible to find a cocompact lattice $\Gamma\subset G$ such that $L^{2}(G/\Gamma)$ contains a complementary series $\mathcal{H}_{\ast}$ of $G=SO(n,1)$  with spectral parameter  $\tilde{\nu}\in(\rho_{n-1},\rho_{n})$ as a direct summand. We write the orthogonal decomposition
 \[\mathcal{H}=\mathcal{H}_{\ast}\oplus\mathcal{H}_{\ast}^{\perp}.\]
       Let $H=SO(2,1)$. When we study the $H$-action on $\mathcal{H}_{\ast}$,   by repeatedly using Theorem \ref{time change116}, there is an $H$-complementary series $\mathcal{H}_{\nu}$ with $\nu=\tilde{\nu}-\rho_{n-1}\in(0,\frac{1}{2})$ such that for any $r\geq0$, we have
 \[ W^{r}_{G}(\mathcal{H}_{\ast})=W_{1}^{r,\nu}\oplus W_{2}^{r,\nu}.\]
 where the restriction map $\Res:W_{1}^{r,\nu} \rightarrow W^{r}_{H}(\mathcal{H}_{\nu})$ is $H$-equivariant   isomorphism.
  Then, for $r\in\mathbf{R}$, we further have the following decomposition
  \begin{equation}\label{time change122}
   W^{r}_{G}(G/\Gamma)=W_{1}^{r,\nu}\oplus W_{2}^{r,\nu} \oplus W^{r}_{G}(\mathcal{H}_{\ast}^{\perp}).
  \end{equation}
  \begin{rem}\label{time change141}
    It needs not be true that $W_{1}^{r,\nu}\subset W_{1}^{t,\nu}$ for $r>t$.
  \end{rem}

    Later, we want to make sure that some specific elements in $W^{r}(G/\Gamma)$ are bounded on $G/\Gamma$.   As $G/\Gamma$ is compact, we only need to verify that they are continuous, which can be done by \textit{Sobolev embedding theorem}\index{Sobolev embedding theorem}.
\begin{lem}[Sobolev embedding theorem]\label{time change126}
 For $r>r_{0}\coloneqq\dim(G/\Gamma)/2$, there is a constant $C=C(G/\Gamma)>0$ such that
  \[|f(x)| < C\|f\|_{W^{r}_{G}}\]
  for any $f\in W^{r}_{G}(G/\Gamma)$ and $x\in G/\Gamma$.
\end{lem}
\begin{proof}
   This is the standard Sobolev embedding theorem, e.g. \cite{aubin1982nonlinear}.
\end{proof}

\subsection{Spectral decomposition of unipotent orbits}
Assume that $G/\Gamma$ is compact. Recall from Section \ref{time change117} that we  let $C^{\infty}(G/\Gamma)$ be the space of infinite differentiable functions on $G/\Gamma$, and   $\mathcal{E}^{\prime}(G/\Gamma)=(C^{\infty}(G/\Gamma))^{\prime}$ be its distributional dual. On the other hand, recall from Section \ref{time change118} that we choose $\mathfrak{a}=\mathbf{R}Y_{n}\subset\mathfrak{g}$. Recall that we fix a nilpotent $U\in  \mathfrak{g}_{-1}^{\flat}$. Thus, $U$ defines a unipotent flow $\phi^{U}_{t}(x)\coloneqq\exp(tU)x$ on $G/\Gamma$ and satisfies
\[[Y_{n},U]=-U.\]
In this section, we want to study the ergodic average
\begin{equation}\label{time change124}
  S_{x,T}(f)\coloneqq\frac{1}{T}\int_{0}^{T}f(\phi^{U}_{t}(x))dt
\end{equation}
 of unipotent flows for functions $f\in W_{1}^{r_{0},\nu}\subset L^{2}(G/\Gamma)$. The proof relies on the characterization of the space of \textit{invariant distributions}\index{invariant distributions} for unipotent flows. Specifically, we make use of the argument in \cite{flaminio2003invariant}. See also \cite{mieczkowski2006cohomological}, \cite{ramirez2013invariant}, \cite{wang2015cohomological} for related discussions.

The space of \textit{$U$-invariant distributions}\index{invariant distributions} for a given $H$-unitary representation $\mathcal{H}$ is then defined by
\[\mathcal{I}_{U}(\mathcal{H})\coloneqq\{\mathcal{D}\in\mathcal{E}^{\prime}(\mathcal{H}):\mathcal{L}_{U}\mathcal{D}=0\}.\]
Similarly, we define The space of \textit{$U$-invariant distributions of order $s$}\index{invariant distributions} to be
\[\mathcal{I}^{r}_{U}(\mathcal{H})\coloneqq\{\mathcal{D}\in W_{H}^{-r}(\mathcal{H}):\mathcal{L}_{U}\mathcal{D}=0\}.\]
Clearly,  the necessary condition  for $g\in W^{r}(\mathcal{H})$ having the form $g=Uf$  for some $f\in W_{H}^{r+1}(\mathcal{H})$ is $g\in \ker\mathcal{I}^{r}_{U}(\mathcal{H})=\{g\in\mathcal{H}:\mathcal{D}(g)=0\ \text{ for any }\mathcal{D}\in\mathcal{I}^{r}_{U}(\mathcal{H})\}$, since
\[\mathcal{D}(g)=\mathcal{D}(Uf)=-\mathcal{L}_{U}\mathcal{D}(f)=0\]
for any $\mathcal{D}\in\mathcal{I}^{r}_{U}(\mathcal{H})$. On the other hand,
Flaminio and Forni \cite{flaminio2003invariant} have characterized the spaces of $U$-invariant distributions for all $SO(2,1)$-irreducible unitary representations, and shown that they are the \textbf{only} obstructions to the existence of smooth solutions of the cohomological equation $Uf=g$. Here we need the results for the complementary series:
\begin{thm}[$SO(2,1)$-complementary series, \cite{flaminio2003invariant}]\label{time change125} For  $\nu\in(0,\frac{1}{2})$, let $(\mathcal{H}_{\nu},\pi_{\nu})$ be a complementary series of $H=SO(2,1)$. Then the space $\mathcal{I}_{U}(\mathcal{H}_{\nu})$ has dimension $2$ and it is generated by two $Y_{n}$-eigenvectors $\mathcal{D}_{\nu}^{\pm}$ of eigenvalues $-(1\pm 2\nu)/2$ and Sobolev order $(1\pm 2\nu)/2$, respectively. In other words
\[\mathcal{L}_{Y_{n}}\mathcal{D}_{\nu}^{\pm}=-\frac{1\pm 2\nu}{2}\mathcal{D}_{\nu}^{\pm},\ \ \ \mathcal{D}_{\nu}^{\pm}\in W_{H}^{-\frac{1\pm 2\nu}{2}}(\mathcal{H}_{\nu}).\]
Besides, let $s>(1+ 2\nu)/2$ and $t<s-1$. Then there is a constant $C(\nu,s,t)>0$ such that, for all $g\in \ker\mathcal{I}^{s}_{U}(\mathcal{H}_{\nu})$, the cohomological equation has a solution $f\in W_{H}^{t}(\mathcal{H}_{\nu})$ which satisfies the Sobolev estimate
\[\|f\|_{W^{t}_{H}(\mathcal{H}_{\nu})}\leq C(\nu,s,t)\|g\|_{W^{s}_{H}(\mathcal{H}_{\nu})}.\]

On the other hand, let $g\in W^{s}(\mathcal{H}_{\nu})$, $s>(1+2\nu)/2$. If the equation $Uf=g$ has a solution $f\in W^{t}(\mathcal{H}_{\nu})$ with $t\geq (2\nu-1)/2$, then $\mathcal{D}_{\nu}^{\pm}(g)=0$.
\end{thm}
    Thus, for $r_{0}\coloneqq\dim(G/\Gamma)/2>1$ and $H$-complementary series $\mathcal{H}_{\nu}$, $\mathcal{I}_{U}^{r_{0}}(\mathcal{H}_{\nu})\subset W_{H}^{-r_{0}}(\mathcal{H}_{\nu})$ is closed. Then the   orthogonal decomposition is of the form
\begin{equation}\label{dynamical systems23}
  W_{H}^{-r_{0}}(\mathcal{H}_{\nu})=\mathcal{I}_{U}^{r_{0}}(\mathcal{H}_{\nu})\oplus\mathcal{I}_{U}^{r_{0}}(\mathcal{H}_{\nu})^{\perp}.
\end{equation}
Combining (\ref{dynamical systems23}) with (\ref{time change122}), we get
\begin{equation}\label{time change123}
  W^{-r_{0}}_{G}(G/\Gamma)=(\mathcal{I}_{1}^{r_{0}}\oplus\mathcal{I}_{2}^{r_{0}})\oplus W_{2}^{-r_{0},\nu}  \oplus W^{-r_{0}}_{G}(\mathcal{H}_{\ast}^{\perp}).
\end{equation}
where $\mathcal{I}_{1}^{r_{0}}\coloneqq\Res^{-1}\mathcal{I}_{U}^{r_{0}}(\mathcal{H}_{\nu})$, $\mathcal{I}_{2}^{r_{0}}\coloneqq\Res^{-1}\mathcal{I}_{U}^{r_{0}}(\mathcal{H}_{\nu})^{\perp}$.
\begin{rem}\label{time change139}
   The spaces $\mathcal{I}_{1}^{r_{0}}$,  $ W_{2}^{-r_{0},\nu}$, $W^{-r_{0}}_{G}(\mathcal{H}_{\ast}^{\perp})$ are  $\phi^{Y_{n}}_{t}$-invariant. However,   $\mathcal{I}_{2}^{r_{0}}$ is not $\phi^{Y_{n}}_{t}$-invariant.
\end{rem}

According to   the previous results,  $\phi^{X}_{t}$ has a spectral decomposition on the space $\mathcal{I}_{U}^{r_{0}}(\mathcal{H}_{\nu})$.  More precisely,  for all $t\in\mathbf{R}$, we have
\[\phi^{Y_{n}}_{t}(\mathcal{D}_{\nu}^{\pm})=e^{-\frac{1\pm 2\nu}{2}t}\mathcal{D}_{\nu}^{\pm}.\]
 Now we consider the ergodic average $S_{x,T}$ defined in (\ref{time change124}) as a distribution in $W^{-r_{0}}_{G}(G/\Gamma)$. By (\ref{time change123}), we can write
\begin{equation}\label{dynamical systems24}
  S_{x,T}= c_{+}(x,T)\mathcal{D}_{\nu}^{+}+c_{-}(x,T)\mathcal{D}_{\nu}^{-}+\mathcal{R}(x,T)+\mathcal{C}(x,T)
\end{equation}
where $\mathcal{D}_{\nu}^{\pm}\in \mathcal{I}_{1}^{r_{0}}$, $\mathcal{R}(x,T)\in \mathcal{I}_{2}^{r_{0}}$, $\mathcal{C}(x,T)\in  W_{2}^{-r_{0},\nu}\oplus W^{-r_{0}}_{G}(\mathcal{H}_{\ast}^{\perp})$.
\begin{rem}\label{time change143}
   The distributions $\mathcal{D}_{\nu}^{\pm}\in \mathcal{I}_{1}^{r_{0}}$ should more appropriately be written as $\mathcal{D}_{\nu}^{\pm}\circ\Res$, as $\mathcal{D}_{\nu}^{\pm}\in W_{H}^{-\frac{1\pm 2\nu}{2}}(\mathcal{H}_{\nu})$ has already been given in Theorem \ref{time change125}. We abuse notation if it makes no confusion.
\end{rem}
Thus, we can analyze the ergodic average $S_{x,T}$ via the distributions. The method has already been used to study the ergodic averages for horocycle flows in \cite{flaminio2003invariant}. We adopt the same strategy here and provide  proofs for the sake of completeness.
\begin{rem}
   It is possible to obtain a more explicit decomposition than (\ref{dynamical systems24}). For instance, \cite{mukunda1968unitary} provides the full decomposition of the complementary series $\mathcal{H}_{\nu}$ of $G=SO(3,1)$ under $H=SO(2,1)$. If $0<\nu\leq  \frac{1}{2}$, it is a sum of two direct integrals of spherical principal series, and if $\frac{1}{2}<\nu<1$, it contains one extra discrete component, the complementary series $\mathcal{H}_{\nu-\frac{1}{2}}$, as we have shown. However,  the relations of Sobolev structures on the principal series are not quite clear. Thus, it seems that we cannot apply the Flaminio-Forni argument to get further information.
\end{rem}

In the following, we want to apply Theorem \ref{time change125} with different Sobolev orders, and hence the space $W_{1}^{r,\nu}$ is no longer convenient, as indicated in Remark \ref{time change141}. Thus, for $s\geq0$, we introduce
\[W^{s,r_{0},\nu}_{1}\coloneqq\Res^{-1}(W_{H}^{r_{0}+s}(\mathcal{H}_{\nu})).\]
In particular, $W^{0,r_{0},\nu}_{1}=W^{r_{0},\nu}_{1}$.
Then we have $W^{s,r_{0},\nu}_{1}\subset W^{t,r_{0},\nu}_{1}$ whenever $s>t$. As in (\ref{time change123}), we have
\begin{equation}\label{time change142}
 W^{-s,-r_{0},\nu}_{1}=\mathcal{I}_{1}^{r_{0}+s}\oplus \mathcal{I}_{2}^{r_{0}+s}
\end{equation}
where $\mathcal{I}_{1}^{r_{0}+s}\coloneqq\Res^{-1}\mathcal{I}_{U}^{r_{0}+s}(\mathcal{H}_{\nu})$,  $\mathcal{I}_{2}^{r_{0}+s}\coloneqq\Res^{-1}\mathcal{I}_{U}^{r_{0}+s}(\mathcal{H}_{\nu})^{\perp}$.

Next, we collect some basic results with respect to the decomposition  (\ref{dynamical systems24}).  First, we observe that   the norms of $S_{x,T}$ in $W_{1}^{-s,-r_{0},\nu}\cong W_{H}^{-r_{0}-s}(\mathcal{H}_{\nu})$ are equivalent to their coefficients in the decomposition.
\begin{lem}[$W_{1}^{-s,-r_{0},\nu}$-norm estimates]\label{time change127}  For $s\geq 0$, we have
  \[ |c_{+}(x,T)|^{2} +|c_{-}(x,T)|^{2} +\|\mathcal{R}(x,T)\|^{2}_{W_{1}^{-s,-r_{0},\nu}} \asymp_{s} \|S_{x,T}\|^{2}_{W_{1}^{-s,-r_{0},\nu}}.\]
\end{lem}
\begin{proof}
   It follows directly from the orthogonal decomposition, and the fact that $\{\mathcal{D}_{\nu}^{\pm}\}\subset \mathcal{I}_{U}^{s}(\mathcal{H}_{\nu})$ is a basis by Theorem \ref{time change125}.
\end{proof}
Combining Lemma \ref{time change127} with Sobolev embedding theorem (Lemma \ref{time change126}), we obtain a uniform upper bound for the coefficients:
\begin{cor}\label{time change148}
  For   $s\geq0$, there exists a constant $C=C(s)>0$ such that
 \[ |c_{+}(x,T)|^{2} +|c_{-}(x,T)|^{2} +\|\mathcal{R}(x,T)\|^{2}_{W_{1}^{-s,-r_{0},\nu}} \leq C\]
 for all $x,T$.
\end{cor}
\begin{proof}
   Note that $|S_{x,T}(f)|\leq \max_{x\in G/\Gamma}|f(x)|$ for any $f\in W_{1}^{s,r_{0},\nu}\subset W_{1}^{r_{0},\nu}$.
\end{proof}

\subsection{Estimates for coefficients via Gottschalk-Hedlund}
Based on the study of the cohomological equation $Uf=g$, we can obtain a better bound for  $\mathcal{R}(x,T)$. Recall that the restriction map $\Res:W_{1}^{s,r_{0},\nu}\rightarrow W_{H}^{r_{0}+s}(\mathcal{H}_{\nu})$ is $H$-equivariant. Thus, the cohomological equation $Uf=g$  on $W_{H}^{r_{0}+s}(\mathcal{H}_{\nu})$ is equivalent to $\Res(Uf)=U\Res(f)=\Res(g)$ on $W_{1}^{s,r_{0},\nu}$.
 \begin{lem}[Pointwise bound for $\mathcal{R}(x,T)$]\label{time change131}
  For   $s>1$, there exists a constant $C=C(\nu,s)>0$ such that
  \[\|\mathcal{R}(x,T)\|_{W_{1}^{-s,-r_{0},\nu}}\leq CT^{-1}.\]
\end{lem}
\begin{proof}
   The orthogonal decomposition (\ref{time change142}) implies
   \begin{equation}\label{time change133}
    W^{s,r_{0},\nu}_{1}=\ker(\mathcal{I}_{1}^{r_{0}+s})\oplus\ker(\mathcal{I}_{2}^{r_{0}+s}).
   \end{equation}
   Then, for any $g\in  W^{s,r_{0},\nu}_{1}$, there is a unique orthogonal decomposition $g=g_{1}+g_{2}$ where $g_{1}\in \ker(\mathcal{I}_{1}^{r_{0}+s})$ and $g_{2}\in  \ker(\mathcal{I}_{2}^{r_{0}+s})$. Since $\mathcal{R}(x,T)\in \mathcal{I}_{2}^{r_{0}+s}$, we have
   \begin{equation}\label{time change129}
     \mathcal{R}(x,T)(g)=\mathcal{R}(x,T)(g_{1}+g_{2})=\mathcal{R}(x,T)(g_{1})=S_{x,T}(g_{1}).
   \end{equation}
   Now since $g_{1}\in\ker(\mathcal{I}_{1}^{r_{0}+s})$, by Theorem \ref{time change125}, there exists a function $f_{1}\in  W^{t,r_{0},\nu}_{1}$ with $t\in(0,s-1)$, such that $Uf_{1}=g_{1}$ and
   \[\|f_{1}\|_{W^{t,r_{0},\nu}_{1}}\ll_{\nu,s,t}  \|g_{1}\|_{W^{s,r_{0},\nu}_{1}}.\]
   By the Sobolev embedding theorem (Lemma \ref{time change126}), we conclude that
   \[\max_{x\in G/\Gamma}|f_{1}(x)|\ll \|f_{1}\|_{W^{t,r_{0},\nu}_{1}}\ll_{\nu,s,t}  \|g_{1}\|_{W^{s,r_{0},\nu}_{1}}.\]
   It follows that
   \begin{equation}\label{time change128}
     |S_{x,T}(g_{1})|=\frac{1}{T}|f_{1}\circ\phi^{U}_{T}(x)-f_{1}(x)|\ll_{\nu,s,t}\frac{1}{T}\|g_{1}\|_{W^{s,r_{0},\nu}_{1}}.
   \end{equation}
   Therefore, using (\ref{time change129}), (\ref{time change128}), we make an appropriate choice of $t(s)\in(0,s-1)$ and then there exists $C=C(\nu,s)>0$ such that
   \[|\mathcal{R}(x,T)(g)|\ll_{\nu,s,t(s)}\frac{1}{T}\|g_{1}\|_{W^{s,r_{0},\nu}_{1}}\leq \frac{C}{T}\|g \|_{W^{s,r_{0},\nu}_{1}}.\]
   The consequence follows.
\end{proof}
We also need a $L^{2}$-bound for $\mathcal{R}(x,T)$ in order to get the lower bound for ergodic averages. The proof for the $L^{2}$-bound (Lemma \ref{time change130}) is completely similar to the pointwise bound (Lemma \ref{time change131}).
\begin{lem}[$L^{2}$-bound  for $\mathcal{R}(x,T)$]\label{time change130}  For   $s>1$, there exists a constant $C=C(\nu,s)>0$ such that for $g\in W^{s,r_{0},\nu}_{1}$, we have
  \[\|\mathcal{R}(\cdot,T)(g)\|_{L^{2}(G/\Gamma)}\leq\frac{C}{T}\|g\|_{W^{s,r_{0},\nu}_{1}}.\]
\end{lem}
\begin{proof}
   As in the proof of Lemma \ref{time change131}, we write
    \[g=g_{1}+g_{2}\in \ker(\mathcal{I}_{1}^{r_{0}+s})\oplus\ker(\mathcal{I}_{2}^{r_{0}+s}).\]
    Then we have  $Uf_{1}=g_{1}$ and
    \begin{align}
  \|\mathcal{R}(\cdot,T)(g)\|_{L^{2}}=&   \|S_{\cdot,T}(g_{1})\|_{L^{2}}\leq \frac{2}{T} \|f_{1}\|_{L^{2}}  \;\nonumber\\
 \leq&   \frac{2}{T}\max_{x\in G/\Gamma}  |f_{1}(x) |  \ll_{\nu,s,t(s)}  \|g_{1}\|_{W^{s,r_{0},\nu}_{1}} \leq  \frac{C}{T}\|g\|_{W^{s,r_{0},\nu}_{1}}.  \;  \nonumber
\end{align}
    This proves   Lemma  \ref{time change130}.
\end{proof}
The following Gottschalk-Hedlund theorem  is a useful criterion for  $L^{2}$-solutions for the cohomological equation  for ergodic measurable flows $\phi_{t}$.
\begin{lem}[Gottschalk-Hedlund]\label{time change134}
If an $L^{2}$-function $f$ is a solution of the equation
\begin{equation}\label{time change132}
  \frac{df\circ\phi_{t}}{dt}\bigg|_{t=0}=g
\end{equation}
then the one-parameter family of functions $G_{T}$ defined by
\[G_{T}(x)\coloneqq\int_{0}^{T}g(\phi_{t}(x))dt\]
is equibounded in $L^{2}$ by $2\|f\|$. Conversely, if the family $G_{T}$ is equibounded, then the cohomological equation has an $L^{2}$-solution.
\end{lem}
\begin{proof}
   The $L^{2}$-norm of $G_{T}$ is clearly bounded by $2\|f\|$ if $f$ is a solution of (\ref{time change132}). On the other hand, if the family of functions $\{G_{T}\}_{T\geq0}$ is equibounded in $L^{2}$, then the family of functions  $\{f_{T}\}_{T\geq0}$ defined by
   \[f_{T}(x)\coloneqq-\frac{1}{T}\int^{T}_{0}\int^{t}_{0}G(\phi_{s}(x))dsdt\]
   is equibounded in $L^{2}$. Then by ergodic theorem, $G$ has zero ergodic average, and any weak limit $f\in L^{2}$ of $\{f_{T}\}_{T\geq0}$ is a $L^{2}$-solution of (\ref{time change132}).
\end{proof}
 The following results provide an important information about $L^{2}$-bounds for the ergodic averages.

 In the following we shall use \textit{Hahn-Banach theorem} to construct functions dual to $\mathcal{D}^{\pm}_{\nu}$ in order to estimate the coefficients. More precisely, there is a $1$-dimensional space $(\mathcal{D}^{+}_{\nu})^{\prime} \subset W^{s,r_{0},\nu}_{1}$ such that $g\in (\mathcal{D}^{+}_{\nu})^{\prime}$ satisfies
  \[\mathcal{D}^{+}_{\nu}(g)\neq0,\ \ \ \mathcal{D}^{-}_{\nu}(g)= \mathcal{R}(x,T)(g)=\mathcal{C}(x,T)(g)=0\]
  and $(\mathcal{D}^{-}_{\nu})^{\prime}$ can be similarly defined.
\begin{lem}[$L^{2}$-bound  for $c_{\pm}(x,T)$]\label{time change158}
   For  $s>1$,  $\mathcal{D}_{\nu}^{\pm}\in \mathcal{I}_{1}^{r_{0}+s}$, there exists a constant $C(\mathcal{D}_{\nu}^{\pm})>0$ such that
   \begin{equation}\label{time change136}
     \|c_{\pm}(\cdot,T)\|_{L^{2}}\leq  C(\mathcal{D}_{\nu}^{\pm}).
   \end{equation}
   On the other hand, $c_{\pm}(x,T)$ satisfies the $L^{2}$ lower bound
   \begin{equation}\label{time change137}
     \sup_{T\in\mathbf{R}^{+}}T\|c_{\pm}(\cdot,T)\|_{L^{2}}=\infty.
   \end{equation}
   Moreover, if $Z\in C_{\mathfrak{g}}(U)$, $\lambda\in\mathbf{R}$, and $g\in (D^{\pm}_{\nu})^{\prime}$ such that the equation
   \begin{equation}\label{time change135}
     g-\phi^{Z}_{\lambda}g=Uf
   \end{equation}
   has no $L^{2}$-solutions $f$,
   then we have
    \begin{equation}\label{time change138}
    \sup_{T\in\mathbf{R}^{+}}T\|c_{\pm}(\cdot,T)-\phi^{Z}_{\lambda}c_{\pm}(\cdot,T)\|_{L^{2}}=\infty.
   \end{equation}
\end{lem}
\begin{proof}
   We only consider the coefficient $c_{+}(x,T)$. Then there is a unique function $g\in (D^{+}_{\nu})^{\prime}\subset  W^{s,r_{0},\nu}_{1}$ (cf.  (\ref{time change133})) such that
   \[D^{+}_{\nu}(g)=1,\ \ \ D^{-}_{\nu}(g)= \mathcal{R}(x,T)(g)=\mathcal{C}(x,T)(g)=0\]
   for all $x,T$.
   It follows that
   \[\|c_{+}(\cdot,T)\|_{L^{2}}=\|S_{\cdot,T}(g)\|_{L^{2}}\leq \|g\|_{L^{2}}.\]
   This proves (\ref{time change136}).
   On the other hand, since $D^{+}_{\nu}(\Res (g))\neq0$ (see Remark \ref{time change143}), by Theorem \ref{time change125}, the equation $U\Res(f)=\Res (g)$ has no solutions   $\Res(f)\in\mathcal{H}_{\nu}$. Thus, we conclude that $Uf=g$ does not have $L^{2}$-solutions $f$. Then by Gottschalk-Hedlund theorem (Lemma \ref{time change134}), the family of functions
   \[Tc_{+}(x,T)=TS_{x,T}(g)=\int_{0}^{T}g(\phi^{U}_{t}(x))dt\]
   is not equibounded in $L^{2}(G/\Gamma)$. This proves (\ref{time change137}). Similarly, if (\ref{time change135}) has no $L^{2}$-solutions $f$, then the family of functions
    \begin{align}
 TS_{x,T}( g-\phi^{Z}_{\lambda}g)= & \int_{0}^{T}g(\phi^{U}_{t}(x))-g(\phi^{U}_{t}\phi^{Z}_{\lambda}(x))dt  \;\nonumber\\
 =&   T(c_{+}(x,T)-c_{+}(\phi^{Z}_{\lambda}(x),T))\mathcal{D}(g)  \;  \nonumber
\end{align}
 is not  in $L^{2}$-equibounded by Gottschalk-Hedlund theorem again.
    This proves (\ref{time change138}).
\end{proof}
\subsection{Estimates for coefficients via geodesic renormalization}
Recall that by the choice of $U$, $Y_{n}$, we have the renormalization
\[\phi_{t}^{Y_{n}}\circ\phi^{U}_{s}=\phi^{U}_{se^{-t}}\circ\phi_{t}^{Y_{n}}.\]
It follows that
\[\phi^{Y_{n}}_{t}(S_{x,T})=S_{\phi^{Y_{n}}_{-t}(x),e^{t}T}.\]
We shall use (\ref{dynamical systems26}) to study the asymptotic behavior of $S_{x,T}$.
Recall that in the decomposition (\ref{time change123}),
   $\mathcal{I}_{U}^{s}(\mathcal{H}_{\nu})^{\perp}$ is not $\phi^{Y_{n}}_{t}$-invariant. We need to show that the remainder term $\mathcal{R}(x,T)\in \mathcal{I}_{U}^{s}(\mathcal{H}_{\nu})^{\perp}$  is still negligible under $Y_{n}$-action.

  It is convenient to discretize the geodesic flow. More precisely, fix $\sigma\in[1,2]$, $x\in G/\Gamma$, $T\geq0$. For any $l\in\mathbf{N}$,  we consider
\begin{equation}\label{dynamical systems26}
\phi^{Y_{n}}_{l\sigma}(S_{x,T})=S_{\phi^{Y_{n}}_{l\sigma}(x),e^{l\sigma}T}.
\end{equation}
  Similar to (\ref{dynamical systems24}), the ergodic average $ \phi^{Y_{n}}_{l\sigma}(S_{x,T})$ has the decomposition
\begin{equation}\label{time change140}
   \phi^{Y_{n}}_{l\sigma}(S_{x,T})= c_{+}^{T}(x,l)\mathcal{D}_{\nu}^{+}+c_{-}^{T}(x,l)\mathcal{D}_{\nu}^{-}+\mathcal{R}^{T}(x,l)+\mathcal{C}^{T}(x,l).
\end{equation}
We prove pointwise and $L^{2}$-bounds for the functions $c_{\pm}^{T}(x,l)$, $\mathcal{R}^{T}(x,l)$. By the identity (\ref{dynamical systems26}) and the definition (\ref{dynamical systems24}), we have
\begin{equation}\label{time change145}
  c_{\pm}^{T}(x,l)=c_{\pm}( \phi^{Y_{n}}_{l\sigma}(x),e^{l\sigma}T),\ \ \ \mathcal{R}^{T}(x,l)=\mathcal{R}( \phi^{Y_{n}}_{l\sigma}(x),e^{l\sigma}T).
\end{equation}
Note that $\mathcal{R}$-component is not $\phi^{Y_{n}}_{t}$-invariant, but we still have $\phi^{Y_{n}}_{t}\mathcal{R}^{T}(x,l)\in W^{-s,-r_{0},\nu}_{1}$. Now we estimate the remainder term $\mathcal{R}^{T}(x,l)$ after pushforward by one geodesic step $\phi^{Y_{n}}_{\sigma}$. Let $r_{\pm}^{T}(x,l)\coloneqq c_{\pm}( \phi^{Y_{n}}_{\sigma}\mathcal{R}^{T}(x,l)) $ be its $\mathcal{D}_{\nu}^{\pm}$-component.  Then, similar to (\ref{dynamical systems24}), we have
\begin{align}
 \phi^{Y_{n}}_{\sigma}\mathcal{R}^{T}(x,l)= & c_{+}(\phi^{Y_{n}}_{\sigma}\mathcal{R}^{T}(x,l))\mathcal{D}_{\nu}^{+}+c_{-}(\phi^{Y_{n}}_{\sigma}\mathcal{R}^{T}(x,l))\mathcal{D}_{\nu}^{-}+\mathcal{R}(\phi^{Y_{n}}_{\sigma}\mathcal{R}^{T}(x,l)) \;\nonumber\\
  =&   r_{+}^{T}(x,l)\mathcal{D}_{\nu}^{+}+ r_{-}^{T}(x,l)\mathcal{D}_{\nu}^{-}+\mathcal{R}^{T}(x,l+1). \;  \label{time change144}
\end{align}
Moreover, we get
\begin{equation}\label{time change146}
  c_{\pm}^{T}(x,l+1)= c_{\pm}( \phi^{Y_{n}}_{\sigma}(c_{\pm}^{T}(x,l)+\mathcal{R}^{T}(x,l)))
  =  c_{\pm}^{T}(x,l)e^{\frac{1\pm2\nu}{2}\sigma} +r_{\pm}^{T}(x,l).
\end{equation}
We want to have an effective estimate for the coefficients $c_{\pm}^{T}(x,l)$ of $\mathcal{D}$-components. To solve the recurrence relation (\ref{time change146}), we need the following elementary result.
\begin{lem}\label{time change147}
  Let $A:V\rightarrow V$ be a linear map. Let $\{R_{l}\}\subset V$. The solution $x_{l}$ of the following difference equation
  \[x_{l+1}=A(x_{l})+R_{l}\]
  has the form
  \[x_{l}=A^{l}(x_{0})+\sum_{j=0}^{l-1}A^{l-j-1}R_{j}.\]
\end{lem}
Thus, it remains to estimate the remainder terms $r_{\pm}^{T}(x,l)$.
\begin{lem}[Pointwise bound for $r_{\pm}^{T}(x,l)$]\label{time change149} For fixed $\sigma\in[1,2]$, $s>1$,
  there exists a constant $C =C (\nu,s)>0$ such that
  \[|r_{+}^{T}(x,l)|^{2}+|r_{-}^{T}(x,l)|^{2}\leq C(e^{l\sigma}T)^{-2}\]
  for all $x\in G/\Gamma,T\in\mathbf{R}^{+},l\in\mathbf{N}$.
\end{lem}
\begin{proof}
Let $C(s)\coloneqq\max_{\sigma\in[1,2]}\|\phi^{Y_{n}}_{\sigma}\|$, where $\|\cdot\|$ denotes the operator norm. Similar to Lemma \ref{time change127}, using (\ref{time change144}), we have the estimate:
\[|r_{\pm}^{T}(x,l)|^{2}= |c_{\pm}( \phi^{Y_{n}}_{\sigma}\mathcal{R}^{T}(x,l))|^{2}\ll_{s} \| \phi^{Y_{n}}_{\sigma}\mathcal{R}^{T}(x,l)\|^{2}_{W_{1}^{-s,-r_{0},\nu}}\ll_{s} C(s)\| \mathcal{R}^{T}(x,l)\|^{2}_{W_{1}^{-s,-r_{0},\nu}}.\]
 Then by (\ref{time change145}) and Lemma \ref{time change131}, we have
 \[\| \mathcal{R}^{T}(x,l)\|^{2}_{W_{1}^{-s,-r_{0},\nu}} =\|\mathcal{R}( \phi^{Y_{n}}_{l\sigma}(x),e^{l\sigma}T)\|^{2}_{W_{1}^{-s,-r_{0},\nu}}\ll_{\nu,s} (e^{l\sigma}T)^{-2}.\]
The consequence follows.
\end{proof}
As in Lemma \ref{time change158}, we can also estimate the $L^{2}$-norm of $r_{\pm}^{T}(x,l)$.
\begin{lem}[$L^{2}$-bound  for $r_{\pm}^{T}(x,l)$]\label{time change155} For fixed $\sigma\in[1,2]$, $s>1$,
   there exists a constant $C=C(\nu,s)>0$ such that
   \[\|r_{\pm}^{T}(\cdot,l)\|_{L^{2}}\leq C(e^{l\sigma}T)^{-1}\]
    for any $T\in\mathbf{R}^{+},l\in\mathbf{N}$.
\end{lem}
\begin{proof}
 Again, we choose the function $g\in (\mathcal{D}^{+}_{\nu})^{\prime}\subset  W^{s,r_{0},\nu}_{1}$  such that
   \[\mathcal{D}^{+}_{\nu}(g)=1,\ \ \ \mathcal{D}^{-}_{\nu}(g)= \mathcal{R}^{T}(x,l+1)(g)=0.\]
   Recall that by the definition
\[ r_{+}^{T}(x,l) =\mathcal{R}^{T}(x,l)(\phi^{Y_{n}}_{-\sigma}g)=\mathcal{R}( \phi^{Y_{n}}_{l\sigma}(x),e^{l\sigma}T)(\phi^{Y_{n}}_{-\sigma}g).\]
Then by Lemma  \ref{time change130}, we have
  \[\|r_{+}^{T}(\cdot,l)\|_{L^{2}}=\|\mathcal{R}^{T}(\cdot,l)(\phi^{Y_{n}}_{-\sigma}g)\|_{L^{2}} \ll_{\nu,s}\|g\|_{W^{s,r_{0},\nu}_{1}}(e^{l\sigma}T)^{-1}.\]
  The argument for $r_{-}^{T}(x,l)$ is similar.
\end{proof}

Now we are in the position to estimate the coefficients  $c_{\pm}^{T}(x,l)$ of $\mathcal{D}$-components.
\begin{lem}[Pointwise bound for $c_{\pm}^{T}(x,l)$]\label{time change153}
  Let  $\sigma\in[1,2]$, $s>1$, $T\in\mathbf{R}^{+}$. Then there exists a constant $C=C(\nu,s,T)>0$ such that
  \[|c_{\pm}^{T}(x,l)|\leq C e^{-\frac{1\pm 2\nu}{2}l\sigma}   \]
  for all $x\in G/\Gamma,l\in\mathbf{N}$.
\end{lem}
\begin{proof}
   Choose $V=\mathbf{C}$, $A=e^{-\frac{1\pm 2\nu}{2}\sigma}$, $x_{l}=|c_{\pm}^{T}(x,l)|$ and $R_{l}=|r_{\pm}^{T}(x,l)|$. Then by (\ref{time change146}) and Lemma \ref{time change147}, we obtain
   \begin{equation}\label{time change152}
    |c_{\pm}^{T}(x,l)|\leq|c_{\pm}^{T}(x,0)|e^{-\frac{1\pm 2\nu}{2}l\sigma}+\sum_{j=0}^{l-1}|r_{\pm}^{T}(x,j)|e^{-\frac{1\pm 2\nu}{2}(l-j-1)\sigma}.
   \end{equation}
   By Corollary \ref{time change148}, we have
   \begin{equation}\label{time change150}
    |c_{\pm}^{T}(x,0)|^{2}=|c_{\pm}(x,T)|^{2}\leq C(s).
   \end{equation}
   On the other hand, by Lemma \ref{time change149}, we have
    \[|r_{\pm}^{T}(x,j)|^{2}\leq C (\nu,s)(e^{j\sigma}T)^{-2}.\]
    It follows that
     \begin{align}
     \sum_{j=0}^{l-1}|r_{\pm}^{T}(x,j)|& e^{-\frac{1\pm 2\nu}{2}(l-j-1)\sigma}  \leq   C (\nu,s)\sum_{j=0}^{l-1}(e^{j\sigma}T)^{-1}e^{-\frac{1\pm 2\nu}{2}(l-j-1)\sigma}\;\nonumber\\
=&  C (\nu,s)e^{\frac{1\pm 2\nu}{2}\sigma}T^{-1}e^{-\frac{1\pm 2\nu}{2}l\sigma}\sum_{j=0}^{l-1} e^{ \frac{-1\pm 2\nu}{2}j\sigma} \leq C (\nu,s)T^{-1}e^{-\frac{1\pm 2\nu}{2}l\sigma}. \;  \label{time change151}
\end{align}
Recall that $\nu\in(0,1/2)$. Then by (\ref{time change152}), (\ref{time change150}), (\ref{time change151}), we conclude
\[ |c_{\pm}^{T}(x,l)|\leq C(\nu,s) T^{-1}e^{-\frac{1\pm 2\nu}{2}l\sigma}\]
The proves Lemma \ref{time change153}.
\end{proof}
\begin{lem}[$L^{2}$-bound for $c_{\pm}^{T}(x,l)$] Let $\sigma\in[1,2]$, $s>1$, $T\in\mathbf{R}^{+}$. Then there exists a constant $C=C(\nu,s,T)>0$ such that
\begin{equation}\label{time change160}
  \|c_{\pm}^{T}(\cdot,l)\|_{L^{2}} \leq C e^{-\frac{1\pm 2\nu}{2}l\sigma}
  \end{equation}
  for all $x\in G/\Gamma,l\in\mathbf{N}$. On the other hand, there exist $C_{0}=C_{0}(\nu,s)>0$, $T_{0}=T_{0}(\nu,s)>0$ such that
  \begin{equation}\label{time change159}
   \|c_{\pm}^{T_{0}}(\cdot,l)\|_{L^{2}} \geq  C_{0} e^{-\frac{1\pm 2\nu}{2}l\sigma}
  \end{equation}
  for all   $l\in\mathbf{N}$. Further, if    $Z\in C_{\mathfrak{g}}(U)$, $\lambda\in\mathbf{R}$, and $g\in (D^{\pm}_{\nu})^{\prime}$ such that the equation
   \begin{equation}\label{time change161}
     g-\phi^{Z}_{\lambda}g=Uf
   \end{equation}
   has no $L^{2}$-solutions $f$,
   then  there exist $C_{1}=C_{1}(\nu,s)>0$, $T_{1}=T_{1}(\nu,s)>0$ such that
    \begin{equation}\label{time change162}
    \|c_{\pm}^{T_{1}}(\cdot,l)-\phi^{Z}_{\lambda}c_{\pm}^{T_{1}}(\cdot,l)\|_{L^{2}}\geq  C_{1} e^{-\frac{1\pm 2\nu}{2}l\sigma}
   \end{equation}
   for all $l\in\mathbf{N}$.
\end{lem}
\begin{proof}
   By (\ref{time change146}) and Lemma \ref{time change147}, we obtain
   \begin{equation}\label{time change154}
    \|c_{\pm}^{T}(\cdot,l)\|_{L^{2}}\leq  \|c_{\pm}^{T}(\cdot,0)\|_{L^{2}}e^{-\frac{1\pm 2\nu}{2}l\sigma}+\sum_{j=0}^{l-1} \|r_{\pm}^{T}(\cdot,j)\|_{L^{2}}e^{-\frac{1\pm 2\nu}{2}(l-j-1)\sigma}.
   \end{equation}
   Similar to the pointwise upper bound (Lemma \ref{time change153}), we can apply Lemma  \ref{time change155} (cf. (\ref{time change151})), and obtain
   \begin{equation}\label{time change157}
     \sum_{j=0}^{l-1} \|r_{\pm}^{T}(\cdot,j)\|_{L^{2}}e^{-\frac{1\pm 2\nu}{2}(l-j-1)\sigma}\leq C (\nu,s)T^{-1}e^{-\frac{1\pm 2\nu}{2}l\sigma}.
   \end{equation}
   Combining with Lemma \ref{time change158}, we obtain the $L^{2}$-upper bound (\ref{time change160}).

   On the other hand, using again  (\ref{time change146}) and Lemma \ref{time change147}, and then (\ref{time change157}), we obtain a lower bound
      \begin{align}
    \|c_{\pm}^{T}(\cdot,l)\|_{L^{2}}\geq & \|c_{\pm}^{T}(\cdot,0)\|_{L^{2}}e^{-\frac{1\pm 2\nu}{2}l\sigma}-\sum_{j=0}^{l-1} \|r_{\pm}^{T}(\cdot,j)\|_{L^{2}}e^{-\frac{1\pm 2\nu}{2}(l-j-1)\sigma}\;\nonumber\\
\geq &  (\|c_{\pm}^{T}(\cdot,0)\|_{L^{2}}- C (\nu,s)T^{-1})e^{-\frac{1\pm 2\nu}{2}l\sigma}. \;  \label{time change156}
\end{align}
  By  (\ref{time change137}), there exists $T_{0}=T_{0}(\nu,s)>0$ such that
  \[T_{0}\|c_{\pm}^{T_{0}}(\cdot,0)\|_{L^{2}}=T_{0}\|c_{\pm}(\cdot,T_{0})\|_{L^{2}}>2 C (\nu,s).\]
  It follows that
  \[\|c_{\pm}^{T_{0}}(\cdot,l)\|_{L^{2}}\geq C (\nu,s)T_{0}^{-1}e^{-\frac{1\pm 2\nu}{2}l\sigma}.\]
  This proves (\ref{time change159}).

  Finally, using again  (\ref{time change146}) and Lemma \ref{time change147}, and then (\ref{time change157}), we obtain a lower bound
      \begin{align}
   &  \|c_{\pm}^{T}(\cdot,l)-\phi^{Z}_{\lambda}c_{\pm}^{T}(\cdot,l)\|_{L^{2}}\;\nonumber\\
    \geq & \|c_{\pm}^{T}(\cdot,0)-\phi^{Z}_{\lambda}c_{\pm}^{T}(\cdot,0)\|_{L^{2}}e^{-\frac{1\pm 2\nu}{2}l\sigma}-\sum_{j=0}^{l-1}( \|r_{\pm}^{T}(\cdot,j)\|_{L^{2}}+\|\phi^{Z}_{\lambda}r_{\pm}^{T}(\cdot,j)\|_{L^{2}})e^{-\frac{1\pm 2\nu}{2}(l-j-1)\sigma}\;\nonumber\\
     = & \|c_{\pm}^{T}(\cdot,0)-\phi^{Z}_{\lambda}c_{\pm}^{T}(\cdot,0)\|_{L^{2}}e^{-\frac{1\pm 2\nu}{2}l\sigma}-\sum_{j=0}^{l-1} 2 \|r_{\pm}^{T}(\cdot,j)\|_{L^{2}} e^{-\frac{1\pm 2\nu}{2}(l-j-1)\sigma}\;\nonumber\\
\geq &  (\|c_{\pm}^{T}(\cdot,0)-\phi^{Z}_{\lambda}c_{\pm}^{T}(\cdot,0)\|_{L^{2}}- 2C (\nu,s)T^{-1})e^{-\frac{1\pm 2\nu}{2}l\sigma}. \;  \label{time change163}
\end{align}
Then, with the assumption (\ref{time change161}),  (\ref{time change138}) shows that there exists $T_{1}=T_{1}(\nu,s)>0$ such that
  \[ T_{1}\|c_{\pm}(\cdot,T_{1})-\phi^{Z}_{\lambda}c_{\pm}(\cdot,T_{1})\|_{L^{2}}>3 C (\nu,s).\]
Then  (\ref{time change163}) shows
\[\|c_{\pm}^{T_{1}}(\cdot,l)-\phi^{Z}_{\lambda}c_{\pm}^{T_{1}}(\cdot,l)\|_{L^{2}}\geq C (\nu,s)T_{1}e^{-\frac{1\pm 2\nu}{2}l\sigma}.\]
It implies (\ref{time change162}).
\end{proof}

Now we recover the continuous time by replacing $e^{l\sigma}$ by $T$. Note that based on our assumption, we have $[1,\infty)\subset
\{e^{l\sigma}:l\in\mathbf{N},\sigma\in[1,2]\}$. Therefore, the above results can be translated to:
\begin{prop}[Estimates for $c_{\pm}(x,T)$]\label{time change164} Let $s>1$. Then there exists $T_{0}=T_{0}(\nu,s)>0$ such that
\[|c_{\pm}(x,T)|\ll_{\nu,s} T^{-\frac{1\pm\nu}{2}},\ \ \ \|c_{\pm}(\cdot,T)\|_{L^{2}} \gg_{\nu,s} T^{-\frac{1\pm 2\nu}{2}}\]
for all $x\in G/\Gamma$, $T\geq T_{0}$. Besides, if    $Z\in C_{\mathfrak{g}}(U)$, $\lambda\in\mathbf{R}$, and $g\in (D^{\pm}_{\nu})^{\prime}$ such that the equation
   \[  g-\phi^{Z}_{\lambda}g=Uf \]
   has no $L^{2}$-solutions $f$,
   then
   \[  \|c_{\pm} (\cdot,T)-\phi^{Z}_{\lambda}c_{\pm} (\cdot,T)\|_{L^{2}}\gg_{\nu,s} T^{-\frac{1\pm 2\nu}{2}} \]
   for all $T\geq T_{0}$.
\end{prop}
The following corollary, given by an elementary integral argument, is an important criterion for the existence of measurable solutions of the cohomological equation $Uf=g$:
\begin{cor}
   There exists $T_{0}=T_{0}(\nu,s)>0$, $\gamma=\gamma(\nu,s)>0$ such that for any $T\geq T_{0}$, there exists a measurable set $A_{T}\subset G/\Gamma$ of measure at least $\gamma$, such that
   \begin{equation}\label{time change165}
     |c_{\pm}(x,T)|\geq_{\nu,s}  T^{-\frac{1\pm 2\nu}{2}}
   \end{equation}
   for all $x\in A_{T}$. Besides, if    $Z\in C_{\mathfrak{g}}(U)$, $\lambda\in\mathbf{R}$, and $g\in (D^{\pm}_{\nu})^{\prime}$ such that the equation
   \begin{equation}\label{time change166}
     g-\phi^{Z}_{\lambda}g=Uf
   \end{equation}
   has no $L^{2}$-solutions $f$,
   then
   \begin{equation}\label{time change167}
    |c_{\pm}(x,T)-c_{\pm}(\phi^{Z}_{\lambda}(x),T)|\geq_{\nu,s}  T^{-\frac{1\pm 2\nu}{2}}
   \end{equation}
   for all $x\in A_{T}$.
\end{cor}
\begin{proof}
   By Proposition \ref{time change164},  we have
   \[ |c_{\pm}(x,T)|\ll_{\nu,s} T^{-\frac{1\pm\nu}{2}}\ll_{\nu,s} \|c_{\pm}(\cdot,T)\|_{L^{2}}\]
   for $T\geq T_{0}$.
   More precisely, there is $C=C(\nu,s)>0$  such that
   \begin{equation}\label{time change168}
   |c_{\pm}(x,T)|\leq C \|c_{\pm}(\cdot,T)\|_{L^{2}}
   \end{equation}
  for $T\geq T_{0}$.  Now let
   \[A_{T}\coloneqq\{x\in G/\Gamma:|c_{\pm}(x,T)|>\frac{1}{2}\|c_{\pm}(\cdot,T)\|_{L^{2}}\}.\]
   Then, we have
   \[ \|c_{\pm}(\cdot,T)\|_{L^{2}} =\int_{A_{T}\cup((G/\Gamma)\setminus A_{T})} |c_{\pm}(x,T)|^{2}d\mu(x)\leq (C\mu(A_{T})+\frac{1}{2}) \|c_{\pm}(\cdot,T)\|_{L^{2}} \]
    for $T\geq T_{0}$.  It follows that for $T\geq T_{0}$
    \[\mu(A_{T})\geq \frac{1}{2C}\eqqcolon\gamma.\]
    This proves (\ref{time change165}).

    Next, assume that (\ref{time change166}) holds. Then using Proposition \ref{time change164} again,  we have
   \[ |c_{\pm}(x,T)-c_{\pm}(\phi^{Z}_{\lambda}(x),T)|\ll_{\nu,s} T^{-\frac{1\pm\nu}{2}}\ll_{\nu,s} \|c_{\pm} (\cdot,T)-\phi^{Z}_{\lambda}c_{\pm} (\cdot,T)\|_{L^{2}}\]
   for $T\geq T_{0}$. A   similar argument as above proves (\ref{time change167}).
\end{proof}

Now for $s>1$, we consider $g^{+}\in  (D^{+}_{\nu})^{\prime} \subset W^{s,r_{0},\nu}_{1}$  satisfying
  \[D^{+}_{\nu}(g^{+})\neq0,\ \ \ D^{-}_{\nu}(g^{+})= \mathcal{R}(x,T)(g^{+})=\mathcal{C}(x,T)(g^{+})=0.\]
  Then by (\ref{dynamical systems24}),the ergodic average $S_{x,T}$ of $g^{+}$ is
  \begin{equation}\label{time change170}
   S_{x,T}(g^{+})=\frac{1}{T}\int_{0}^{T}g^{+}(\phi^{U}_{t}(x))dt= c_{+}(x,T)\mathcal{D}_{\nu}^{+}(g^{+}).
  \end{equation}
 Then there are several interesting consequences related to these functions. The following result implies that the \textit{central limit theorem}\index{central limit theorem} does no hold for unipotent flow on $G/\Gamma$.
 \begin{cor}\label{time change202106.8}
   As $T\rightarrow\infty$, any weak limit of the probability distributions
    \[\frac{\frac{1}{T}\int_{0}^{T}g^{+}(\phi^{U}_{t}(x))dt}{\left\|\frac{1}{T}\int_{0}^{T}g^{+}(\phi^{U}_{t}(\cdot))dt\right\|_{L^{2}}}\]
has a nonzero compact support.
 \end{cor}
 \begin{proof}
    By (\ref{time change168}), the   distributions are uniformly bounded above by $C$ for sufficiently large $T$. On the other hand, (\ref{time change165}) shows that the distributions are bounded  below on a measurable set of positive measure $\gamma$. The consequence follows.
 \end{proof}

  Moreover, the functions $g^{+}$  is not \textit{measurably trivial}\index{measurably trivial}, in the sense that there are no measurable functions $f$ satisfy
  \begin{equation}\label{time change169}
 \int_{0}^{T}g^{+}(\phi^{U}_{t}(x))dt=f(\phi^{U}_{T}(x))-f( x).
  \end{equation}
  Hence, we finally arrive at Theorem \ref{time change191}:
  \begin{cor}\label{time change202106.7}
     The functions $g^{\pm}\in  (D^{\pm}_{\nu})^{\prime} \subset W^{s,r_{0},\nu}_{1}$ are not measurably trivial. Moreover, if there are some    $Z\in C_{\mathfrak{g}}(U)$, $\lambda\in\mathbf{R}$    such that  $\phi^{Z}_{\lambda}g^{\pm}$ are not $L^{2}$-cohomologous to $g^{\pm}$, then $\phi^{Z}_{\lambda}g^{\pm}$ are not  measurably cohomologous to $g^{\pm}$.
  \end{cor}
\begin{proof}
   Assume by contradiction that there is a measurable function $f$ satisfies (\ref{time change169}). Then by \textit{Luzin's theorem}\index{Luzin's theorem}, for given $\gamma>0$, there exists a constant $C=C(\gamma)>0$ such that for any $T>0$, there exists a measurable set $B_{T}=B_{T}(\gamma)\subset G/\Gamma$ of measure  $\mu(B_{T})<\gamma$ such that
   \[ \left|\int_{0}^{T}g^{+}(\phi^{U}_{t}(x))dt\right|\leq C\]
   for all $x\in B_{T}^{c}$.
   On the other hand, (\ref{time change170}) and (\ref{time change165}) imply that there exists  a measurable set $A_{T} \subset G/\Gamma$ of measure  $\mu(A_{T})\geq\gamma$ such that
   \[\left|\int_{0}^{T}g^{+}(\phi^{U}_{t}(x))dt\right|= |Tc_{+}(x,T)|\geq_{\nu,s}  T^{\frac{1- 2\nu}{2}}\]
   for all $x\in A_{T}$. It is a contradiction. Thus, $g^{\pm}$ are not measurably trivial.

   Similarly, if there are some    $Z\in C_{\mathfrak{g}}(U)$, $\lambda\in\mathbf{R}$    such that  $\phi^{Z}_{\lambda}g^{+}$ is not $L^{2}$-cohomologous to $g^{+}$, then $\phi^{Z}_{\lambda}g^{\pm}$ are not  measurably cohomologous to $g^{\pm}$. Then (\ref{time change167}) and (\ref{time change170}) imply that there exists  a measurable set $A_{T} \subset G/\Gamma$ of measure  $\mu(A_{T})\geq\gamma$ such that
     \begin{align}
   & \left| \int_{0}^{T}g^{+}(\phi^{U}_{t}(x))- \phi^{Z}_{\lambda}g^{+}(\phi^{U}_{t}(x))dt\right|\;\nonumber\\
   =&  T|c_{+}(x,T)-c_{+}(\phi^{Z}_{\lambda}(x),T)|\geq_{\nu,s}    T^{\frac{1- 2\nu}{2}} . \;  \nonumber
\end{align}
Again, it contradicts   \textit{Luzin's theorem}\index{Luzin's theorem}.
\end{proof}

\bibliographystyle{alpha}
 \bibliography{text}
\end{document}